\documentclass[11pt]{amsart}
\usepackage{mathabx} 
\usepackage{amsmath,amsfonts,amssymb,amsthm,epsfig,color, mathrsfs, scalerel, stackengine} 
\usepackage{hyperref} 
\usepackage{bbm}
\usepackage{mathrsfs, euscript}
\usepackage{upgreek}
\usepackage{tensor}
\usepackage{enumitem}

\newcommand{\bb}{\mathbb}

\newcommand{\ZZ}{\bb Z}

\newcommand{\RR}{\bb R}
\newcommand{\II}{\bb I}

\newcommand{\NN}{\bb N}

\newcommand{\QQ}{\bb Q}

\newcommand{\mK}{\mathcal K}
\newcommand{\mQ}{\mathcal Q}
\newcommand{\Am}{\mathcal A}
\newcommand{\B}{\mathcal B}
\newcommand{\mR}{\mathcal R}
\newcommand{\mS}{\mathcal S}

\newcommand{\mC}{\mathcal C}

\newcommand{\eJ}{\mathcal{J}}
\newcommand{\mW}{\mathcal{W}}
\newcommand{\Ls}{\mathscr{L}}
\newcommand{\sP}{\mathscr{P}}
\newcommand{\mP}{\mathcal{P}}
\newcommand{\sL}{\ell}

\newcommand{\diam}{\operatorname{diam}}
\newcommand{\inte}{\operatorname{int}}

\newcommand{\F}{\mathcal F}
\newcommand{\eO}{\operatorname{ord}}
\newcommand{\col}{\operatorname{col}}

\newcommand{\Fs}{\mathfrak s}
\newcommand{\Fb}{\mathfrak b}

\newcommand{\hh}{\mathcal H}

\newcommand{\BA}{\operatorname{\bold{Bad}}}

\newcommand{\Dir}{\operatorname{\bold{Dir}}}

\newtheorem{Theorem}{Theorem}
\numberwithin{Theorem}{section}

\newtheorem{coro}[Theorem]{Corollary}

\newtheorem{theo}[Theorem]{Theorem}
\newtheorem{prop}[Theorem]{Proposition}

\newtheorem{lemm}[Theorem]{Lemma}

\newtheorem{defn}[Theorem]{Definition}

\newcommand\reallywidehat[1]{%
\savestack{\tmpbox}{\stretchto{%
  \scaleto{%
    \scalerel*[\widthof{\ensuremath{#1}}]{\kern-.6pt\bigwedge\kern-.6pt}%
    {\rule[-\textheight/2]{1ex}{\textheight}}
  }{\textheight}%
}{0.5ex}}%
\stackon[1pt]{#1}{\tmpbox}%
}

\newtheorem*{lemma*}{Lemma}
\newtheorem*{question*}{Question}

\newtheorem*{theorem*}{Theorem}

\newtheorem*{defn*}{Definition}

\theoremstyle{remark}

\newtheorem{rema}[Theorem]{\sc Remark}
\newtheorem*{rema*}{\sc Remark}

\numberwithin{equation}{section}
\begin{document}
\title{$\delta$-Badly approximable numbers and ubiquitously losing sets}
\author{Jimmy Tseng}

\address{Department of Mathematics and Statistics, University of Exeter, Exeter EX4 4QF, U.K}
\email{j.tseng@exeter.ac.uk}

\begin{abstract}  We consider a natural filtration $\BA(\delta) \subset \BA(\delta')$ for $\delta \geq \delta'>0$ on the set of badly approximable numbers to complement the filtration of the well approximable numbers by the $\tau$-well approximable numbers.  We show that the set $\BA(\delta)$ is a $(1/3, 18 \delta)$-winning set and give a lower bound on its Hausdorff dimension.  We introduce the notion of {\em $(\alpha, \beta)$-ubiquitously losing sets} to the theory of Schmidt games, give an upper bound on the Hausdorff dimension of an $(\alpha, \beta)$-ubiquitously losing set that is strictly less than full Hausdorff dimension, show that $\BA(\delta)$ is a $(1/2, 18/\delta)$-ubiquitously losing set, and give an upper bound on the Hausdorff dimension of $\BA(\delta)$ that is strictly less than one.  Combined with a finite intersection property and a bilipschitz transfer property, we obtain results for finite intersections of translates of $\BA(\delta)$.
\end{abstract}

\maketitle

\tableofcontents

\section{Introduction}

Classically, Diophantine approximation is the study of the approximation properties of the real numbers by the rational numbers.  The set of real numbers is partitioned into the set of badly approximable numbers, which we denote $\BA$, and the set of well approximable numbers.  There is a (non-exhaustive) filtration on the set of well approximable numbers by the $\tau$-well approximable numbers, a filtration indexed by the real numbers $\tau$.  The natural tool to study the well approximable numbers and the $\tau$-well approximable numbers is the mass transference principle.  For details, see, for example, the survey~\cite[Sections~1.1--1.3]{BRV16}.

Conversely, a natural tool to study the set of badly approximable numbers $\BA$ is the theory of Schmidt games.  In this paper, our first goal is to consider a theory analogous to that of the $\tau$-well approximable numbers for the badly approximable numbers, namely the filtration on $\BA$ whose subobjects are the sets $\BA(\delta)$ (Definition~\ref{defn:DeltaBANumber}) and which is indexed by the real numbers $1/\sqrt 5 >\delta>0$.  To study these sets, we introduce the notion of $(\alpha, \beta)$-ubiquitously losing sets in Definition~\ref{defn:UbitLos} to the theory of Schmidt games for real numbers $0 < \alpha <1$ and $0< \beta<1$, further reinforcing the utility of the theory of Schmidt games for studying $\BA$.  The benefit of the notion of  $(\alpha, \beta)$-ubiquitously losing sets is that it enables us to bound the upper Hausdorff dimension by a bound that is strictly less than full Hausdorff dimension (see Theorem~\ref{theo:LosingBADeltaUpperHD} for the precise statement) and to also retain a version of the key property associated with Schmidt games, namely a finite intersection property (see Theorem~\ref{theo:WinningBADeltaLowerHDCtbleIntersect} as an example and Theorem~\ref{thm:CtbleWinPropForWinSets} for the finite intersection property).  

More generally, the introduction of $(\alpha, \beta)$-ubiquitously losing sets (for $\RR^d$ where $d \in \NN$ but even more generally is possible) into the theory of Schmidt games enables us to achieve our second goal for this paper, namely extending the study of sets that satisfy the countable intersection property and have full Hausdorff dimension to the study of sets that satisfy the finite intersection property but need not have full Hausdorff dimension.  This may enable the theory of Schmidt games to become a more versatile tool in fractal geometry itself.  

Finally, the ideas in this paper may have greater applicability in studying the finer structure of other $(\alpha, \beta)$-winning sets and amplifying the utility of other games motivated by Schmidt games (see Remark~\ref{rem:Discuss}).

\subsection{$\delta$-Badly approximable numbers}   In this section, we consider the following definition (see also~\cite[Section~1.3]{Sim18}).

\begin{defn}\label{defn:DeltaBANumber} Let $\delta>0$.  An irrational real number $x$ is a {\bf $\delta$-badly approximable number} if there exists an integer $Q>0$ (depending on $x$) such that \begin{align}\label{eqnBadDelta} \left |x - \frac p q \right| \geq \frac \delta{q^2} \end{align} holds for all integers $q \geq Q$ and all integers $p$ for which $(p,q)=1$.\footnote{The notation $(p,q)=1$ means that $p$ is relatively prime to $q$.  For convenience, we extend the definition of relatively prime to $0$ and a natural number $q$ by asserting that $0$ is relatively prime to $1$ and only to $1$.}  The set of $\delta$-badly approximable numbers is denoted $\BA(\delta)$.  

\end{defn}

\begin{prop}\label{prop:BAFiltration}
We have \[\bigcup_{\delta>0} \BA( \delta) = \BA.\] 
\end{prop}
\begin{proof}  Let $x \in \BA$.  Recall that a definition of badly approximable number is that there exists a constant $c>0$ such that $|x - p/q| > c/q^2$ holds for all $p \in \ZZ$ and $q \in \NN$ for which $(p, q)=1$.  Letting $\delta = c$ and taking $Q =1$, we have that $x \in \BA(\delta)$.

Let $\delta>0$ and let $x \in \BA(\delta)$.  Define the strictly positive real numbers \[\delta_j := \inf_{p \in \ZZ} \left(j^2|x-p/j|\right)\] for $j = 1, \cdots, Q$.  Let $c = \min(\delta_1/2, \cdots, \delta_Q/2, \delta/2)$.  Then $x \in \BA$. \end{proof}

Let $\delta>0$.  The sets $\BA(\delta)$ are natural in two separate ways, namely from the point of view of Schmidt games (see Theorems~\ref{theo:WinningBADelta} and~\ref{theo:LosingBADelta}) and the point of view of Dirichlet's approximation theorem.  Recall that a well-known and important corollary to Dirichlet's approximation theorem says that, for every irrational number $x$, there exists infinitely many $q \in \NN$ and $p \in \ZZ$  such that $(p,q)=1$ for which $|x - \frac p q | < \frac 1{q^2}$ holds.  Let $\II := \RR \backslash\QQ$.  Define \begin{align*}
 &\Dir(\delta) :=  \\ &\left\{x \in \II :\exists \textrm{ infinitely many } (p,q) \in \ZZ \times \NN \textrm{ such that } (p,q)=1 \textrm{ for which } \left|x - \frac p q \right| < \frac \delta{q^2}\textrm{ holds} \right\} .\end{align*}  Note that the elements of $\Dir(\delta)$ are the $\psi(q)$-approximable numbers in $\II$ where $\psi(q) = \delta/q$.  In the literature, the set of $\psi(q)$-approximable numbers in $\RR$ is (often) denoted $W(\psi)$ (see the survey~\cite{BRV16} for example), and thus we have that $W(\psi) =\Dir(\delta) \sqcup \QQ$. Also, Hurwitz's theorem implies that, for $0<\delta < 1 /{\sqrt{5}}$, we have that $\Dir(\delta) \subsetneq \II$.  The complement is $\BA(\delta)$, which is one of the main objects of study in this paper.

\begin{prop}\label{prop:BADeltaComplement}  Let $\delta>0$.  We have that \[\II \backslash \BA(\delta) = \Dir(\delta).\] 
\end{prop}
\begin{proof}
Let $x \in \II \backslash \BA(\delta)$.  Assume that $x \notin \Dir(\delta)$.  This implies that there exists only finitely many $(p,q) \in \ZZ \times \NN$ such that $(p,q)=1$ for which \begin{align}\label{eqnDirDelta}
\left|x - \frac p q \right| < \frac \delta{q^2} \end{align} holds.  Let us denote this finite set by $(p_1, q_1) \cdots (p_n, q_n)$ and let $Q:= \max(q_1, \cdots, q_n)+1$.  Then for all $q \geq Q$ such that $(p, q)=1$, we have that (\ref{eqnBadDelta}) holds, giving that $x \in \BA(\delta)$ which is a contradiction.

Let $x \in \Dir(\delta)$.  Assume that $x \notin \II \backslash \BA(\delta)$.  Then $x \in \BA(\delta)$ and, hence, there exists a $Q\in \NN$ such that, for all integers $q \geq Q$ and for all $p \in \ZZ$ for which $(p,q)=1$, we have that (\ref{eqnBadDelta}) holds.  Since $x \in \Dir(\delta)$, there exist infinitely many $(p,q) \in \ZZ \times \NN$ such that $(p,q)=1$ for which (\ref{eqnDirDelta}) holds.  Choose such a pair $(p, q)$ such that $q \geq Q$.  This implies that $(p, q)$ satisfies both (\ref{eqnBadDelta}, \ref{eqnDirDelta}), which is a contradiction. \end{proof}

If we require $Q=1$ in Definition~\ref{defn:DeltaBANumber}, then the resulting sets, which we denote $\BA_1(\delta)$  for $\delta>0$, have been studied extensively (see~\cite{Jar28} for example).  Also, these sets $\BA_1(\delta)$ are closely related (see~\cite[Section~1.1]{Wei15} for example) to sets $M_N$ of real numbers in $[0,1]$ whose continued fraction expansions have partial quotients that are all bounded by $N=\lfloor \delta^{-1}\rfloor$ where $\lfloor \cdot \rfloor$ is the floor function.  Likewise, the sets $\BA(\delta)$ and the sets of real numbers in $[0,1]$ whose continued fraction expansions have partial quotients that are all eventually bounded by $\lfloor \delta^{-1}\rfloor$ are closely related.  We note that the sets $\BA(\delta)$ are more natural than the sets $\BA_1(\delta)$ for Schmidt games because there is no dependence of $\delta$ on the radius of the initial choice of ball.  This allows us to prove a result like Theorem~\ref{theo:WinningBADeltaLowerHDCtbleIntersect}.  See the remark following the statement of Theorem~\ref{theo:WinningBADelta} and the discussion following the statement of Theorem~\ref{theo:LosingBADeltaUpperHD} for further details.

Finally, note that Proposition~\ref{prop:BAFiltration} implies that $\BA(\delta)$ are sets of Lebesgue measure zero for all $\delta>0$.  In fact, the sets $\BA(\delta)$ for all $\delta >0$ have Hausdorff dimension strictly less than one (Theorem~\ref{theo:LosingBADeltaUpperHD}, its remark, and the discussion following).  Conversely, for all $\delta>0$ small enough, we have that the sets $\BA(\delta)$ have strictly positive Hausdorff dimension (Theorem~\ref{theo:WinningBADeltaLowerHD} and the discussion following), even if we take a finite intersection of translates (Theorem~\ref{theo:WinningBADeltaLowerHDCtbleIntersect}).

\subsection{Schmidt games and $(\alpha, \beta)$-ubiquitously losing sets}\label{sec:LosingEverywhere}  In this section, we recall Schmidt games and introduce Definition~\ref{defn:UbitLos}.  Let $d,n \in  \NN$, $0< \alpha<1$, $0< \beta <1$, and $S \subset \RR^d$.  Schmidt games refer to the $(\alpha, \beta;S)$-games introduced in~\cite[Section~3]{Sch66}.  In an $(\alpha, \beta;S)$-game, two players, Bob and Alice, alternate choosing nested closed balls $B_1 \supset A_1 \supset B_2 \supset A_2 \supset \cdots$ on $\RR^d$ such that $\rho(A_n) = \alpha \rho(B_n)$ and $\rho(B_{n+1}) = \beta \rho(A_{n})$ where $\rho(\cdot)$ denotes the radius of the closed ball.  Note that Bob can make an arbitrary choice of initial ball $B_1$, but we have that $B_{n+1} \in A_n^\beta$ and $A_{n} \in B_n^\alpha$ where, for a closed ball $B$ and a constant $0 < \gamma \leq 1$, the notation $B^\gamma$ denotes the collection of all closed balls $\widetilde{B} \subset B$ such that $\rho(\widetilde{B})= \gamma \rho(B)$.  Alice has a \textit{winning strategy} for the $(\alpha, \beta; S)$-game on $\RR^d$ if Alice can ensure that \[\bigcap_{n=1}^ \infty B_n \in S.\]  The set $S$ is called an \textit{$(\alpha, \beta)$-winning set (of $\RR^d$)}.  Conversely, Bob has a \textit{winning strategy} for the $(\alpha, \beta; S)$-game on $\RR^d$ if Bob can ensure that \[\bigcap_{n=1}^ \infty B_n \in S^c.\]  The set $S$ is called an \textit{$(\alpha, \beta)$-losing set (of $\RR^d$)}. See~\cite{Sch66} (and also Section~\ref{sec:AcceleratedGames}) for more details.  In general, a set $S$ may be neither $(\alpha, \beta)$-winning nor $(\alpha, \beta)$-losing.  However, if $S$ is a Borel set, then Borel determinacy (\cite{Mar75}) implies that $S$ is either $(\alpha, \beta)$-winning or $(\alpha, \beta)$-losing.  See~\cite{Mar82, AFG23, CFJ23} for more details.  The sets, $\BA(\delta)$ and $\Dir(\delta)$, that we study in this paper are Borel sets.  

\begin{rema}\label{rem:PositionalWinning}
Note that, for an $(\alpha, \beta)$-winning set, Alice has a positional winning strategy (\cite[Theorem~7]{Sch66}).  Likewise, the analogous proof of \cite[Theorem~7]{Sch66} shows that, for an $(\alpha, \beta)$-losing set, Bob has a positional winning strategy.  In this paper, we assume that a winning strategy for an $(\alpha, \beta)$-game is a positional winning strategy.
\end{rema}

\begin{defn}\label{defn:UbitLos}  Let $d \in  \NN$, $0< \alpha<1$, and $0< \beta <1$.  A set $S \subset \RR^d$ is an \textbf{$(\alpha, \beta)$-ubiquitously losing set (of $\RR^d$)} if, for every choice of closed ball $B_1$, Bob has a winning strategy for the $(\alpha, \beta; S)$-game on $\RR^d$ and is a \textbf{ubiquitously losing set (of $\RR^d$)} if $S$ is an $(\alpha, \beta)$-ubiquitously losing set of $\RR^d$ for some $\alpha$ and $\beta$.  We refer to $\alpha$ and $\beta$ as the \textbf{winning parameters for Bob}.  Furthermore, let $0< \beta_0 < 1$.  A set $S \subset \RR^d$ is an \textbf{$(\alpha \mid \beta_0)$-ubiquitously losing set (of $\RR^d$)} if $S$ is an $(\alpha, \beta)$-ubiquitously losing set of $\RR^d$ for all $0< \beta \leq \beta_0$.
\end{defn}

\begin{rema}\label{rmk:UbiqLosSet1}  The following observations are immediate.
\begin{enumerate}
\item Every $(\alpha, \beta)$-ubiquitously losing set of $\RR^d$ is an $(\alpha, \beta)$-losing set of $\RR^d$.
\item Every subset of an $(\alpha, \beta)$-ubiquitously losing set of $\RR^d$ is an $(\alpha, \beta)$-ubiquitously losing set of $\RR^d$.
\item There are $(\alpha, \beta)$-losing sets of $\RR^d$ that are not $(\alpha, \beta)$-ubiquitously losing sets of $\RR^d$.  For example, the set $\{0\} \cup [1, 2]$ is an $(\alpha, \beta)$-losing set of $\RR$ for every $0< \alpha <1$ and $0<\beta<1$ but is not an $(\alpha, \beta)$-ubiquitously losing set of $\RR$ for any $0< \alpha <1$ and $0<\beta<1$.
\item  If $S\subset \RR^d$ is an $(\alpha, \beta)$-ubiquitously losing set of $\RR^d$, then $S^c$ is dense in $\RR^d$.
\end{enumerate}

\end{rema}

\begin{rema}\label{rmk:UbiqLosSet2}
The notion of an $(\alpha, \beta)$-ubiquitously losing set (just like the notion of an $(\alpha, \beta)$-losing set defined in~\cite{Sch66}) need not be defined only for $\RR^d$ but can be defined on more general spaces such as for complete metric spaces.  This is what Schmidt does in~\cite{Sch66}.  Similar to~\cite{Sch66}, many of the properties of $(\alpha, \beta)$-ubiquitously losing sets that we show (in Sections~\ref{sec:AcceleratedGames} and~\ref{sec:PropertiesWinAndUbitLose}) are valid over these more general spaces using the same proofs.  We, however, will not need to consider spaces more general than $\RR^d$ in this paper and, thus, will use Definition~\ref{defn:UbitLos}.
\end{rema}

\subsection{Statement of main results}\label{sec:StateMainResults}

\begin{theo}\label{theo:WinningBADelta} Let $0< \alpha \leq 1/3$, $0<\beta<1$, and $\delta:= \frac{\alpha \beta} 6$.  Then $\BA(\delta)$ is an $(\alpha, \beta)$-winning set of $\RR$.

\end{theo}

\begin{rema*}

Note that \cite[Lemma~8]{Sch66} may be applied to Theorem~\ref{theo:WinningBADelta} to obtain other winning parameters.  Also note that Proposition~\ref{prop:BAFiltration} and Theorem~\ref{theo:WinningBADelta} provide finer information than~\cite[Theorem~3]{Sch66}.  In addition, $\BA(\delta)$ can be shown to be $\beta'$-HAW (see~\cite[Definition~2.1]{BK15}) for some $0<\beta' < 1/3$ (see the proofs of~\cite[Theorem~2.5]{BK15} and Theorem~\ref{theo:WinningBADelta}).  Being $\beta'$-HAW will give $\BA(\delta)$ some better properties than $\BA_1(\delta)$.
\end{rema*}  

\noindent Theorem~\ref{theo:WinningBADelta} is proved in Section~\ref{sec:WinningBADelta}.

\begin{theo}\label{theo:LosingBADelta} Let $0<\beta \leq1/2$, and $\delta:=18\beta$.  Then $\BA(\delta)$ is an $(1/2, \beta)$-ubiquitously losing set of $\RR$.

\end{theo}
\begin{rema*}
Note that Proposition~\ref{prop:UbitLosingInvariance} may be applied to Theorem~\ref{theo:LosingBADelta} to obtain other winning parameters for Bob. Also note that Hurwitz's theorem and Proposition~\ref{prop:BADeltaComplement} imply that, for all $\delta \geq 1/\sqrt 5$, we have that $\BA(\delta) = \emptyset$.  Observe that the empty set is an $(\alpha, \beta)$-ubiquitously losing set of $\RR$ for any $0< \alpha <1$ and $0<\beta<1$. 

\end{rema*}

\noindent Theorem~\ref{theo:LosingBADelta} is proved in Section~\ref{sec:prooftheo:LosingBADelta}.  The key observation from the proof is number-theoretical, namely that, roughly speaking, the Farey sequence provides the correct scale for the Schmidt game (via our notions of the half Farey partition and the Farey half-interval, both introduced in Section~\ref{sec:halfFareyPartionHalfInterval}).  The generalization of this observation is to the Farey sequence on the space of unimodular lattices (see~\cite[Section~2]{Tse23}).  Also note that Theorems~\ref{theo:WinningBADelta} and~\ref{theo:LosingBADelta} (and their remarks) give part of a Schmidt diagram (\cite{ZTK13} and~\cite{NZ25}) for $\BA(\delta)$.

Let $\lfloor \cdot \rfloor$ be the floor function, $\lceil \cdot \rceil$ be the ceiling function, and $\dim_H(\cdot)$ denote Hausdorff dimension.

\begin{theo}\label{thm:UpBndHDLosingSets} Let $d \in \NN$, $\RR^d$ be equipped with the supremum norm $\| \cdot \|_\infty$, $s\geq d, j \geq 2$ be integers, $0 < \beta \leq 1/2$, and $S \subset \RR^d$ be a $(1/j, \beta)$-ubiquitously losing set of $\RR^d$.  Then we have that \begin{align}\label{eqn:thm:UpBndHDLosingSets1}
 \dim_H(S) \leq \frac{d \log(j) + \log\left(\left(\lceil j^s \beta^{-s-1}\rceil +1\right)^d -  \left(\lfloor j^s \beta^{-s}\rfloor -1\right)^d\right)}{\log\left(j^{s+1}\beta^{-(s+1)} \right)} < d.\end{align}

 If, furthermore, $1/\beta$ is an integer, then we have that \begin{align}\label{eqn:thm:UpBndHDLosingSets2}
 \dim_H(S) \leq \frac{d \log(j) + \log\left(\left(j^s \beta^{-s-1} +1\right)^d - j^{sd} \beta^{-sd}\right)}{\log\left(j^{s+1}\beta^{-(s+1)} \right)} <d . \end{align}
 \end{theo}

\begin{rema*}
 Since all norms on $\RR^d$ are equivalent, the analogous result for another norm on $\RR^d$ can be obtained by applying Proposition~\ref{prop:BilipschitzInvar2} to Theorem~\ref{thm:UpBndHDLosingSets}.  Also, as an example, for $d=s=1$, $j=2$, and $\beta = 1/2$, Theorem~\ref{thm:UpBndHDLosingSets} gives that \[\dim_H(S) \leq 1+ \frac{\log\left(1 - \beta + \frac{\beta^2}{j}\right)} {2 \log\left(\frac j \beta \right)} = 1 + \frac{\log(1-\frac 3 8)}{4 \log(2)}\approx0.83.\] 
\end{rema*}

\noindent Theorem~\ref{thm:UpBndHDLosingSets} is proved in Section~\ref{subsec:ProofThm:UpBndHDLosingSets}.  Heuristically, the gist of the proof is to follow the estimation of an upper bound for the upper box dimension of the Cantor set and fill in the removed sets with suitably scaled copies of the Cantor set.  Significant complications, however, arise when this heuristic is applied to Schmidt games.  These complications are resolved in Section~\ref{subsec:ProofThm:UpBndHDLosingSets} by our construction, using inner and outer recursions, of plays that are winning for Bob for various $(1/j, \beta; S)$-accelerated games for Bob.  Such plays and games are introduced in Section~\ref{sec:AcceleratedGames}.

\begin{theo}\label{theo:WinningBADeltaLowerHD}  For $0 < \delta < \frac 1 {18}$, we have that \[\dim_H\left(\BA(\delta)\right) \geq  \frac{\log\left(\lfloor \frac 1 {18 \delta}\rfloor\right)}{\log\left(\frac 1 {6 \delta} \right)}.\]
\end{theo}

\begin{proof}
Theorem~\ref{theo:WinningBADelta} gives that $\BA(\delta)$ is an $(1/3, 18 \delta)$-winning set.  Applying Theorem~\ref{thm:LowerBndHDWinSet} yields the desired result.
\end{proof}

\begin{theo}\label{theo:LosingBADeltaUpperHD}  Let $0 < \delta < 1/\sqrt{5}$.  We have that \[\dim_H\left(\BA(\delta)\right) \leq \frac{\log 2 + \log\left(\lceil\frac{648}{\delta^2} \rceil - \lfloor \frac{36}{\delta} \rfloor + 2 \right)}{2 \log\left(\frac {36}\delta \right)} <1. \]  If, furthermore, $18/\delta \in \NN$, we have that \[\dim_H\left(\BA(\delta)\right) \leq 1+ \frac{\log\left(1 - \frac {\delta}{18} + \frac{\delta^2}{648}\right)} {2 \log\left(\frac {36} \delta \right)}<1.\]
\end{theo}

\begin{rema*}
For $\delta \geq 1/\sqrt{5}$, Hurwitz's theorem and Proposition~\ref{prop:BADeltaComplement} yield that $\BA(\delta) = \emptyset$.

\end{rema*}
\begin{proof}  For $0< \delta \leq 9$, Theorem~\ref{theo:LosingBADelta} gives that $\BA(\delta)$ is a $(1/2, \delta/18)$-ubiquitously losing set of $\RR$.  (Note that the empty set is an $(\alpha, \beta)$-ubiquitously losing set for any $0 < \alpha <1$ and $0<\beta<1$.)   Applying Theorem~\ref{thm:UpBndHDLosingSets} yields the desired result.  \end{proof}

\noindent Theorems~\ref{theo:WinningBADeltaLowerHD} and~\ref{theo:LosingBADeltaUpperHD} give alternate proofs of known results due to Jarn\'ik~\cite[Satz~4]{Jar28}.  Moreover, an asymptotic formula due to Hensley is also known~\cite{Hen92}.  These have further strengthenings and generalizations (see~\cite{Sha92, Kur51, BK15, Wei15, Sim18, DFSU23} for some examples, including recent ones).  While Theorems~\ref{theo:WinningBADeltaLowerHD} and~\ref{theo:LosingBADeltaUpperHD} do not give as strong results as these, there are some advantages.  First, while many of these known results are proved using techniques from analytic number theory or homogeneous dynamics, our technique, which is intrinsic to Schmidt games, may be more widely applicable (Remark~\ref{rem:Discuss}), especially the using of ubiquitously losing sets for upper bounds on the dimension.  Second, our technique, due to properties of Schmidt games, easily handles intersections.  Theorem~\ref{theo:WinningBADeltaLowerHDCtbleIntersect} is an example of this.  Other examples, involving other $(\alpha, \beta)$-winning sets and/or other bilipschitz mappings, can also be constructed.  Note that the intersection property is powerful because it asserts that subsets of $\RR^d$ meet even if they arise in, perhaps, vastly different ways, as long as they both are $(\alpha, \beta)$-winning for suitable winning parameters.

\begin{theo}\label{theo:WinningBADeltaLowerHDCtbleIntersect}  Let $N \in \NN$, $\{r_i\}_{i=1}^N$ be a collection of real numbers, and $0 < \delta < \frac { 3^{1-N}} {18}$.  Then $\bigcap_{i=1}^N \left(\BA\left( \delta\right)+r_i\right)$ is a dense subset of $\RR$, and, moreover, we have that   \[\dim_H\left(\bigcap_{i=1}^N \left(\BA\left( \delta\right)+r_i\right)\right) \geq  \frac{N \log\left(\left\lfloor \left(\frac 1 {18 \delta \cdot 3^{N-1}} \right)^{1/N}\right\rfloor\right)}{\log\left(\frac 1 {6 \delta} \right)}.\] 
\end{theo}

\begin{proof}
 Apply the proof of Theorem~\ref{theo:WinningBADeltaLowerHD} with \cite[Proposition~5.3]{Dan87} and Theorem~\ref{thm:CtbleWinPropForWinSets} (with its remark).
\end{proof}

\begin{rema}\label{rem:Discuss}  The theory of Schmidt games is a natural tool to study not only $\BA$ but also various generalizations and related sets in Diophantine approximation and also in ergodic theory (see \cite{Sch69, Ts09, ET11,Ts09Non, Ts16, Fi09, MT13, MM13, ABV18,GW18, BNY21, BNY22, DS25,DS26} for some examples).  It may be that the analogous notion of $\BA(\delta)$ to $\BA$ exists for some of these generalizations and related sets and would provide insight into their finer structure, further enhancing the utility of the theory of Schmidt games.  Note that, applied to ergodic theory, this would deepen our understanding of exceptional orbits, which is important because a central goal of the field of dynamical systems, of which ergodic theory is a subfield, is to better understand orbits.  Moreover, Schmidt games have motivated a number of other games (\cite{McM10, KW10, BFKRW12, BHNS25} for example), and it may be that ubiquitously losing sets have suitable analogs for some of these other games and would serve a function similar to that in this paper.  \end{rema}

\subsubsection*{Acknowledgements}  The author wishes to thank Dmitry Kleinbock for pointing him to~\cite{DFSU23} and Barak Weiss for pointing him to~\cite{Sim18} and wishes to thank both for their helpful comments regarding the dimensions of the sets $M_N$.  For the purpose of open access, the author has applied a Creative Commons Attribution (CC BY) licence to any Author Accepted Manuscript version arising from this submission.  This study did not generate any new data.

\section{Winning properties of $\BA(\delta)$}\label{sec:WinningBADelta}  In this section, we prove Theorem~\ref{theo:WinningBADelta} (see Section~\ref{subsec:proof:theo:WinningBADelta}).  The proof is an adaption of the proof of~\cite[Theorem~3]{Sch66}.  To begin, we need the following proposition.  

\begin{prop}\label{prop:WinningBADelta}  Let $0< \alpha \leq 1/3$, $0<\beta<1$, and $\delta:= \frac{\alpha \beta} 6$.  In any $(\alpha, \beta)$-game on $\RR$ for which $\rho(B_1) \leq 1/2$, Alice can ensure any \begin{align*} x := \bigcap_{n=1}^\infty B_n \end{align*} is an irrational real number and satisfies (\ref{eqnBadDelta}) for all integers \[q \geq \sqrt{\frac{\alpha \beta}{2\rho(B_1)}}\] and all $p \in \ZZ$ such that $(p,q) =1$.

\end{prop}

\subsection{Proof of Proposition~\ref{prop:WinningBADelta}}\label{subsec:Proof:prop:WinningBADelta}

Set $\rho:= \rho(B_1)$, $R := \frac 1 {\sqrt{\alpha \beta}}$, and \[Q:= \left\lceil \sqrt{\frac{\alpha \beta}{2\rho}} \right\rceil.\]  Define \[\begin{cases}  \mQ := \left\{ \frac p q \in \QQ : 1 \leq q \leq Q-1 \textrm{ and } (p, q) = 1 \right\} & \text{ if } Q\geq 2 \\
\mQ := \emptyset & \text{ if } Q = 1\end{cases}.\]
\begin{lemm}\label{lemm:WinningBADelta}   Let $n, q \in \NN$.  If \begin{align*}
x \in B_n \textrm{ and }  \sqrt{\frac{\alpha \beta}{2\rho}} \leq q < \sqrt{\frac{\alpha \beta}{2\rho}} R^{n-1}\end{align*} holds, then Alice can ensure\begin{align}\label{eqn:lemm:WinningBADelta:FirstResult} \textrm{
  that (\ref{eqnBadDelta}) holds for any $p \in \ZZ$ such that $(p,q)=1$.} \end{align}
  In addition, Alice can ensure \begin{align}\label{eqn:lemm:WinningBADelta:SecondResult}
 B_2 \cap \mQ = \emptyset. \end{align}
 
\end{lemm}
\begin{proof}
 We first prove (\ref{eqn:lemm:WinningBADelta:FirstResult}).  The proof is by induction on $n$.  The initial step $n=1$ is vacuously true.  We assume that the lemma holds for $n$ and show that it holds for $n+1$.  Thus, the induction hypothesis implies that we only need to consider integers $q$ lying in the interval \begin{align}\label{eqn:lemm:WinningBADelta}
\left[\sqrt{\frac{\alpha \beta}{2\rho}}  R^{n-1}, \sqrt{\frac{\alpha \beta}{2\rho}} R^{n}\right). \end{align}  

\noindent If the interval (\ref{eqn:lemm:WinningBADelta}) does not contain any integers, then Alice can choose $A_n$ to be any element of $B_n^\alpha$ to obtain (\ref{eqn:lemm:WinningBADelta:FirstResult}), the desired result.  Otherwise, the interval (\ref{eqn:lemm:WinningBADelta}) does contain integers.

We have following observation.  Let $q \in \NN$ lie in (\ref{eqn:lemm:WinningBADelta}) and $p \in \ZZ$ be such that $(p, q)=1$.  Note that we have $\frac {\delta} {q^2}\leq \frac{1}{3} (\alpha \beta)^{n-1} \rho = \frac{1}{3} \rho(B_{n})$.  Consequently, to complete the proof of (\ref{eqn:lemm:WinningBADelta:FirstResult}), we require that Alice choose $A_n$ to avoid the open ball of radius $\frac{1}{3} \rho(B_{n})$ around any such $p/q$.  We will show that $A_n$ can be chosen in this way.

Note that, if there are distinct integers $q, q'$ lying in the interval (\ref{eqn:lemm:WinningBADelta}) for which $(p,q)=1$ and $(p',q')=1$, then we have that \begin{align} \label{eqn2:lemm:WinningBADelta}\left | \frac p q - \frac {p'}{q'}\right | \geq \frac 1 {q q'} > \frac{2\rho}{(\alpha \beta) R^{2n}}= 2\rho(B_{n}). \end{align}  Consequently, there are two cases.  \begin{description}
 
\item[Case 1]  There is a unique such $q$ for which $p/q \in B_n$.  
\item[Case 2]  For every integer $q$ lying in the interval (\ref{eqn:lemm:WinningBADelta}) and every $p \in \ZZ$ such that $(p,q)=1$, we have that $p/q \notin B_n$.
\end{description}

Let us consider each case.  For Case 1, we have that there exists a unique such $q$.  Note that, if there exist distinct $p, p' \in \ZZ$ such that $p/q, p'/q \in B_n$, then $2 \rho(B_n) \geq 1/q >\sqrt{2 \rho(B_n)}$, which implies that $\rho \geq \rho(B_n) > 1/2$, a contradiction.  Thus, the integer $p$ is also unique.  Consequently, we have a unique $p/q \in B_n$.

Now let $b_n$ denote the center of $B_n$.  If $p/q \leq b_n$, then \[\frac p q + \frac{1}{3} \rho(B_{n}) \leq b_n +  \frac{1}{3} \rho(B_{n}) \quad \textrm{ and } \quad \frac p q +\rho(B_{n}) \leq b_n +  \rho(B_{n}).\]  Let $\tilde{q}$ be an integer which lies in (\ref{eqn:lemm:WinningBADelta}) and $\tilde{p} \in \ZZ$ such that $(\tilde{p}, \tilde{q})=1$.  If $p/q < \tilde{p}/\tilde{q}$, then  (\ref{eqn2:lemm:WinningBADelta}) implies that \[\frac p q + \rho(B_{n}) < \frac p q + \frac 5 3 \rho(B_{n}) < \frac {\tilde{p}} {\tilde{q}} - \frac 1 3 \rho(B_{n}).\]  As $\alpha \leq 1/3$, let the closed interval $A_n$ be contained in the interval \begin{align}\label{eqn:lemm:WinningBADelta:Case1main}
 \left[\frac p q + \frac{1}{3} \rho(B_{n}), \frac p q + \rho(B_{n})\right]. \end{align}  Applying the observation, this gives (\ref{eqn:lemm:WinningBADelta:FirstResult}) for Case 1 when $p/q \leq b_n$.  Note that we may take the same $A_n$ if no such $\tilde{q}$ and $\tilde{p}$ exist.  For $p/q > b_n$, the analogous proof gives (\ref{eqn:lemm:WinningBADelta:FirstResult}).  This proves (\ref{eqn:lemm:WinningBADelta:FirstResult}) for Case 1.

For Case 2, first consider distinct integers $q, q'$ lying in the interval (\ref{eqn:lemm:WinningBADelta}) for which $(p,q)=1$ and $(p',q')=1$.  As we are in Case~2, $p/q$ and $p'/q'$ do not lie in $B_n$.  However, $p/q$ and $p'/q'$ may be near the endpoints of $B_n$.  Without loss of generality, we may assume that $p/q < p'/q'$ and that these are the closest such to $B_n$.  Since (\ref{eqn2:lemm:WinningBADelta}) and $\alpha \leq 1/3$ hold, if Alice were to choose $A_n$ to be contained in \begin{align}\label{eqn:lemm:WinningBADetla:ClosedMiddleTwoThirds}
 \left[b_n- \frac 2 3 \rho(B_n)-\varepsilon, b_n+ \frac 2 3 \rho(B_n)+\varepsilon'\right], \end{align} where \begin{align*}
 \varepsilon &:=  \min\left(b_n- \rho(B_n) -\frac p q, \frac 1 3 \rho(B_n)\right)>0, \\  
 \varepsilon' &:=  \min\left( \frac{p'}{q'} - (b_n +  \rho(B_n)), \frac 1 3 \rho(B_n)\right)>0.  \end{align*}  Then, by the observation, (\ref{eqn:lemm:WinningBADelta:FirstResult}) will now follow for Case 2 when $q, q'$ are distinct. 

Now, for $Q \geq 2$, let $p''/q'' \in \mQ$.  Then note that (\ref{eqn2:lemm:WinningBADelta}) still holds if we replace $p'/q'$ by $p''/q''$.  Consequently, at most one element $p''/q''$ of $\mQ$ for $Q \geq 2$ lies in the interval (\ref{eqn:lemm:WinningBADetla:ClosedMiddleTwoThirds}).  If $p''/q'' < b_n$, then Alice chooses $A_n$ to be contained in \[\left [b_n+\varepsilon', b_n+ \frac 2 3 \rho(B_n)+\varepsilon'\right].\]  Otherwise, we have that $p''/q'' \geq b_n$ and Alice chooses $A_n$ to be contained in \[\left [b_n- \frac 2 3 \rho(B_n)-\varepsilon, b_n-\varepsilon\right].\]

If $Q=1$, then Alice chooses $A_n$ to be contained in the interval (\ref{eqn:lemm:WinningBADetla:ClosedMiddleTwoThirds}).  If $q = q'$, but $p \neq p'$, (\ref{eqn2:lemm:WinningBADelta}) still holds and we may apply the proof for distinct $q, q'$ to show (\ref{eqn:lemm:WinningBADelta:FirstResult}).  

Now, if exists exactly one $q$ and $p$, then we must modify the proof of Case~2 for distinct $q, q'$ as follows.  As we are in Case~2, we have that $p/q \notin B_n$.  Therefore, $p/q$ is either \begin{enumerate}[label=(\alph*)]
\item strictly less than all the elements of $B_n$ or 
\item strictly greater than all the elements of $B_n$.
\end{enumerate}  First, let us consider Case~(a).  There is no $p'/q'$.  Set $ \varepsilon':=\frac 1 3 \rho(B_n)$ in the proof for distinct $q, q'$ to obtain $A_n$ to give (\ref{eqn:lemm:WinningBADelta:FirstResult}) for Case~(a).  Next, let us consider Case~(b).  Since $p/q$ is strictly greater than all the elements of $B_n$, rename it $p'/q'$ and omit $p/q$.  Therefore, we have that $q'$ is the only integer lying in the interval (\ref{eqn:lemm:WinningBADelta}) for which $(p',q')=1$.   Now set $ \varepsilon:=\frac 1 3 \rho(B_n)$ in the proof for distinct $q, q'$ to obtain $A_n$ to give (\ref{eqn:lemm:WinningBADelta:FirstResult}) for Case~(b).  This shows (\ref{eqn:lemm:WinningBADelta:FirstResult}) for all cases and completes the proof of (\ref{eqn:lemm:WinningBADelta:FirstResult}).

We now prove (\ref{eqn:lemm:WinningBADelta:SecondResult}).  Let $Q=1$,  Then $\mQ = \emptyset$ and (\ref{eqn:lemm:WinningBADelta:SecondResult}), the desired result, holds.  Now let $Q \geq 2$ and  $p''/q'' \in \mQ$.  Since (\ref{eqn2:lemm:WinningBADelta}) still holds if we replace $p'/q'$ by $p''/q''$, we have that the interval (\ref{eqn:lemm:WinningBADelta:Case1main}) for $n=1$ does not contain any element of $\mQ$.  Consequently, for $p/q \leq b_1$ in Case~1 of the construction above, we have that $A_1 \cap \mQ = \emptyset$ and, thus, $B_2 \cap \mQ = \emptyset$.  Analogously, we obtain $B_2 \cap \mQ = \emptyset$ for $p/q > b_1$ in Case~1 of the construction above.  Finally, for Case~2, we have $A_1 \cap \mQ = \emptyset$ by construction.  Thus, we have $B_2 \cap \mQ = \emptyset$ in all cases and, thus (\ref{eqn:lemm:WinningBADelta:SecondResult}) holds.  This proves the lemma.

\end{proof}

The Lemma implies that for any \[x \in \bigcap_{n=1}^\infty B_n,\] we have that (\ref{eqnBadDelta}) holds for any integer $q \geq Q$ and any $p \in \ZZ$ such that $(p,q)=1$.  Therefore, $x \notin \QQ \backslash \mQ$.  The Lemma also implies that $\{x\} \cap \mQ = \emptyset$.  Thus, $x$ is an irrational real number.  This proves the proposition.

\subsection{Proof of Theorem~\ref{theo:WinningBADelta}}\label{subsec:proof:theo:WinningBADelta}

We play an $(\alpha, \beta)$-game.  If $\rho(B_1) > 1/2$, then there exists an integer $N \geq 2$ such that $\rho(B_N) \leq 1/2 < \rho(B_{N-1})$.  Alice regards $B_N$ as Bob's first choice of ball and, by reindexing, refers to it as $B_1$.  Applying Proposition~\ref{prop:WinningBADelta} implies that \[\{x \in \II: \exists Q \in \NN \textrm{ such that (\ref{eqnBadDelta}) holds } \forall (p,q) \in \ZZ \times \NN \textrm{ for which } q \geq Q \textrm { and }(p,q)=1\}\] is an $(\alpha, \beta)$-winning set of $\RR$. This proves the theorem.

\section{Ubiquitously losing properties of $\BA(\delta)$}\label{sec:LosingBADelta}

In this section, we prove Theorem~\ref{theo:LosingBADelta} (see Section~\ref{sec:prooftheo:LosingBADelta}).  Our main tool is the Farey sequence (see~\cite[Chapter~3] {HW08} or ~\cite[Chapter~1]{Hat22} for an introduction).  Let $n \in \NN$.  Recall that the Farey sequence $\F_n$ of order $n$ is the ascending sequence:\[\F_n:= \left\{\frac p q \in [0,1): (p, q) \in \ZZ^2 \textrm{ such that } 1 \leq q \leq n \textrm{ and } (p, q)=1 \right\} \cup \left\{\frac 1 1 \right\}.\]  Note that $\F_n$ is ordered by the usual ordering $<$ on $\RR$.  We will use the following two properties:

\begin{lemm}[{\cite[Theorem~28]{HW08}}]\label{lemm:Fareypair}  Let $n \in \NN$.  If $p/q$ and $p'/q'$ are consecutive elements of $\F_n$, then
\[ \frac {p'} {q'}-  \frac p q  = \frac 1{qq'}.\]

\end{lemm}

\noindent Let $p/q$ and $p'/q'$ be consecutive elements of $\F_n$.  The \textit{mediant} of $p/q$ and $p'/q'$ is \[\frac{p+p'}{q+q'}.\]  Note that $q+q' > n$ (\cite[(3.1.4)]{HW08}) and $(p+p', q+q')=1$ (\cite[(2) of Section~3.2]{HW08}).

\begin{lemm}[{\cite[Theorem~29 and Section~3.2]{HW08}}]\label{lemm:Mediant} Let $n \in \NN$ and let $p/q$ and $p'/q'$ be consecutive elements of $\F_n$.  Then the following hold.
\begin{itemize}
\item For all integers $m$ such that $n \leq m < q+q'$, there are no elements of $\F_m$ between $p/q$ and $p'/q'$.
\item For $m = q+q'$, \[\frac p q, \quad \frac {p+p'}{q+q'}, \quad \frac {p'} {q'}\] are consecutive elements of $\F_m$.
\end{itemize}
\end{lemm}

\noindent Finally, define the \textit{Farey sequence of order $n$ in $\RR$} to the be ascending sequence \[\F_n^{\ZZ} := \bigcup_{k \in \ZZ} \left\{\frac p q + k : \frac p q \in \F_n \backslash \left\{\frac 1 1 \right\}\right\}.\]  As $0/1+k = 1/1 + (k-1)$ holds for all $k \in \ZZ$,  Lemmas~\ref{lemm:Fareypair} and~\ref{lemm:Mediant} both apply to $\F_n^{\ZZ}.$

\subsection{The half Farey partition and the Farey half-interval}\label{sec:halfFareyPartionHalfInterval}

We now use the Farey sequence to construct a partition near certain rational points.  This partition, roughly speaking, will give suitable distances between relevant points so that Bob can have a winning strategy for the Schmidt game in Theorem~\ref{theo:LosingBADelta}.  Let $I \subset \RR$ be a closed interval such that $0<D:=\diam(I) < 1$.  If $I \cap \F_1^{\ZZ} \neq \emptyset$, then $|I \cap \F_1^{\ZZ}|=1$ and \[I \cap \F_1^{\ZZ}=\frac 0 1 +k=k\] for some $k \in \ZZ$.  We refer to $0/1+k$ as the {\em minimal-order Farey element of $I$} and define the {\em order} of this minimal-order Farey element of $I$ to be $1$.  Let us, furthermore, denote the order of this minimal-order Farey element of $I$ by $\eO(I):=\eO(0/1+k)$, and, consequently, we have that $\eO(I)=1$.  Finally, we define the {\em half Farey partition for $I$} to be the ascending sequence \[\left\{ k - \frac 1 2, k, k + \frac 1 2\right\}\] and the {\em cover} formed by the half Farey partition for $I$ to be the closed interval $[k-1/2, k+1/2]$.

Otherwise, we have that $I \cap \F_1^{\ZZ} = \emptyset$, and, thus, there exists a least integer $q \geq 2$ such that $I \cap \F_{q-1}^{\ZZ} = \emptyset$ and $I \cap \F_{q}^{\ZZ} \neq \emptyset$.  By Lemma~\ref{lemm:Mediant}, we have that $|I \cap \F_q^{\ZZ}|=1$ and \[I \cap \F_q^{\ZZ}= \frac p q + k\] for some $k \in \ZZ$ and some integer $0<p<q$ such that $(p,q)=1$.  We refer to $p/q+k$ as the {\em minimal-order Farey element of $I$} and define the {\em order} of this minimal-order Farey element of $I$ to be $q$.  Let us, furthermore, denote the order of this minimal-order Farey element of $I$ by $\eO(I):=\eO(p/q+k)$, and, consequently, we have that $\eO(I)=q$.  Now let \[\frac a b + k, \quad  \frac p  q + k, \quad \frac c d + k\] be consecutive elements of $\F_q^{\ZZ}$.  Since $p/q$ is the mediant of $a/b$ and $c/d$, we have that $p = a+c$ and $q = b+d$.  Let \begin{align}\label{eqn:halfFareyPartitionPara}
 l := l(I):=\frac q b  \quad \textrm{ and }\quad r:= r(I):=\frac q d. \end{align}  Note that Lemma~\ref{lemm:Fareypair} implies that $l$ and $r$ are not integers.  We define the {\em half Farey partition for $I$} to be the ascending sequence whose elements are the elements of the union of the two sets \begin{align*}& \left\{\frac p q +k, \frac {a+p} {b+q} +k,  \frac {2a+p} {2b+q} +k, \cdots, \frac {\lceil l \rceil a+p} {\lceil l \rceil b+q} +k\right\} \\ & \left\{\frac p q +k, \frac {p+c} {q+d} +k,  \frac {p+2c} {q+2d} +k, \cdots, \frac {p+\lceil r \rceil c} {q+\lceil r \rceil d} +k \right\}.
  \end{align*} Note that, for all integers $1 \leq j \leq \lceil r \rceil$, we have that \[ \frac {p+jc} {q+jd}\] is the mediant of \begin{align}\label{eqn:iterativeDefnHalfFareyPartition}
 \frac {p+(j-1)c} {q+(j-1)d}  \quad \textrm{ and } \quad \frac c d, \end{align} and, consequently, we have that \[\frac {p+(j-1)c} {q+(j-1)d} <  \frac {p+jc} {q+jd}<  \frac c d.\]  Analogously, for all integers $1 \leq j \leq \lceil l \rceil$, we have that \[ \frac {ja+p} {jb+q}\] is the mediant of \[ \frac a c \quad \textrm{ and } \quad \frac {(j-1)a+p} {(j-1)b+q},\] and, consequently, we have that \[\frac a c < \frac {ja+p} {jb+q}< \frac {(j-1)a+p} {(j-1)b+q}.\]  We define the {\em cover} formed from the half Farey partition for $I$ to be the closed interval \[\left[\frac {\lceil l \rceil a+p} {\lceil l \rceil b+q}+k, \frac {p+\lceil r \rceil c} {q+\lceil r \rceil d} +k\right].\]

\begin{lemm}\label{lemm:WeakIncOrderMinFarey} Let $I \subset \RR$ be a closed interval with $0<D:=\diam(I) < 1$ and $J$ be a closed interval contained in $I$ such that $\diam(J)>0$.  Then $\eO(I) \leq \eO(J)$.   Moreover, if $\eO(I)=\eO(J)$, then the minimal-order Farey element of $I$ is equal to the minimal-order Farey element of $J$.
 
\end{lemm}  
  \begin{proof}
 The result follows from the definitions of minimal-order Farey element and its order.
\end{proof}
  \begin{lemm}\label{lemm:halfFareyPartitionSeparate} Let $I \subset \RR$ be a closed interval with $0<D:=\diam(I) < 1$ and minimal-order Farey element $p/q+k$.  Let $s/t+k$ and $u/v+k$ be consecutive elements of the half Farey partition for $I$.  Then we have \[\frac 1 {6 q^2} < \frac u v - \frac s t < \frac 1 {q^2}.\]
\end{lemm}

\begin{proof}
 We have that $s/t+k<u/v+k$. For $p=0$, we have that $q=1$ and the result follows from the definition of the half Farey partition for $I$.  Otherwise, we have that $p \neq 0$.  Consider the case that $p/q+k \leq s/t+k$.  By (\ref{eqn:iterativeDefnHalfFareyPartition}), we have that \[ \begin{cases}  s &= p+(j-1)c \\ t &=q+(j-1)d \\ u &= p+jc \\ v &= q+jd  \end{cases} \] for some $1 \leq j \leq \lceil r \rceil$.  By Lemma~\ref{lemm:Fareypair}, we have that \[\frac 1 {6 q^2} \leq \frac1 {(q+rd)(q+(r+1)d)}<\frac u v - \frac s t  = \frac 1 {(q+(j-1)d)(q+jd)} < \frac 1 {q^2}.\]  This shows the case that $p/q+k \leq s/t+k$.
 
 Otherwise, we have that $p/q+k > s/t+k$, which implies that $p/q+k \geq u/v+k$.  This case is proved analogously and completes the proof the lemma.  \end{proof}

 \begin{lemm}\label{lemm:halfFareyPartitionLength} Let $I \subset \RR$ be a closed interval with $0<D:=\diam(I) < 1$ and minimal-order Farey element $p/q+k$ such that $p\neq0$.  Let $a/b+k$, $p/q+k$, and $c/d+k$ be consecutive elements of $ \F_q^{\ZZ}$.  Then we have \begin{align*}\frac {p+\lceil r \rceil c} {q+\lceil r \rceil d} +k  - \left(\frac p q + k \right) >  \frac 1{2dq} \quad \textrm { and } \quad
  \frac p q + k - \left( \frac {\lceil l \rceil a+p} {\lceil l \rceil b+q} +k  \right) > \frac 1{2bq}.
   \end{align*}

\end{lemm}
\begin{proof}

By Lemma~\ref{lemm:Fareypair}, we have that \begin{align}\label{eqn:halfFareyDist}
 \frac c d + k - \left(\frac p q + k \right) = \frac 1{dq} \quad \textrm { and } \quad
  \frac p q + k - \left(\frac a b + k  \right) = \frac 1{bq}. \end{align}  By (\ref{eqn:iterativeDefnHalfFareyPartition}) and Lemma~\ref{lemm:Mediant} , we have that \[\frac {p+\lceil r \rceil c} {q+\lceil r \rceil d} +k \quad \textrm{ and } \quad \frac c d +k\] are consecutive elements of $\F^{\ZZ}_{q+\lceil r \rceil d}$.  Whence, Lemma~\ref{lemm:Fareypair} and the definition of $r$ from (\ref{eqn:halfFareyPartitionPara}) imply that \[ \frac c d +k - \left(\frac {p+\lceil r \rceil c} {q+\lceil r \rceil d} +k \right) = \frac1 {d(q+\lceil r \rceil d)} < \frac 1 {2dq}.\]  The analogous proof gives that \[ \frac {\lceil l \rceil a+p} {\lceil l \rceil b+q} +k  - \left(\frac a b + k\right) = \frac1 {(\lceil l \rceil b+q)b} < \frac 1 {2bq}.\]  Consequently, these bounds together with (\ref{eqn:halfFareyDist}) give that \begin{align*}\label{eqn:halfFareyDist2}\frac {p+\lceil r \rceil c} {q+\lceil r \rceil d} +k  - \left(\frac p q + k \right) >  \frac 1{2dq} \quad \textrm { and } \quad
  \frac p q + k - \left( \frac {\lceil l \rceil a+p} {\lceil l \rceil b+q} +k  \right) > \frac 1{2bq}.
   \end{align*}

\end{proof}

 \begin{lemm}\label{lemm:halfFareyPartitionCover}  Let $I:= [\xi, \eta] \subset \RR$ be a closed interval with $0<D:=\diam(I) < 1$ and minimal-order Farey element $p/q+k$.  Let $L:= L(I):= p/q+k- \xi$ and $R:= R(I):= \eta - (p/q+k)$.  Then the closed interval \[\F(I):= \F\left(I, \frac p q +k \right) := \left[ \frac p q+k - \frac L 2, \frac p q+k + \frac R 2 \right]\] is contained in both $I$ and the cover formed from the half Farey partition for $I$.  Note that $\diam(\F(I))= L/2 + R/2 = D/2$.

\end{lemm}

\begin{rema}
We will refer to $\F(I)$ as the {\em Farey half-interval of $I$}, $L(I)$ as the {\em left Farey half-length of $\F(I)$}, and $R(I)$ as the {\em right Farey half-length of $\F(I)$}.

\end{rema}
\begin{proof}
For $p=0$, the desired result follows from the definition of cover.  For $p \neq 0$, consider the following.  Let $a/b+k$, $p/q+k$, and $c/d+k$ be consecutive elements of $ \F_q^{\ZZ}$.  Since $p/q+k$ is the minimal-order Farey element for $I$, we have that $I \subsetneq [a/b+k,  c/d+k]$ and, hence, $R < 1/dq$ and $L < 1/bq$ by Lemma~\ref{lemm:Fareypair}.  Using these inequalities along with Lemma~\ref{lemm:halfFareyPartitionLength} gives \begin{align*}
\frac {p+\lceil r \rceil c} {q+\lceil r \rceil d} +k &>   \left(\frac p q + k \right) +  \frac 1{2dq} > \left(\frac p q + k \right)+  \frac R 2  \\
\frac {\lceil l \rceil a+p} {\lceil l \rceil b+q} +k &< \frac p q + k -  \frac 1{2dq} < \frac p q + k -  \frac L{2}, \end{align*} which shows that $\F(I)$ is a subset of the cover formed from the half Farey partition for $I$.  Finally, by construction, $\F(I) \subset I$ and has  diameter $D/2$.  This proves the lemma.  \end{proof}

\subsection{Proof of Theorem~\ref{theo:LosingBADelta}}\label{sec:prooftheo:LosingBADelta}

In this section (Section~\ref{sec:prooftheo:LosingBADelta}), we set $\alpha:=1/2$.  To show the desired result, we must show that, for any choice of $B_1$, Bob has a winning strategy for the $(\alpha, \beta; \BA(\delta))$-game on $\RR$. Let $B_1$ be a closed ball.  If $\rho(B_1) > 1/2$, then there exists an integer $M \geq 2$ such that $\rho(B_M) \leq 1/2 < \rho(B_{M-1})$.  Bob regards $B_M$ as his first choice of ball and, by reindexing, refers to it as $B_1$.  Set $\rho:= \rho(B_1)$.  We, thus, have \begin{align}\label{eqn:BobFirstBallRadius}
 \rho \leq 1/2.   \end{align}

Consider $A_1$.  If $A_1 \cap \F_1^{\ZZ} \neq \emptyset$, set $q=1$.  As $2\rho(A_1) \leq \alpha$, we have that $|A_1 \cap \F_1^{\ZZ}| = 1$ and \[A_1 \cap \F_1^{\ZZ} = \frac 0 1 +k = k\] where $\gamma:=0/1+k$ is the minimal-order Farey element of $A_1$.  Otherwise, there exist a least integer $q \geq 2$ such that $A_1 \cap \F_q^{\ZZ} \neq \emptyset$ and $A_1 \cap \F_{q-1}^{\ZZ} = \emptyset$.  By Lemma~\ref{lemm:Mediant}, we have that $|A_1 \cap \F_q^{\ZZ}|=1$ and \[A_1 \cap \F_q^{\ZZ} = \frac p q +k = \frac{p+qk} q \] where $\gamma:=p/q+k$ is the minimal-order Farey element of $A_1$.  Since $(p, q)=1$, we also have that $(p + q k, q) =1$.   Also let $a/b+k$, $p/q+k$, and $c/d+k$ be consecutive elements of $ \F_q^{\ZZ}$.

Let us consider the case that $\rho(A_1) < 1/{q^2}.$  Bob chooses $B_2 \subset \F(A_1)$ such that $p/q+k \in B_2$, and such a choice for Bob is possible because of Lemma~\ref{lemm:halfFareyPartitionCover}.  Consequently, we have that \begin{align}\label{eqn:MainDirProof0}
 \left | x - \frac{p+qk} q \right | \leq 2 \rho(B_2) < \frac {2 \beta} {q^2} \end{align} for all $x \in B_2$.  Now Alice chooses $A_2 \subset B_2$.

Otherwise, we have the case that \begin{align}\label{eqn:RadA1Big}
 \rho(A_1) \geq \frac 1 {q^2}.  \end{align}  Note that, for this case, we have that $q \geq 2$ by (\ref{eqn:BobFirstBallRadius}) and thus, $p \neq 0$.  Let $a_1$ be the center of $A_1$.  We claim that an element of the half Farey partition for $A_1$ is contained in the closed interval \[J:=\left[a_1 - \frac{\rho(A_1)} 2, a_1+ \frac{\rho(A_1)}2\right].\]  We now prove this claim.  If $p/q+k \in J$, we are done.  Otherwise, consider \[p/q+k \in A_1 \backslash J= \left[a_1 - \rho(A_1), a_1 - \frac{\rho(A_1)} 2\right) \cup \left(a_1 + \frac{\rho(A_1)} 2, a_1 + \rho(A_1), \right].\]  

Let us consider the case that \begin{align}\label{eqn:minimal Fareyleftcenter}
 a_1 - \rho(A_1) \leq p/q+k < a_1 - \rho(A_1)/2. \end{align}  Let $p/q+k$ and $u/v+k$ be consecutive elements of the half Farey partition for $A_1$.  Thus, we have $p/q+k < u/v+k$.  First consider the subcase that $a_1 - \rho(A_1)/2\leq u/v+k$.  If, furthermore, we have that $a_1 + \rho(A_1)/2< u/v+k$, then \[\frac u v + k - \left(\frac p q +k \right) > \rho(A_1)\geq \frac 1 {q^2},\] which contradicts Lemma~\ref{lemm:halfFareyPartitionSeparate}.  Consequently, we have that $u/v+k \in J,$ which yields the desired result for this subcase.  
 
Otherwise, we have the other subcase that $u/v+k < a_1 - \rho(A_1)/2$.  As mentioned, $p\neq 0$ for the case (\ref{eqn:RadA1Big}).  For $p \neq 0$, consider the following.  Since $p/q+k$ is the minimal-order Farey element of $A_1$, we have that \[\left[a_1 - \rho(A_1), \frac p q +k \right] \subset  \left[\frac a b+k, \frac p q +k\right] \textrm{ and } \left[ \frac p q +k, a_1 + \rho(A_1)\right] \subset \left[\frac p q +k, \frac c d+k\right].\]  Recall the definition of $r:= r(A_1)$ from (\ref{eqn:halfFareyPartitionPara}).  By Lemmas~\ref{lemm:halfFareyPartitionLength} and~\ref{lemm:Fareypair} and (\ref{eqn:minimal Fareyleftcenter}), we have \begin{align*}\frac {p+\lceil r \rceil c} {q+\lceil r \rceil d} +k  - \left(\frac p q + k \right) >  \frac 1{2dq} > \frac 1 2 \left(a_1+ \rho(A_1) - \left(\frac p q +k \right) \right) > \frac 3 4 \rho(A_1),
   \end{align*} which, using (\ref{eqn:minimal Fareyleftcenter}) again, implies that \[a_1 - \frac{\rho(A_1)}2 < \frac {p+\lceil r \rceil c} {q+\lceil r \rceil d} +k.\]  If \[\frac {p+\lceil r \rceil c} {q+\lceil r \rceil d} +k \leq a_1 + \frac{\rho(A_1)}2 ,\] also holds, then \[\frac {p+\lceil r \rceil c} {q+\lceil r \rceil d} +k\] is the desired element of the half Farey partition for $A_1$ contained in $J$.
   
 Otherwise, we have that \[a_1 + \frac{\rho(A_1)}2 <\frac {p+\lceil r \rceil c} {q+\lceil r \rceil d} +k.\] Thus, if there is no element of the half Farey partition for $A_1$ is contained in $J$, then $J$ lies between consecutive elements of the half Farey partition for $A_1$, which, by Lemma~\ref{lemm:halfFareyPartitionSeparate}, implies that $1/q^2>\diam(J) = \rho(A_1)$.  This contradicts (\ref{eqn:RadA1Big}) and yields the desired result in this subcase.  This completes the proof of the claim for the case that (\ref{eqn:minimal Fareyleftcenter}) holds.

Finally, we consider the case that $ a_1 + \rho(A_1)/2 < p/q+k \leq a_1 + \rho(A_1)$. The proof is analogous to that for the case that (\ref{eqn:minimal Fareyleftcenter}) holds.  This proves the claim.

 We now continue with the proof of the theorem in the case that (\ref{eqn:RadA1Big}) holds.  Let us first consider the subcase that  $\rho(B_2) < 2/q^2$.  By the claim, we have an element $p'/q'+k$ of the half Farey partition for $A_1$ such that $p'/q'+k \in J$.  Since $\beta \leq 1/2$, let $B_2$ have center $p'/q'+k$.  Consequently, since $\rho(A_2)= \rho(B_2)/2$, any valid choice of $A_2$ must contain $p'/q'+k$.  Bob now chooses $B_3$ to contain $p'/q'+k$, yielding \begin{align}\label{eqn:PreDirProof}
 \left | x - \frac{p'+kq'} {q'} \right | \leq 2 \rho(B_3) =  \beta \rho(B_2)  < \frac {2\beta} {q^2} \end{align} for all $x \in B_3$.  Since $p'/q'+k$ is an element of the half Farey partition for $A_1$, we have that \[\begin{cases} q' = q \textrm{ or }\\  q' = q+jb \textrm{ for some } 1 \leq j \leq \lceil l \rceil \textrm { or } \\  q' = q+ j d \textrm{ for some } 1 \leq j \leq \lceil r \rceil\end{cases}\] where $l:= l(A_1)$ and $r:=r(A_1)$ are defined in (\ref{eqn:halfFareyPartitionPara}).  Hence, we have that $q' < 3 q$ and that \begin{align}\label{eqn:MainDirProof}
 \left | x - \frac{p'+kq'} {q'} \right |< \frac {2\beta} {q^2}< \frac {18\beta} {(q')^2} \end{align} for all $x \in B_3$.  Now Alice chooses $A_3 \subset B_3$.
 
 Finally, let us consider the other subcase that $\rho(B_2) \geq 2/q^2$.  Hence, there exist an integer $N \geq 3$ such that $\rho(B_N) < 2/q^2 \leq \rho(B_{N-1})$.  Let $B_2 \subset \F(A_1)$ which Lemma~\ref{lemm:halfFareyPartitionCover} implies is a valid choice for Bob since $\beta \leq 1/2$.  Furthermore, Lemma~\ref{lemm:halfFareyPartitionCover} also implies that $A_{N-1} \subset B_{N-1} \subset \cdots \subset B_2$ are contained in the cover formed by the half Farey partition for $A_1$.  Let $a_{N-1}$ be the center of $A_{N-1}$ and form the interval \[K:= \left[ a_{N-1} - \frac{\rho(A_{N-1})} 2,  a_{N-1} + \frac{\rho(A_{N-1})} 2\right].\]  Since $\rho(A_{N-1}) \geq 1/q^2$ holds, Lemma~\ref{lemm:halfFareyPartitionSeparate} implies that $K$ contains an element $p'/q'+k$ of the half Farey partition for $A_1$.  Now Bob chooses $B_N$ to have center $p'/q'+k$, and Bob's choice is valid because $\beta \leq 1/2$.  Since $\rho(A_N) = \rho(B_N)/2$, any valid choice of $A_N$ must contain $p'/q'+k$.  Now Bob chooses $B_{N+1}$ to contain $p'/q'+k$, yielding \[\left | x - \frac{p'+kq'} {q'} \right | \leq 2 \rho(B_{N+1}) =  \beta \rho(B_N)  < \frac {2\beta} {q^2}\]for all $x \in B_{N+1}$.  Using the same proof that allows us to show (\ref{eqn:MainDirProof}) from (\ref{eqn:PreDirProof}), we obtain \begin{align}\label{eqn:MainDirProof2}
 \left | x - \frac{p'+kq'} {q'} \right |< \frac {18\beta} {(q')^2} \end{align} for all $x \in B_{N+1}$.  Now Alice chooses $A_{N+1} \subset B_{N+1}$.
 
 This completes one iteration $m=1$ of the recursion.  In all the cases, Alice makes the final choice of ball, denoted above as $A_2, A_3,$ or $A_{N+1}$ depending on the case. In the next step of the recursion, this final choice of Alice's ball is treated in the same way as $A_1$ in the above proof.  This yields the next step $m=2$ which we can recursively continue.
 
 Now, for each step $m$ of the recursion, we have a unique $\gamma$.  Let us denote this unique $\gamma$ at step $m$ of the recursion by $\gamma(m)$.  By Lemma~\ref{lemm:WeakIncOrderMinFarey}, $\eO(\gamma(m))$ increases weakly with $m$.  If $\eO(\gamma(m))$ is bound from above for all $m \in \NN$, then let $M \in \NN$ be such that $\eO(\gamma(M))$ is the least upper bound.  Lemma~\ref{lemm:WeakIncOrderMinFarey} further implies that $\gamma(m) = \gamma(M)$ for all $m \geq M$.  Consequently, we have \[\bigcap_{n=1} A_n = \bigcap_{n=1} B_n = \gamma(M) \in \QQ.\]

Otherwise, $\eO(\gamma(m))$ is not bounded from above.  For a step $m$ of the recursion, we have $\nu(m):=p/q+k$ from (\ref{eqn:MainDirProof0}) or $\nu(m):=p'/q'+k$ from either (\ref{eqn:MainDirProof}) or (\ref{eqn:MainDirProof2}).  Since possible $\nu(m)$ are all Farey elements, they are in lowest terms and, consequently, we have that $q' \geq \eO(\gamma(m))$ and $q = \eO(\gamma(m))$ because $\gamma(m)$ is the minimal-order Farey element.  Hence, Lemma~\ref{lemm:WeakIncOrderMinFarey} implies that the denominators of the $\nu(m)$ expressed in lowest terms are not bounded from above as $m$ increases, which yields that \[\bigcap_{n=1} A_n = \bigcap_{n=1} B_n \in \Dir_{\RR} (18 \beta)\] where $\Dir_{\RR}(\widetilde{\delta})$ is, for any $\widetilde{\delta}>0$, the following set \begin{align*}
&\left\{x \in \RR :\exists \textrm{ infinitely many } (s,t) \in \ZZ \times \NN \textrm{ such that } (s,t)=1 \textrm{ for which } \left|x - \frac s t \right| < \frac {\widetilde{\delta}}{t^2}\textrm{ holds} \right\}.\end{align*}  This shows that $\left(\Dir_{\RR} (18 \beta) \cup \QQ\right)^c$ is an $(1/2, \beta)$-ubiquitously losing set of $\RR$.
 
 Now it is well-known that $\Dir_{\RR}(\widetilde{\delta})=\Dir(\widetilde{\delta})$ because $\Dir_{\RR}(\widetilde{\delta})$ contains no rational numbers.  The proof is brief so we provide it.  If there is a rational number $x=s'/t'$ in $\Dir_{\RR}(\widetilde{\delta})$, then we have that \[ \frac 1 {tt'} \leq \left|x - \frac s t \right| < \frac {\widetilde{\delta}}{t^2}\] and, thus, $t < \widetilde{\delta} t'$ whenever $s/t \neq s'/t'$.  Since $t \rightarrow \infty$, we have a contradiction.  Consequently, using set algebra, DeMorgan's laws, and Proposition~\ref{prop:BADeltaComplement}, we have that $\BA(18 \beta)$ is an $(1/2, \beta)$-ubiquitously losing set of $\RR$.  This proves Theorem~\ref{theo:LosingBADelta}.

\section{Accelerated games}\label{sec:AcceleratedGames}

To compute an upper bound for the Hausdorff dimension in Section~\ref{sec:UpBndHDLosingSets}, we introduce a generalization of Schmidt games.  Let $d, n \in  \NN$, $0< \alpha<1$, $0< \beta <1$, $\{s_n\}:= \{s_n\}_{n = 1}^\infty \subset \NN \cup \{0\}$, and $S \subset \RR^d$.  We define the {\em $(\alpha, \beta^{\{s_n\}}; S)$-game (on $\RR^d$)} or the {\em $(\alpha, \beta^{\{s_n\}})$-game (on $\RR^d$)} to be a two-player game in which the two players, Bob and Alice, alternate choosing nested closed balls $B_n$ and $A_n$ at {\em move $n$} as follows.  To start the $(\alpha, \beta^{\{s_n\}} )$-game, Bob chooses $B_1$ and Alice chooses $A_1 \in B_1^\alpha$.  Then Bob chooses \[B_2 \in A_1^{\beta (\alpha \beta)^{s_1}}\] and Alice chooses $A_2 \in  B_2^\alpha$.  Recursively, Bob chooses \[B_n \in A_{n-1}^{\beta (\alpha \beta)^{s_{n-1}}}\] and Alice chooses $A_n \in  B_n^\alpha$.  {\em Alice wins} the game if $\bigcap_{n=1}^\infty B_n \in S$, and {\em Bob wins} the game if $\bigcap_{n=1}^\infty B_n \in S^c$.  The set $S$ is called an {\em $(\alpha, \beta^{\{s_n\}})$-winning set (of $\RR^d$)} if Alice can always win or, equivalently, as for Schmidt games, Alice has an {\em $(\alpha, \beta^{\{s_n\}}; S)$-winning strategy}.  On the other hand, the set $S$ is called an {\em $(\alpha, \beta^{\{s_n\}})$-losing set (of $\RR^d$)} if Bob can always win or, equivalently, as for Schmidt games, Bob has an {\em $(\alpha, \beta^{\{s_n\}}; S)$-winning strategy}.  Moreover, if, for every choice of closed ball $B_1$, Bob has a winning strategy for the $(\alpha, \beta^{\{s_n\}}; S)$-game, then $S$ is called an {\em $(\alpha, \beta^{\{s_n\}})$-ubiquitously losing set (of $\RR^d$)}.  Note that Remarks~\ref{rmk:UbiqLosSet1} and~\ref{rmk:UbiqLosSet2} also apply to $(\alpha, \beta^{\{s_n\}})$-ubiquitously losing sets.  For a given $\alpha, \beta,$ and $S$, we will also refer to the $(\alpha, \beta^{\{s_n\}} )$-game as an {\em $(\alpha, \beta; S)$-accelerated game for Bob} and $\{s_n\}$ as its {\em acceleration sequence}.\footnote{One could, analogously, define $(\alpha, \beta; S)$-accelerated games for Alice, but we will not use such games in this paper.}  Finally, note that some of the accelerated games for Bob are the same as the usual Schmidt games.  For example, the $(\alpha, \beta^{\{0\}};S)$-game is the same as the $(\alpha, \beta; S)$-game, and, for every $s \in \NN$, the $(\alpha, \beta^{\{s\}}; S)$-game is the same as the $(\alpha, \beta(\alpha \beta)^{s}; S)$-game.  We call an $(\alpha, \beta; S)$-accelerated game for Bob {\em proper} if it is not the $(\alpha, \beta; S)$-game.

The following lemma is the generalization of the analog of~\cite[Lemma~9]{Sch66} with essentially the same proof, given below for the convenience of the reader.

\begin{lemm}\label{lemm:UbitLosImplies} Let $d, n \in  \NN$, $0< \alpha<1$, and $0< \beta <1$.  Then every $(\alpha, \beta)$-ubiquitously losing set of $\RR^d$ is an $(\alpha,  \beta^{\{s_n\}})$-ubiquitously losing set of $\RR^d$ for every $\{s_n\} \subset \NN \cup \{0\}$.
 
\end{lemm}
\begin{proof}  Let $m \in \NN$ and $S \subset \RR^d$ be an $(\alpha, \beta)$-ubiquitously losing set of $\RR^d$.  Suppose in an $(\alpha, \beta; S)$-game, Bob chooses his own balls $B_m$ and those balls $A_m$ of Alice except where \[m = 1 \quad \textrm{ and } \quad m=1+ \sum_{n=1}^{k} (s_n+1)\] for all $k \in \NN$.  Consequently, Alice picks the first ball and every $ \sum_{n=1}^{k} (s_n+1)$-st ball, namely the balls \[A_1 \supset A_{1+(s_1+1)} \supset A_{1+(s_1+1)+(s_2+1)} \cdots \supset A_{1+\sum_{n=1}^{k} (s_n+1)} \supset \cdots.\]  The balls \[B_1 \supset A_1 \supset B_{1+(s_1+1)} \supset A_{1+(s_1+1)}  \supset \cdots \supset B_{1+\sum_{n=1}^{k} (s_n+1)} \supset A_{1+\sum_{n=1}^{k} (s_n+1)} \supset \cdots \] are the balls of an $(\alpha,  \beta^{\{s_n\}})$-game.  Since Bob can win the $(\alpha, \beta; S)$-game, he can also win the $(\alpha,  \beta^{\{s_n\}}; S)$-game.  This gives the desired result.
 
\end{proof}

\begin{rema}   Let $d, n \in  \NN$, $0< \alpha<1$, and $0< \beta <1$.  The proof of Lemma~\ref{lemm:UbitLosImplies} also shows that every $(\alpha, \beta)$-losing set of $\RR^d$ is an $(\alpha,  \beta^{\{s_n\}})$-losing set of $\RR^d$ for every $\{s_n\} \subset \NN \cup \{0\}$.

\end{rema}

\begin{rema}\label{rmk:UbitLosImplies2}   Let $d,k,n \in  \NN$, $0< \alpha<1$, $0< \beta <1$, and $S \subset \RR^d$.  Every $(\alpha, \beta; S)$-winning strategy for Bob yields, for every $\{s_n\}:= \{s_n\}_{n = 1}^\infty \subset \NN \cup \{0\}$, an $(\alpha, \beta^{\{s_n\}}; S)$-winning strategy for Bob.  This follows because at every move $k \geq 2$ for the $(\alpha, \beta^{\{s_n\}}; S)$-game, Bob uses his $(\alpha, \beta; S)$-winning strategy for the move $1+ \sum_{n=1}^{k-1} (s_n+1)$.  In this paper, all the $(\alpha, \beta^{\{s_n\}}; S)$-winning strategies for Bob comes from, in this way, an $(\alpha, \beta; S)$-winning strategy for Bob, and we will (usually) omit explicit mention of this.

\end{rema}

Let $\{s_n\} \subset \NN \cup \{0\}$ and fix an $(\alpha, \beta^{\{s_n\}} )$-game.  A sequence of closed balls $B_1 \supset B_2 \supset B_3 \supset \cdots$ is a {\em play} for the $(\alpha, \beta^{\{s_n\}} )$-game if Bob chooses each $B_n$ according to the rules of the $(\alpha, \beta^{\{s_n\}} )$-game.  Its associated {\em dyadic play} is the sequence \[B_1 \supset A_1 \supset B_2 \supset A_2 \supset B_3 \supset A_3 \supset \cdots\] where the $A_n$ are the closed balls chosen by Alice according to the rules of the $(\alpha, \beta^{\{s_n\}} )$-game.  For $k \in \NN$, we call $B_1 \supset B_{2} \supset \cdots \supset B_{k}$ a {\em finite play (or $k$-finite play)} for the $(\alpha, \beta^{\{s_n\}} )$-game if there exist closed balls $B_{k+1} \supset B_{k+2} \supset \cdots$ such that \[B_1 \supset B_2 \supset \cdots \supset B_k \supset B_{k+1} \supset B_{k+2} \supset \cdots\] is a play for the $(\alpha, \beta^{\{s_n\}} )$-game.  Its associated {\em dyadic finite play} is the sequence \[B_1 \supset A_1 \supset B_2 \supset A_2 \supset \cdots  \supset B_k \supset A_k\] where the $A_n$ are the closed balls chosen by Alice according to the rules of the $(\alpha, \beta^{\{s_n\}} )$-game.  The ball $B_k$ is referred to as the {\em end ball} of the finite play $B_1 \supset B_{2} \supset \cdots \supset B_{k}$.  The collection of finite plays \[\{B_1, B_1 \supset B_2, B_1 \supset B_2 \supset B_3, \cdots\}\] together with the play $B_1 \supset B_2 \supset B_3 \supset \cdots$ will be referred to as a {\em game} or {\em game played}.  Likewise, the collection of dyadic finite plays \[\{B_1 \supset A_1, B_1 \supset A_1\supset B_2\supset A_2, B_1\supset A_1 \supset B_2 \supset A_2\supset B_3\supset A_3, \cdots\}\] together with the dyadic play \[B_1\supset A_1 \supset B_2\supset A_2 \supset B_3\supset A_3 \supset \cdots\] will be also referred to as a {\em game} or {\em game played}.   We say a play $B_1 \supset B_2 \supset B_3 \supset \cdots $ is {\em winning for Bob} if, for all $n \geq 2$, the $B_n$ are chosen according to an $(\alpha, \beta; S)$-winning strategy for Bob with initial ball $B_1$.\footnote{On~\cite[Page~179]{Sch66}, there is the notion of a chain for a winning strategy.  The analogous notion of a chain for a winning strategy for Bob can also be defined.  The notion of play is different from that of chain.}  Similarly, for integers $k \geq 2$, a $k$-finite play $B_1 \supset B_{2} \supset \cdots \supset B_{k}$ is {\em winning for Bob} if, for all integers $n$ such that $k \geq n \geq 2$, the $B_n$ are chosen according to an $(\alpha, \beta; S)$-winning strategy for Bob with initial ball $B_1$.  Consequently, if $B_1 \supset B_2 \supset B_3 \supset \cdots $ is a winning play for Bob, then $\bigcap_{n=1}^\infty B_n \in S^c$ holds.  Also, note that, for all integers $k \geq 2$, any $k$-finite play $B_1 \supset B_{2} \supset \cdots \supset B_{k}$ that is winning for Bob can be completed into a play that is winning for Bob by Bob choosing, recursively, successive $B_n$ for $n \geq k+1$ according to the $(\alpha, \beta; S)$-winning strategy for Bob with initial ball $B_1$.  Moreover, if $S$ is an $(\alpha, \beta)$-ubiquitously losing set, then we say any $1$-finite play $B_1$ is {\em winning for Bob}, and it can be completed into a play that is winning for Bob by Bob choosing, recursively, successive $B_n$ for $n \geq 2$ according to an $(\alpha, \beta; S)$-winning strategy for Bob with initial ball $B_1$.

Finally, a subsequence of a play $B_1 \supset B_2 \supset B_3 \supset \cdots$ for the $(\alpha, \beta^{\{s_n\}} )$-game is referred to as its {\em restriction}, and the collection of restrictions of a play includes the play itself.  

\begin{rema}\label{rmk:RestrictionPlay}
Note that a restriction of a play is a play for an $(\alpha, \beta; S)$-accelerated game for Bob and its the acceleration sequence is determined by the restriction.  Also note that, by the analog of the proof of Lemma~\ref{lemm:UbitLosImplies}, we have that any restriction of a play that is winning for Bob is itself a play that is winning for Bob.
\end{rema}

\subsection{Induction and adapted winning strategies for Bob}  In this section, we show how to induce finite plays that are winning for Bob from a certain class of finite plays that are winning for Bob.  This certain class of finite plays that are winning for Bob comes from the restrictions that are defined in Remark~\ref{rmk:DefnRestriction}.  The induced finite plays that are winning for Bob will be chosen according to an adapted winning strategy for Bob (defined in Remark~\ref{rmk:DefnAdaptedWinningStrat}).  Let $d \in  \NN$, $0< \alpha<1$, $0< \beta <1$, and $S \subset \RR^d$ be an $(\alpha, \beta)$-ubiquitously losing set of $\RR^d$.  First let $B_1 \supset B_{2} \supset  B_3  \supset \cdots$ be a play that is winning for Bob for an $(\alpha, \beta; S)$-game, and let \[B_1 \supset A_1 \supset B_2 \supset A_2 \supset B_3 \supset A_3  \supset  \cdots\] be its associated dyadic play.  Recall that all the winning strategies in this paper are positional (Remark~\ref{rem:PositionalWinning}).  Since the play is winning for Bob, the $B_n$, for $n \geq 2$, are chosen according to $g_{B_1}$, which is an $(\alpha, \beta; S)$-winning strategy for Bob with initial ball $B_1$.  Note that $g_{B_1}$ is defined analogously to that of a winning strategy for Alice (\cite[Page~179]{Sch66}).  Precisely, we have the following.  Let \[\Omega(B_1):= \{A \subset B_1 : A \textrm{ is a closed ball}\}.\]  The function $g_{B_1}(A)$, defined for elements $A \in \Omega(B_1)$ such that $g_{B_1}(A) \in A^\beta$, is an {\em $(\alpha, \beta; S)$-winning strategy for Bob with initial ball $B_1$} if the following holds:  for elements $B_1, B_2, \cdots , A_1, A_2, \cdots$ of $\Omega(B_1)$ such that \begin{align*} \begin{cases}  A_n \in B_n^\alpha & \textrm{ for } n \in \NN \\
B_{n+1} = g_{B_1}(A_n) \in A_n^{\beta} & \textrm{ for } n \in \NN  \end{cases},
  \end{align*} we have that $\bigcap_{n=1}^\infty B_n \in S^c$.

Next, let $B_1 \supset B_{2} \supset  B_3  \supset \cdots$ be a play for an $(\alpha, \beta; S)$-accelerated game for Bob with acceleration sequence $\{s_n\}:= \{s_n\}_{n = 1}^\infty \subset \NN \cup \{0\}$, and let \[B_1 \supset A_1 \supset B_2 \supset A_2 \supset B_3 \supset A_3  \supset  \cdots\] be its associated dyadic play.  Then, we say that the sequence of functions $g_{B_1, 1}, g_{B_1, 2}, \cdots$, defined for elements $A \in \Omega(B_1)$ such that $g_{B_1, n}(A) \in A^{\beta(\alpha \beta)^{s_n}}$, is an {\em $(\alpha, \beta^{\{s_n\}}; S)$-winning strategy for Bob with initial ball $B_1$} if the following holds:  for elements $B_1, B_2, \cdots , A_1, A_2, \cdots$ of $\Omega(B_1)$ such that \begin{align}\label{eqn:RestrictedStrat} \begin{cases}  A_n \in B_n^\alpha & \textrm{ for } n \in \NN \\ B_{n+1} = g_{B_1, n}(A_n) \in A_n^{\beta (\alpha \beta)^{s_n}} & \textrm{ for } n \in \NN  \end{cases},
  \end{align} we have that $\bigcap_{n=1}^\infty B_n \in S^c$.

 \begin{lemm}\label{lemm:RestrictedStrat}  The strategy $g_{B_1}$ for the $(\alpha, \beta; S)$-game yields a family of $(\alpha, \beta^{\{s_n\}}; S)$-winning strategies for Bob all with initial balls $B_1$.
 
\end{lemm}

  \begin{proof}  For an $(\alpha, \beta^{\{s_n\}}; S)$-game, let Alice's moves $A_1, A_2, \cdots \in \Omega(B_1)$ and Bob's moves $B_2, B_3, \cdots \in \Omega(B_1)$ be made as follows: \begin{align}\label{eqn:RestrictedStrat2} \begin{cases}  A_n \in B_n^\alpha & \textrm{ for } n \in \NN \\
  \widetilde{A}_n \in A_n^{(\alpha \beta)^{s_n}} & \textrm{ for } n \in \NN\\
B_{n+1} = g_{B_1}(\widetilde{A}_n) \in \widetilde{A}_n^{\beta} & \textrm{ for } n \in \NN  \end{cases}.
  \end{align}  Since $g_{B_1}$ is an $(\alpha, \beta; S)$-winning strategy for Bob with initial ball $B_1$, then regarding, in the same way as in the proof of Lemma~\ref{lemm:UbitLosImplies}, the $(\alpha, \beta^{\{s_n\}}; S)$-game being played as an $(\alpha, \beta; S)$-game yields that $\bigcap_{n=1}^\infty B_n \in S^c$. Hence, for a choice of $\widetilde{A}_n \in A_n^{(\alpha \beta)^{s_n}}$, define $g_{B_1, n}(A_n):= g_{B_1}(\widetilde{A}_n)$.  Since $\bigcap_{n=1}^\infty B_n \in S^c$ holds, we have that $g_{B_1, 1}, g_{B_1, 2}, \cdots$ is an $(\alpha, \beta^{\{s_n\}}; S)$-winning strategy for Bob with initial ball $B_1$. \end{proof}
\begin{rema}\label{rmk:DefnRestriction} An $(\alpha, \beta^{\{s_n\}}; S)$-winning strategy for Bob with initial ball $B_1$ constructed as in Lemma~\ref{lemm:RestrictedStrat} is called an {\em $\{s_n\}$-restriction (of $g_{B_1}$)} or, more briefly, a {\em restriction (of $g_{B_1}$)}.  Usually, we denote an $\{s_n\}$-restriction by (\ref{eqn:RestrictedStrat2}) rather than (\ref{eqn:RestrictedStrat}).  Also note that a play or finite play that is winning for Bob is chosen according to a restriction of an $(\alpha, \beta;S)$-winning strategy for Bob.

\end{rema}

Now let $k \in \NN$, let $B_1 \supset B_{2} \supset \cdots \supset B_{k}$ be a $k$-finite play for an $(\alpha, \beta; S)$-accelerated game for Bob with acceleration sequence $\{s_n\}:= \{s_n\}_{n = 1}^\infty \subset \NN \cup \{0\}$, and let \[B_1 \supset A_1 \supset B_2 \supset A_2 \supset \cdots  \supset B_k \supset A_k\] be its associated dyadic finite play. We say that a closed ball $B$ such that \begin{align}\label{eqn:InsertionCond}
 \rho(B) = \rho(B_1) (\alpha \beta)^\ell \end{align} holds for some $\ell \in \NN$ is {\em insertable for the $k$-finite play $B_1 \supset B_{2} \supset \cdots \supset B_{k}$} if there exists an integer $m$ such that $ 1 \leq m < k$ for which \begin{align} \label{eqn:InsertionCond2}
 B_m \supsetneq B \supsetneq B_{m+1} \end{align} holds and is {\em appendable for the $k$-finite play $B_1 \supset B_{2} \supset \cdots \supset B_{k}$} if \begin{align}\label{eqn:AppendCond}
 B_k \supsetneq B \end{align} holds.  The integer $m$ is called the {\em insertion index (of $B$) (into $B_1 \supset B_{2} \supset \cdots \supset B_{k}$)}.  Likewise, if the $k$-finite play is replaced by a play $B_1 \supset B_{2} \supset  B_3  \supset \cdots$ in the above, then we have the analogous definition of insertable, which precisely is the following:  a closed ball $B$ such that (\ref{eqn:InsertionCond}) holds for some $\ell \in \NN$ is {\em insertable for the play $B_1 \supset B_{2} \supset  B_3  \supset \cdots$} if there exists an integer $m$ such that (\ref{eqn:InsertionCond2}) holds.  The integer $m$ is called the {\em insertion index (of $B$) (into $B_1 \supset B_{2} \supset  B_3  \supset \cdots$)}. 
  
 First consider the case for which $B$ is insertable for $B_1 \supset B_{2} \supset \cdots \supset B_{k}$.  Let \[\widetilde{B}_1 := B_1, \cdots, \widetilde{B}_m := B_m, \widetilde{B}_{m+1} := B, \widetilde{B}_{m+2} := B_{m+1}, \cdots, \widetilde{B}_{k+1} := B_k.\]  By (\ref{eqn:InsertionCond}, \ref{eqn:InsertionCond2}), we have that \begin{align}\label{eqn:InsertExpMain}
\begin{cases}  \sum_{i=1}^{m-1} (s_i+1) < \ell <  \sum_{i=1}^{m} (s_i+1)  & \text{ for } k>m \geq 2 \\
0< \ell < s_1 +1 & \text{ for } m = 1\end{cases} \end{align}  Consequently, $s_m \geq 1$.  Now consider the two cases in (\ref{eqn:InsertExpMain}) separately.  For the case $k>m \geq 2$, let \begin{align}\label{eqn:InsertExp1} \widetilde{s}_{j} :=
 \begin{cases}  s_j & \text{ for } j = 1, \cdots, m-1 \\  \ell -  \sum_{i=1}^{j-1} (s_i+1) - 1 & \text{ for } j = m \\  \sum_{i=1}^{j-1} (s_i+1)- \ell  - 1  & \text{ for } j = m+1 \\  s_{j-1} & \text{ for } j \geq m+2  \end{cases}.\end{align}  Observe that $\widetilde{s}_{j} \in \NN \cup \{0\}$ for all $j \in \NN$ and, thus, $\{\widetilde{s}_{j}\}$ is an acceleration sequence.  We call $\{\widetilde{s}_{j}\}$ the {\em induced} (or {\em $B$-induced}) acceleration sequence.  Note that $\{\widetilde{s}_{j}\}$ is obtained from $\{{s}_{j}\}$ by removing the $m$-th element and replacing it with the two successive elements $\widetilde{s}_{m}$ and $\widetilde{s}_{m+1}$.  Also observe that \begin{align}\label{eqn:InsertExp2}   \sum_{i=1}^{j} (\widetilde{s}_i+1) = \begin{cases}  \sum_{i=1}^{j} (s_i+1)  & \text{ for } j = 1, \cdots, m-1 \\
  \ell & \text{ for } j = m \\   \sum_{i=1}^{j-1} (s_i+1)  & \text{ for } j \geq m+1\end{cases} .
  \end{align}

Now it follows from (\ref{eqn:InsertExp2}) that \begin{align*} 
 \rho(\widetilde{B}_{m+1}) = \rho(\widetilde{B}_1) (\alpha \beta)^{\sum_{i=1}^{m} (\widetilde{s}_i+1)} =  (\alpha \beta)^{(\widetilde{s}_m+1)} \rho(\widetilde{B}_{m}).  \end{align*} Thus, there exists a closed ball $\widetilde{A}_{m}$ such that $\widetilde{A}_{m} \in \widetilde{B}^\alpha_{m}$ and $\widetilde{A}_{m} \supset \widetilde{B}_{m+1}$.  Likewise, we have that $\rho(\widetilde{B}_{m+2})  =  (\alpha \beta)^{(\widetilde{s}_{m+1}+1)} \rho(\widetilde{B}_{m+1})$.  Thus, there exists a closed ball $\widetilde{A}_{m+1}$ such that $\widetilde{A}_{m+1} \in \widetilde{B}^\alpha_{m+1}$ and $\widetilde{A}_{m+1} \supset \widetilde{B}_{m+2}$.  Now let $\widetilde{A}_1 := A_1, \cdots, \widetilde{A}_{m-1} := A_{m-1}$, $\widetilde{A}_{m+2} := A_{m+1}$, $\widetilde{A}_{m+3} := A_{m+2}, \cdots, \widetilde{A}_{k+1} := A_{k}$.  Hence, for the case $k>m \geq 2$, we call $\widetilde{B}_1 \supset \widetilde{B}_{2} \supset \cdots \supset \widetilde{B}_{k+1}$ the {\em induced} (or {\em $B$-induced}) ($(k+1)$-finite) play and \[\widetilde{B}_1 \supset \widetilde{A}_1 \supset \widetilde{B}_{2} \supset  \widetilde{A}_2 \supset \cdots \supset \widetilde{B}_{k+1} \supset \widetilde{A}_{k+1}\] an {\em associated dyadic (finite) play}.  Likewise, for $B$ insertable into a play $B_1 \supset B_{2} \supset  B_3  \supset \cdots$, we obtain the {\em induced} (or {\em $B$-induced}) play and its {\em associated dyadic play} in the analogous way.

\begin{rema}\label{rmk:InsertPlay}  Note that, while an associated dyadic play is not unique, it will be immaterial for this paper.   Also note that the induced play is a play for an $(\alpha, \beta^{\{\widetilde{s}_j\}}; S)$-game.
\end{rema}

We now consider the other case in (\ref{eqn:InsertExpMain}), namely the case $m=1$.  The analog of (\ref{eqn:InsertExp1}) is \begin{align}\label{eqn:InsertExp1MIsOne} \widetilde{s}_{j} :=
 \begin{cases}    \ell -  1 & \text{ for } j = 1 \\  s_1- \ell  & \text{ for } j = 2 \\   s_{j-1} & \text{ for } j \geq 3  \end{cases}.\end{align}  Similar to the previous case, we call $\{\widetilde{s}_{j}\}$ the {\em induced} (or {\em $B$-induced}) acceleration sequence.  The analog of (\ref{eqn:InsertExp2}) is \begin{align}\label{eqn:InsertExp2MIsOne} \sum_{i=1}^{j} (\widetilde{s}_i+1) = \begin{cases}  \ell & \text{ for } j = 1 \\   \sum_{i=1}^{j-1} (s_i+1)  & \text{ for } j \geq 2\end{cases} .
  \end{align}  Analogous to the previous case, we have that there exists a closed ball $\widetilde{A}_{1}$ such that $\widetilde{A}_{1} \in \widetilde{B}^\alpha_{1}$ and $\widetilde{A}_{1} \supset \widetilde{B}_{2}$ and that there exists a closed ball $\widetilde{A}_{2}$ such that $\widetilde{A}_{2} \in \widetilde{B}^\alpha_{2}$ and $\widetilde{A}_{2} \supset \widetilde{B}_{3}$.    Now let $\widetilde{A}_{3} := A_{2}$, $\widetilde{A}_{4} := A_{3}, \cdots, \widetilde{A}_{k+1} := A_{k}$.  Hence, for the case $m=1$, we call $\widetilde{B}_1 \supset \widetilde{B}_{2} \supset \cdots \supset \widetilde{B}_{k+1}$ the {\em induced} (or {\em $B$-induced}) ($(k+1)$-finite) play and \[\widetilde{B}_1 \supset \widetilde{A}_1 \supset \widetilde{B}_{2} \supset  \widetilde{A}_2 \supset \cdots \supset \widetilde{B}_{k+1} \supset \widetilde{A}_{k+1}\] an {\em associated dyadic (finite) play}.  Likewise, for $B$ insertable into a play $B_1 \supset B_{2} \supset  B_3  \supset \cdots$, we obtain the {\em induced} (or {\em $B$-induced}) play and its {\em associated dyadic play} in the analogous way.  Remark~\ref{rmk:InsertPlay} also applies for the case $m=1$.
  
\begin{rema}  Note that $B_1 = \widetilde{B}_1$ in all cases.

\end{rema}
  
\begin{lemm}\label{lemm:InducedWinningPlayInsert} Let $k \in  \NN$, $B_1 \supset B_{2} \supset \cdots \supset B_{k}$ be a $k$-finite play that is winning for Bob, and $B$ be closed ball that is insertable for $B_1 \supset B_{2} \supset \cdots \supset B_{k}$.  Then the $B$-induced play $\widetilde{B}_1 \supset \widetilde{B}_{2} \supset \cdots \supset \widetilde{B}_{k+1}$ is winning for Bob.

\end{lemm}
\begin{proof}  Since $B_1 \supset B_{2} \supset \cdots \supset B_{k}$ is winning for Bob, the $B_n$ are chosen according to $g_{B_1}$, an $(\alpha, \beta; S)$-winning strategy for Bob (with initial ball $B_1$).  Let $m$ be the insertion index of $B$.  Define the sets \[\Am := \{A \in  \widetilde{B}_m^{\alpha (\alpha \beta)^{\widetilde{s}_m}}: B \subset A\} \textrm{ and } \Am' := \{A \in  B^{\alpha (\alpha \beta)^{\widetilde{s}_{m+1}}}: B_{m+1} \subset A\}.\]  Note that $\Am \cap \Am' = \emptyset$. Define a function \begin{align*} \widetilde{g}_{B_1}( A) := \begin{cases}  B & \text{ if } A  \in \Am\\
  B_{m+1}  & \text{ if } A \in \Am' \\
g_{B_1}( A) & \text{ if } A \in \Omega(B_1) \backslash \left( \Am \cup \Am'\right)
 \end{cases}.
  \end{align*}  We claim that $\widetilde{g}_{B_1}$ is an $(\alpha, \beta; S)$-winning strategy for Bob with initial ball $B_1$.  We now prove the claim.  First note that, by (\ref{eqn:InsertExp2}, \ref{eqn:InsertExp2MIsOne}),  $\widetilde{g}_{B_1}(A) \in A^\beta$ for all $A \in  \Am \cup \Am'$ and, thus, for all $A \in \Omega(B_1)$.  Now let  $B'_1 := B_1$, let $B'_1 \supset B'_{2} \supset \cdots$ be a play for an $(\alpha, \beta; S)$-game, and let \[B'_1 \supset A'_1 \supset  B'_{2} \supset A'_2 \supset \cdots\] be its associated dyadic play for which the $A'_1, A'_2, A'_3, \cdots$ and $B'_2, B'_3, B'_4, \cdots $ are chosen by  \begin{align*} \begin{cases}  A'_n \in (B'_n)^\alpha & \textrm{ for } n \in \NN \\
B'_{n+1} = \widetilde{g}_{B_1}(A'_n) \in (A'_n)^{\beta} & \textrm{ for } n \in \NN  \end{cases}.
  \end{align*}  There are two cases to consider.
  
 \begin{enumerate}
\item There exists an $n \in \NN$ for which $A'_n \in \Am$.

\item There does not exist an $n \in \NN$ for which $A'_n \in \Am$.
\end{enumerate}

Let us consider Case~(1).  Since every element of $\Am$ has the same radius, we have that $n$ is unique.  The condition of Case~(1) gives that $B'_{n+1} = B$.  There are two subcases to consider for Case~(1).
\begin{enumerate}
\item [(1a)] There exists an $N \in \NN$ for which $A'_N \in \Am'$.
\item [(1b)] There does not exist an $N \in \NN$ for which $A'_N \in \Am'$.
\end{enumerate}
Let us consider Case~(1a) first.  Since every element of $\Am'$ has the same radius, we have that $N$ is unique.  The condition of Case~(1a) gives that $B'_{N+1} = B_{m+1}$.  Let $M \in \NN$ be such that $M > N$.  As $\rho(A'_{M}) < \rho (A'_N)$, we have that $A'_{M} \in \Omega(B_1) \backslash \left( \Am \cup \Am'\right)$.  Consequently, we have $\widetilde{g}_{B_1}( A'_{M}) = g_{B_1}( A'_{M})$ for all $M >N$.  Since $g_{B_1}$ is an $(\alpha, \beta; S)$-winning strategy for Bob with initial ball $B_1$, we have that $\bigcap_{i=N+1}^\infty B'_i \in S^c$ and, thus,\begin{align}\label{eqn:lemm:InducedWinningPlay}
  \bigcap_{i=1}^\infty B'_i \in S^c. \end{align}  This completes Case~(1a).

Next, let us consider Case~(1b).  We have that \[\{A'_P\}_{P=n+1}^\infty \subset  \Omega(B_1) \backslash \left( \Am \cup \Am'\right)\] and, thus, that $\widetilde{g}_{B_1}( A'_{P}) = g_{B_1}( A'_{P})$ for all $P >n$, yielding (\ref{eqn:lemm:InducedWinningPlay}).  This completes Case~(1b) and, thus, Case~(1).

Now let us consider Case~(2). There are two subcases to consider for Case~(2).
\begin{enumerate}
\item [(2a)] There exists an $N \in \NN$ for which $A'_N \in \Am'$.
\item [(2b)] There does not exist an $N \in \NN$ for which $A'_N \in \Am'$.
\end{enumerate}
The proof of Case~(2a) is analogous to that of Case~(1a) and, thus, yields (\ref{eqn:lemm:InducedWinningPlay}).  The proof of Case~(2b) is analogous to that of Case~(1b) with $n$ is replaced by $0$, yielding (\ref{eqn:lemm:InducedWinningPlay}).  This completes Case~(2) and the claim.
  
We now claim that the $\widetilde{B}_n$, for $n =1, \cdots, k+1$, are chosen according to $\widetilde{g}_{B_1}$ (as an $\{\widetilde{s}_j\}$-restriction of $\widetilde{g}_{B_1}$ as noted in Remark~\ref{rmk:DefnRestriction}).  We now prove the claim.  Let \[\widetilde{B}_1 \supset \widetilde{A}_1\supset \widetilde{B}_{2}\supset \widetilde{A}_2 \supset \cdots \supset \widetilde{B}_{k+1}\supset \widetilde{A}_{k+1}\] be the associated dyadic finite play of the $B$-induced play.  Consider the case $m \geq 2$ first.  Since $\widetilde{B}_1 = B_1, \cdots, \widetilde{B}_m = B_m$ and any closed ball $A \supset B_m$ has radius strictly larger than the radii of the elements of $\Am \cup \Am'$, we have that $\widetilde{B}_2, \cdots, \widetilde{B}_m$ are a chosen according to $g_{B_1}$, which, for these $A$, coincides with $\widetilde{g}_{B_1}$.  Note that $\widetilde{A}_m \in  \widetilde{B}_m^\alpha$ such that $B \subset \widetilde{A}_m$.  Now, for any choice $A \in \widetilde{B}_m^\alpha$ such that $B \subset A$, we have that \[\{A^* \in A^{(\alpha \beta)^{\widetilde{s}_m}}: B \subset A^*\} \subset \Am\] and, thus, $\widetilde{B}_{m+1} = B$ is chosen according to $\widetilde{g}_{B_1}$.  Likewise, we have $\widetilde{B}_{m+2}$ is chosen according to $\widetilde{g}_{B_1}$.  Now since $\widetilde{B}_{m+2} = B_{m+1}, \cdots, \widetilde{B}_{k+1} = B_k$ and any closed ball $A \subset \widetilde{B}_{m+2}$ has radius strictly smaller than the radii of the elements of $\Am \cup \Am'$, we have that $\widetilde{B}_{m+3}= B_{m+2}, \cdots, \widetilde{B}_{k+1}= B_{k}$ are a chosen according to $g_{B_1}$, which, for these $A$, coincides with $\widetilde{g}_{B_1}$.  This shows the claim for $m \geq 2$ and, thus, that the $B$-induced play is winning for Bob for $m \geq 2$.

For $m=1$, the analog of the proof for the $m \geq 2$, starting with note that $\widetilde{A}_m \in  \widetilde{B}_m^\alpha$ such that $B \subset \widetilde{A}_m$, shows the claim for $m=1$ and, thus, that the $B$-induced play is winning for Bob for $m =1$.   This yields the desired result and completes the proof of the lemma.  \end{proof}

\begin{lemm}\label{lemm:InducedWinningPlayInsert2} Let $B_1 \supset B_{2} \supset  B_{3}\supset \cdots $ be a play that is winning for Bob and $B$ be closed ball that is insertable for $B_1 \supset B_{2} \supset  B_{3}\supset \cdots $.  Then the $B$-induced play $\widetilde{B}_1 \supset \widetilde{B}_{2} \supset \widetilde{B}_{3} \supset \cdots$ is winning for Bob.
 
\end{lemm}

\begin{proof}
 Let $m$ be the insertion index of $B$.  The proof of this lemma is identical to that of Lemma~\ref{lemm:InducedWinningPlayInsert}, except that we have $\widetilde{B}_{m+3}= B_{m+2}, \cdots $ are a chosen according to $g_{B_1}$, which, for these $A$, coincides with $\widetilde{g}_{B_1}$.  This proves the lemma.
\end{proof}

 Finally, we consider the case for which $B$ is appendable.   Let \[\widetilde{B}_1 := B_1, \cdots, \widetilde{B}_k := B_k,  \widetilde{B}_{k+1} := B\] and let $\{s^*_n\}:= \{s^*_n\}_{n=1}^\infty$ be any acceleration sequence.  Let us first consider the case $k \geq 2$.  By (\ref{eqn:InsertionCond}, \ref{eqn:AppendCond}), we have that \begin{align}\label{eqn:AppendExpMain}
 \sum_{i=1}^{k-1} (s_i+1) < \ell.\end{align}  Now let \begin{align}\label{eqn:AppendExp1} \widetilde{s}_{j} :=
 \begin{cases}  s_j & \text{ for } j = 1, \cdots, k-1 \\  \ell -  \sum_{i=1}^{j-1} (s_i+1) - 1 & \text{ for } j = k  \\  s^*_{j-k} & \text{ for } j \geq k+1  \end{cases}.\end{align}  Observe that $\widetilde{s}_{j} \in \NN \cup \{0\}$ for all $j \in \NN$ and, thus, $\{\widetilde{s}_{j}\}$ is an acceleration sequence.  We call $\{\widetilde{s}_{j}\}$ the {\em induced} (or {\em $(B, \{s^*_n\})$-induced}) acceleration sequence.  Observe that \begin{align}\label{eqn:AppendExp2}   \sum_{i=1}^{j} (\widetilde{s}_i+1) = \begin{cases}  \sum_{i=1}^{j} (s_i+1)  & \text{ for } j = 1, \cdots, k-1 \\
  \ell & \text{ for } j = k \\   \ell + \sum_{i=k+1}^{j} (s^*_{i-k}+1)  & \text{ for } j \geq k+1\end{cases} .
  \end{align}  
  
  Now it follows from (\ref{eqn:AppendExp2}) that \begin{align*} 
 \rho(\widetilde{B}_{k+1}) = \rho(\widetilde{B}_1) (\alpha \beta)^{\sum_{i=1}^{k} (\widetilde{s}_i+1)} =  (\alpha \beta)^{(\widetilde{s}_k+1)} \rho(\widetilde{B}_{k}).  \end{align*} Thus, there exists a closed ball $\widetilde{A}_{k}$ such that $\widetilde{A}_{k} \in \widetilde{B}^\alpha_{k}$ and $\widetilde{A}_{k} \supset \widetilde{B}_{k+1}$.  Now let $\widetilde{A}_1 := A_1, \cdots, \widetilde{A}_{k-1} := A_{k-1}$ and let $\widetilde{A}_{k+1} \in \widetilde{B}^\alpha_{k+1}$.  We call $\widetilde{B}_1 \supset \widetilde{B}_{2} \supset \cdots \supset \widetilde{B}_{k+1}$ the {\em induced} ({\em $B$-induced}, or {\em $(B, \{s^*_n\})$-induced}) ($(k+1)$-finite) play and \[\widetilde{B}_1 \supset \widetilde{A}_1 \supset \widetilde{B}_{2} \supset  \widetilde{A}_2 \supset \cdots \supset \widetilde{B}_{k+1} \supset \widetilde{A}_{k+1}\] an {\em associated dyadic finite play}.  Remark~\ref{rmk:InsertPlay} also applies for the case for which $B$ is appendable and $k \geq2$.
 
 Now consider the case $k=1$.  Let \begin{align}\label{eqn:AppendExp1KOne} \widetilde{s}_{j} :=
 \begin{cases} \ell  - 1 & \text{ for } j = 1  \\  s^*_{j-1} & \text{ for } j \geq 2  \end{cases}.\end{align} Similar to the previous case, we call $\{\widetilde{s}_{j}\}$ the {\em induced} (or {\em $(B, \{s^*_n\})$-induced}) acceleration sequence.  The analog of (\ref{eqn:AppendExp2}) is \begin{align}\label{eqn:AppendExp2MIsOne} \sum_{i=1}^{j} (\widetilde{s}_i+1) = \begin{cases}   \ell & \text{ for } j = 1 \\  \ell + \sum_{i=2}^{j} (s^*_{i-1}+1) & \text{ for } j \geq 2\end{cases} .
  \end{align}   Thus, there exists a closed ball $\widetilde{A}_{1}$ such that $\widetilde{A}_{1} \in \widetilde{B}^\alpha_{1}$ and $\widetilde{A}_{1} \supset \widetilde{B}_{2}$.  Now let $\widetilde{A}_{2} \in \widetilde{B}^\alpha_{2}$. We call $\widetilde{B}_1 \supset \widetilde{B}_{2}$ the {\em induced} ({\em $B$-induced}, or {\em $(B, \{s^*_n\})$-induced}) ($2$-finite) play and \[\widetilde{B}_1 \supset \widetilde{A}_1 \supset \widetilde{B}_{2} \supset  \widetilde{A}_2\] an {\em associated dyadic finite play}.  Remark~\ref{rmk:InsertPlay} also applies for the case for which $B$ is appendable and $k =1$.

\begin{lemm}\label{lemm:InducedWinningPlayAppend} Let $k \in  \NN$, $B_1 \supset B_{2} \supset \cdots \supset B_{k}$ and $B^* \supset B$ be finite plays that are winning for Bob such that $B$ is appendable for $B_1 \supset B_{2} \supset \cdots \supset B_{k}$, and $\{s^*_n\}:= \{s^*_n\}_{n=1}^\infty$ be any acceleration sequence.  Then the $(B, \{s^*_n\})$-induced play $\widetilde{B}_1 \supset \widetilde{B}_{2} \supset \cdots \supset \widetilde{B}_{k+1}$ is winning for Bob.
 
\end{lemm}
\begin{proof}  Since $B_1 \supset B_{2} \supset \cdots \supset B_{k}$ and $B^* \supset B$ are winning for Bob, the $B_n$ are chosen according to $g_{B_1}$, which is an $(\alpha, \beta; S)$-winning strategy for Bob with initial ball $B_1$, and $B$ is chosen according to $g_{B^*}$, which is an $(\alpha, \beta; S)$-winning strategy for Bob with initial ball $B^*$.  Define the set \[\Am := \{A \in  \widetilde{B}_k^{\alpha (\alpha \beta)^{\widetilde{s}_k}}: B \subset A\}.\]  Note that the elements of $\Am$ have the same radius, and this radius is strictly larger than the radius of any element of $\Omega(B)$.  Thus, $\Am \cap \Omega(B) = \emptyset$.  Define a function \begin{align*} \widetilde{g}_{B_1}( A) := \begin{cases} B & \text{ if } A  \in \Am\\
g_{B^*}(A) & \text{ if } A \in \Omega(B) \\
g_{B_1}( A) & \text{ if } A \in \Omega(B_1) \backslash \{\Am \cup \Omega(B)\}
 \end{cases}.
  \end{align*}  We claim that $\widetilde{g}_{B_1}$ is an $(\alpha, \beta; S)$-winning strategy for Bob with initial ball $B_1$.  We now prove the claim.  First note that, by (\ref{eqn:AppendExp2}, \ref{eqn:AppendExp2MIsOne}),  $\widetilde{g}_{B_1}(A) \in A^\beta$ for all $A \in  \Am$ and, thus, for all $A \in \Omega(B_1)$.  Now let  $B'_1 := B_1$, let $B'_1 \supset B'_{2} \supset \cdots$ be a play for an $(\alpha, \beta; S)$-game, and let \[B'_1 \supset A'_1 \supset  B'_{2} \supset A'_2 \supset \cdots\] be its associated dyadic play for which the $A'_1, A'_2, A'_3, \cdots$ and $B'_2, B'_3, B'_4, \cdots $ are chosen by  \begin{align*} \begin{cases}  A'_n \in (B'_n)^\alpha & \textrm{ for } n \in \NN \\
B'_{n+1} = \widetilde{g}_{B_1}(A'_n) \in (A'_n)^{\beta} & \textrm{ for } n \in \NN  \end{cases}.
  \end{align*}  There are two cases to consider.
  
 \begin{enumerate}
\item There exists an $n \in \NN$ for which $A'_n \in \Am$.
\item There does not exist an $n \in \NN$ for which $A'_n \in \Am$.
\end{enumerate}

Let us consider Case~(1).  Since every element of $\Am$ has the same radius, we have that $n$ is unique.  The condition of Case~(1) gives that $B'_{n+1} = B$.  Let $N \geq n+1$.  Since $A'_N \in \Omega(B)$, we have that $\widetilde{g}_{B_1}( A'_{N}) = g_{B^*}( A'_{N})$ and, thus, that $\bigcap_{N=n+1}^\infty B'_n \in S^c$.  Hence, we have that \begin{align}\label{eqn:lemm:AppendWinningPlay}
  \bigcap_{m=1}^\infty B'_m \in S^c. \end{align}  This completes Case~(1).
 
Now let us consider Case~(2).  By the rules of the Schmidt game, there exists a unique $M \in \NN$ such that $\rho(A'_M)$ is equal to the radius of the elements of $\Am$.  Consequently, we must have that $A'_M \not\supset B$ or $A'_M \not\subset\widetilde{B}_k$ or both and, thus, $\widetilde{g}_{B_1}( A'_{M}) = g_{B_1}( A'_{M})$.  If there exists an integer $\widetilde{N} \geq M+1$ such that $A'_{\widetilde{N}} \in \Omega(B)$, then let $N \geq M+1$ be the least such.  Consequently, since $A_i \subset A_N$ for all $i \geq N$, we have that $\widetilde{g}_{B_1}( A'_{i}) = g_{B^*}( A'_{i})$.  Thus, we have $\bigcap_{i=N}^\infty B'_i \in S^c$, which yields (\ref{eqn:lemm:AppendWinningPlay}).

Otherwise, there exists no integer $\widetilde{N} \geq M+1$ such that $A'_{\widetilde{N}} \in \Omega(B)$.  Thus, $A'_m \notin \Omega(B)$ for all $m \in \NN$.  Therefore, $\widetilde{g}_{B_1}( A'_{m}) = g_{B_1}( A'_{m})$ for all $m \in \NN$ and (\ref{eqn:lemm:AppendWinningPlay}) holds.  This completes Case~(2) and the claim.

We now claim that the $\widetilde{B}_n$, for $n =1, \cdots, k+1$, are chosen according to $\widetilde{g}_{B_1}$.  The proof of this claim is similar to the proof of the corresponding claim in Lemma~\ref{lemm:InducedWinningPlayInsert}.  This yields the desired result and completes the proof of the lemma.  \end{proof}

\begin{rema}\label{rmk:DefnAdaptedWinningStrat}
Given a $g_{B_1}$, i.e. an $(\alpha, \beta; S)$-winning strategy for Bob with initial ball $B_1$, the strategy $\widetilde{g}_{B_1}$ constructed from it in the proof of Lemma~\ref{lemm:InducedWinningPlayInsert} is called an {\em adapted $(\alpha, \beta; S)$-winning strategy for Bob (with initial ball $B_1$)}.  Note that, for $B$ insertable, it does not matter how $B$ is chosen (whether according to a winning strategy for Bob or not).  Also note that, if Bob were to apply $\widetilde{g}_{B_1}$ to choose $\widetilde{B}_{k+2}$, it would be the same as applying $g_{B_1}$  to choose $\widetilde{B}_{k+2}$.

For $B$ appendable, we also have $g_{B^*}$, i.e. the $(\alpha, \beta; S)$-winning strategy for Bob with initial ball $B^*$, and the strategy $\widetilde{g}_{B_1}$ constructed from $g_{B_1}$ and $g_{B^*}$ in the proof of Lemma~\ref{lemm:InducedWinningPlayAppend} is also called an {\em adapted $(\alpha, \beta; S)$-winning strategy for Bob (with initial ball $B_1$)}.  Note that, if Bob were to apply $\widetilde{g}_{B_1}$  to choose $\widetilde{B}_{k+2}$, it would be the same as applying $g_{B^*}$ to choose $\widetilde{B}_{k+2}$.  

\end{rema}

 \begin{rema}\label{rmk:DistinctPlaysWhenDiffBallsSameRad}  If $B \neq B'$ are closed balls with $\rho(B) = \rho(B')$ such that both are insertable for a $k$-finite play, then the $B$-induced and $B'$-induced $k+1$-finite plays are distinct.  Let $\{s^*_n\}:= \{s^*_n\}_{n=1}^\infty$ be any acceleration sequence.  A similar remark applies for $B$ and $B'$ appendable.

\end{rema}

\section{Tessellations}\label{sec:Tesselations}  Let $d \in \NN$ and $0< \alpha<1$, $0< \beta <1$.  In this section, we develop the tools, Propositions~\ref{prop:RepsMinTess} and~\ref{prop:RepsMaxTess}, that we need to compute an upper bound for the Hausdorff dimension of an $(\alpha, \beta)$-ubiquitously losing set of $\RR^d$ for certain $\alpha$ and $\beta$.  We do this in Theorem~\ref{thm:UpBndHDLosingSets} for the supremum norm $\| \cdot \|_\infty$ on $\RR^d$ because this norm yields families of closed balls that tessellate $\RR^d$, and such tessellations (defined below) are convenient for computing an upper bound on the Hausdorff dimension.  The general case of any norm on $\RR^d$ follows by applying Proposition~\ref{prop:BilipschitzInvar2} to Theorem~\ref{thm:UpBndHDLosingSets}.

Note that, for $d=1$, the proof of our main tools, Propositions~\ref{prop:RepsMinTess} and~\ref{prop:RepsMaxTess}, are simpler and, heuristically, can be thought of as sliding a closed interval along an equally-spaced partition on the real line and counting the number of partition elements that the closed interval intersects.  Readers who are interested only in dimension $d=1$ may wish to omit this section, Section~\ref{sec:Tesselations}, except for some definitions and the statements of  Propositions~\ref{prop:RepsMinTess} and~\ref{prop:RepsMaxTess}, and continue directly to Section~\ref{sec:UpBndHDLosingSets}.

\subsection{Elementary properties of tessellations}  Let $\boldsymbol{x}:=(x_1, \cdots, x_d)^t\in \RR^d$, and define the projection maps \begin{align}\label{eqn:DefnProjMaps}
 \pi_i(\boldsymbol{x})= x_i  \end{align} for integers $i$ such that $1 \leq i \leq d$.  Next, consider the coordinate axes in $\RR^d$.  Let us denote them by $1$-axis, $\cdots,$ $d$-axis.  Thus, for an integer $ 1 \leq i \leq d$, if, for $x \in \RR$, we let $(0, \cdots, 0, x, 0, \cdots 0)^t$ denote the element of $\RR^d$ with $x$ in the $i$-th component and zeros in all other components, we have that the $i$-axis is the subset of $\RR^d$ given by $\{(0, \cdots, 0, x, 0, \cdots 0)^t: x \in \RR\}$.  Recall that a {\em $d$-dimensional hypercube whose edges are parallel to the coordinate axes in $\RR^d$} is the Cartesian product $\prod_{i=1}^d [a_i, b_i]$ where $a_i < b_i$ are real numbers such that $b_i - a_i$ are the same for all integers $1\leq i \leq d$.

Let $R>0$ and $B(\boldsymbol{x}, R)$ denote the closed ball given by $\|\cdot\|_\infty$ of radius $R$ and center $\boldsymbol{x}$.  In coordinates, we have \[B(\boldsymbol{x}, R) = [-R+ x_1, x_1 +R] \times \cdots \times [-R+x_d, x_d+R].\]  Now, for all integers $ 1 \leq i \leq d$, let $a_i < b_i$ be real numbers such that $b_i - a_i = 2R$.  Then the Cartesian product $\prod_{i=1}^d [a_i, b_i]$ is a closed ball given by $\|\cdot\|_\infty$ of radius $R$ and center $(a_1 +R, \cdots, a_d + R)^t$.  Consequently, a $d$-dimensional hypercube with edges parallel to the coordinate axes and edge length $2R$ is a closed ball given by $\|\cdot\|_\infty$ of radius $R$, and vice versa.   

Note that any closed ball given by $\|\cdot\|_\infty$ of radius $R$ can be translated to another such closed ball by applying the translation of $\RR^d$ that takes the center of the first ball to the center of the second.  Finally, we note that the interior of a closed ball given by $\|\cdot\|_\infty$ is an open ball given by $\|\cdot\|_\infty$:\begin{align}\label{eqn:defnInterior}
 \inte\left(\prod_{i=1}^d [a_i, b_i]\right) = \prod_{i=1}^d (a_i, b_i). \end{align}

Let $\boldsymbol{y}:=(y_1, \cdots, y_d)^t\in B(\boldsymbol{x}, R).$  The {\em chord (of $B(\boldsymbol{x}, R)$) parallel to the $i$-axis (through the point $\boldsymbol{y}$)}, which we denote by $\sL_i(\boldsymbol{y}) := \sL_i(\boldsymbol{y}, B(\boldsymbol{x}, R))$, is the set \[\sL_i(\boldsymbol{y}):= \sL_i(\boldsymbol{y}, B(\boldsymbol{x}, R)):=\{y_1\} \times \cdots \times \{y_{i-1}\} \times [-R+x_i, x_i+R] \times \{y_{i+1}\} \times \cdots \times \{y_{d}\}\] and the points \begin{align}\label{eqn:EndPointsChordDefn}
 \boldsymbol{e}\left(\sL_i(\boldsymbol{y})\right)&:=\boldsymbol{e}\left(\sL_i(\boldsymbol{y}), B(\boldsymbol{x}, R)\right):=(y_1, \cdots, y_{i-1}, x_i-R, y_{i+1}, \cdots, y_{d}), \\\nonumber  \boldsymbol{e'}\left(\sL_i(\boldsymbol{y})\right)&:= \boldsymbol{e'}\left(\sL_i(\boldsymbol{y}),B(\boldsymbol{x}, R)\right):=(y_1, \cdots, y_{i-1}, x_i+R, y_{i+1}, \cdots, y_{d}) \end{align} are called the {\em endpoints of $\sL_i(\boldsymbol{y})$}.  Thus, we have that any chord parallel to any axis through any point of $B(\boldsymbol{x}, R)$ has length $2R$.  In particular, the length of any edge is $2R$. 

We define an equivalence relation $\sim_i$ on $B(\boldsymbol{x}, R)$ as follows:  $\boldsymbol{y} \sim_i \boldsymbol{z}$ if and only if $\sL_i(\boldsymbol{y})= \sL_i(\boldsymbol{z})$.  The verification that $\sim_i$ is indeed an equivalence relation is immediate from the definition of a chord of $B(\boldsymbol{x}, R)$ parallel to the $i$-axis.  (Note that, in coordinates, the condition that $\sL_i(\boldsymbol{y})= \sL_i(\boldsymbol{z})$ is equivalent to the condition that $y_1 = z_1, \cdots, y_{i-1} = z_{i-1}, y_{i+1} = z_{i+1}, \cdots, y_{d} = z_d$ and $y_i, z_i \in [-R + x_i, x_i +R]$.)

\begin{lemm}  Let $d \in \NN$, $i$ be an integer such that $1 \leq i \leq d$, $R>0$,  and $\boldsymbol{x}\in \RR^d$.  Then, for each $i$, we have that the chords of $B(\boldsymbol{x}, R)$ parallel to the $i$-axis partition $B(\boldsymbol{x}, R)$.
\end{lemm} 
\begin{proof}
 This follows because $\sim_i$ is an equivalence relation on $B(\boldsymbol{x}, R)$ for each $i$.
\end{proof}

\begin{lemm}\label{lemm:ChordsThruInteriorPointsChar}  Let $d \in \NN$, $i$ be an integer such that $1 \leq i \leq d$, $R>0$, and $\boldsymbol{x}\in \RR^d$.  Let $\boldsymbol{y} \in \inte(B(\boldsymbol{x}, R))$.  Then \begin{itemize}
\item  $\sL_i(\boldsymbol{y}) \backslash \{\boldsymbol{e}\left(\sL_i(\boldsymbol{y})\right), \boldsymbol{e'}\left(\sL_i(\boldsymbol{y})\right)\} \subset \inte(B(\boldsymbol{x}, R))$ and 
\item $ \boldsymbol{e}\left(\sL_i(\boldsymbol{y})\right), \boldsymbol{e'}\left(\sL_i(\boldsymbol{y})\right) \in B(\boldsymbol{x}, R) \backslash \inte(B(\boldsymbol{x}, R))$.
\end{itemize}
\end{lemm}

\begin{proof}
 Using coordinates, we have \begin{align*} \sL_i(\boldsymbol{y}) \backslash & \{\boldsymbol{e}\left(\sL_i(\boldsymbol{y})\right), \boldsymbol{e'}\left(\sL_i(\boldsymbol{y})\right)\}  \\ &=  \{y_1\} \times \cdots \times \{y_{i-1}\} \times (-R+x_i, x_i+R) \times \{y_{i+1}\} \times \cdots \times \{y_{d}\}\end{align*} and \begin{align*}  \inte(B(\boldsymbol{x}, R)) = (-R+ x_1, x_1 +R) \times \cdots \times (-R+x_d, x_d+R).
  \end{align*}  Since $\boldsymbol{y} \in \inte(B(\boldsymbol{x}, R))$, the first result follows.  The second result follows from the definitions.

\end{proof}

Let $M \in \NN$ and $x$ be a real number.  The {\em $M$-tessellation of the closed interval $[-R+x, x+R]$} is the family of closed intervals\begin{align*} 
\left\{B[m, x]:=\left[\left(\frac {2 (m-1) } M-1\right)R+x, x+ \left(\frac {2 m} M-1\right)R \right] : m = 1, \cdots, M \right\},  \end{align*} where $\rho(B[m, x])=R/M$ for all $m = 1, \cdots, M$.  (Recall that $\rho(\cdot)$ denotes the radius of a ball.)  Furthermore, the {\em $M$-tessellation of $B(\boldsymbol{x}, R)$} is the family \begin{align}\label{eqn:MTesselationDefn}
\left\{B[(m_1, \cdots, m_d), \boldsymbol{x}] :=\prod_{i=1}^d\left[\left(\frac {2 (m_i-1) } M-1\right)R+x_{i}, x_{i}+ \left(\frac {2 m_i} M-1\right)R \right] : m_i = 1, \cdots, M \textrm{ for all } i \right\} \end{align} comprised of $M^d$ closed balls $B[(m_1, \cdots, m_d), \boldsymbol{x}]$ given by $\| \cdot \|_\infty$ such that $\rho(B[(m_1, \cdots, m_d), \boldsymbol{x}])=R/M$ for all $i = 1, \cdots, d$ and $m_i= 1, \cdots, M$.  When $d=1$, these two constructions agree.  \begin{rema}\label{rema:MTessUnique}

For any $d \in \NN$, the $M$-tessellation of $B(\boldsymbol{x}, R)$ is, by construction, unique.

\end{rema}

\begin{lemm}\label{lemm:TesselBallFund} Let $d, M \in \NN$, $R>0$,  $\boldsymbol{x}:=(x_1, \cdots, x_d)^t\in \RR^d$, and $\B$ be the $M$-tessellation of $B(\boldsymbol{x}, R)$.  Then the following hold. \begin{itemize}
\item $\bigcup_{B \in \B} B = B(\boldsymbol{x}, R)$ and,
\item $\inte(B) \cap B' = \emptyset $ for any $B, B' \in \B$ such that $B \neq B'$.

\end{itemize} 
\end{lemm}
\begin{proof}
 The first assertion follows by the construction of the $M$-tessellation $\B$.  Let $B= B[(m_1, \cdots, m_d), \boldsymbol{x}]$ and $B'=B[(m'_1, \cdots, m'_d), \boldsymbol{x}]$.  Since $B$ and $B'$ are distinct, there exists an $i$ such that the $m_i \neq m'_i$.  Thus, we have that \begin{align*}&
\left(\left(\frac {2 (m_i-1) } M-1\right)R+x_{i}, x_{i}+ \left(\frac {2 m_i} M-1\right)R \right) \cap  \\ &\quad \quad \left[\left(\frac {2 (m'_i-1) } M-1\right)R+x_{i}, x_{i}+ \left(\frac {2 m'_i} M-1\right)R \right] = \emptyset \end{align*}  By (\ref{eqn:defnInterior}), we have that $\inte(B) \cap B' = \emptyset$.
\end{proof}

Let $S \subset \RR^d$ and $\boldsymbol{a}:=(a_1, \cdots, a_d)^t\in \RR^d$.  We use the notation $S +  \boldsymbol{a}$ for translation by $\boldsymbol{a}$:  namely, we have \begin{align}\label{eqn:TranslationEuclid}
 S +  \boldsymbol{a}:= \{\boldsymbol{x} + \boldsymbol{a}: \boldsymbol{x} \in S\}.
 \end{align} \begin{lemm}\label{lemm:TesselBallTrans} Let $d, M \in \NN$, $R>0$,  $\boldsymbol{x}, \boldsymbol{a} \in \RR^d$,  $\B$ be the $M$-tessellation of $B(\boldsymbol{x}, R)$, and $\B_{ \boldsymbol{a}}$ be the $M$-tessellation of $B(\boldsymbol{x}+ \boldsymbol{a}, R)$.  Then $\B_{ \boldsymbol{a}} = \{B + \boldsymbol{a} : B \in \B\}$.\end{lemm}

\begin{proof}
 Let $\boldsymbol{x}:=(x_1, \cdots, x_d)^t$ and $\boldsymbol{a}:=(a_1, \cdots, a_d)^t$.  Consider $B[(m_1, \cdots, m_d), \boldsymbol{x}] \in \B$.  Then \[B[(m_1, \cdots, m_d), \boldsymbol{x}]+ \boldsymbol{a} = B[(m_1, \cdots, m_d), \boldsymbol{x} +  \boldsymbol{a} ]\] by the construction of an $M$-tessellation.
\end{proof}

We say a family $\B$ of closed balls, each given by $\| \cdot \|_\infty$ and with the same radius $R >0$, is a {\em tessellation} if there exists a closed ball $B \subset \RR^d$ given by $\| \cdot \|_\infty$ such that $\B$ is an $M$-tessellation of $B$ for some $M \in \NN$.\footnote{In the literature (see~\cite[Definition~2.4.1]{HS24} for example), tessellations can be families of more general polytopes, not just hypercubes.  However, we will not use these more general tessellations in this paper.}  By Lemma~\ref{lemm:TesselBallFund}, the closed ball $B$ is unique.

Also, we will need to tessellate all of $\RR^d$.  First, for a given $\boldsymbol{x} \in \RR^d$, we construct the following countably infinite family of closed balls each given by $\| \cdot \|_\infty$ and of radius $R>0$:  \begin{align*} &\B_{\RR^d}(\boldsymbol{x}, R) :=\\& \quad
\left\{B[(m_1, \cdots, m_d), \boldsymbol{x}] :=\prod_{i=1}^d\left[{(2 m_i-1) }R + x_i, x_i+ (2 m_i+1)R \right] : (m_1, \cdots, m_d) \in \ZZ^d \right\}.\end{align*} 

\begin{rema}\label{rmk:UniqCompleteTess}  Note that, for a given $\boldsymbol{x} \in \RR^d$ and $R>0$, the family $\B_{\RR^d}(\boldsymbol{x}, R)$ is uniquely determined by construction.  
\end{rema}

\begin{lemm}\label{lemm:TesselEuclidFund} Let $d \in \NN$, $R>0$,  $\boldsymbol{x}:=(x_1, \cdots, x_d)^t\in \RR^d$, and $\B:= \B_{\RR^d}(\boldsymbol{x}, R)$.  Then the following hold. \begin{itemize}
\item $\bigcup_{B \in \B} B = \RR^d$ and,
\item $\inte(B) \cap B' = \emptyset $ for any $B, B' \in \B$ such that $B \neq B'$.

\end{itemize} 
\end{lemm}
\begin{proof}
 The first assertion follows by the construction of the complete tessellation $\B$.  The proof of the second assertion is analogous to the proof of the second assertion in Lemma~\ref{lemm:TesselBallFund}.
 \end{proof}

\noindent Furthermore, if two of these families share an element in common, then they are the same.
\begin{lemm}\label{lemm:UnifExtrapolCompTess}
 Let $d \in \NN$, $R >0$, and $\boldsymbol{x}, \boldsymbol{x'} \in \RR^d$.  If $\B_{\RR^d}(\boldsymbol{x}, R) \cap \B_{\RR^d}(\boldsymbol{x'}, R) \neq \emptyset$, then $\B_{\RR^d}(\boldsymbol{x}, R) = \B_{\RR^d}(\boldsymbol{x'},R)$.
\end{lemm}
\begin{proof}
Since $\B_{\RR^d}(\boldsymbol{x}, R) \cap \B_{\RR^d}(\boldsymbol{x'}, R) \neq \emptyset$, there is a common element:  $B[(m_1, \cdots, m_d), \boldsymbol{x}]= B[(m'_1, \cdots, m'_d), \boldsymbol{x'}]$.  Consequently, we have that $(2 m_i-1) R + x_i = (2 m'_i-1) R + x'_i$ and, thus, $x'_i = x_i+ 2(m_i - m'_i)R$  for all $i$.  Since $m_i - m'_i \in \ZZ$, we have that $B[(m_1+s_1, \cdots, m_d+s_d), \boldsymbol{x}]= B[(m'_1+s_1, \cdots, m'_d+s_d), \boldsymbol{x'}]$ for all $(s_1, \cdots, s_d) \in \ZZ^d$.
\end{proof}

\begin{lemm}\label{lemm:CompTessExtrapol}
Let $d \in \NN$, $R >0$, and $\boldsymbol{x}, \boldsymbol{x'} \in \RR^d$.  Then $\B_{\RR^d}(\boldsymbol{x}, R) = \B_{\RR^d}(\boldsymbol{x'}, R)$ if and only if $\boldsymbol{x} - \boldsymbol{x'} \in 2R \ZZ^d$.
\end{lemm}

\begin{proof}
We first prove the forward implication.  Since $\B_{\RR^d}(\boldsymbol{x}, R) = \B_{\RR^d}(\boldsymbol{x'}, R)$, we have \[B[(m_1, \cdots, m_d), \boldsymbol{x}]= B[(m'_1, \cdots, m'_d), \boldsymbol{x'}]\] for some $(m_1, \cdots, m_d), (m'_1, \cdots, m'_d) \in \ZZ^d$.  Consequently, we have  $x_i -x'_i= 2(m'_i - m_i)R \in 2R \ZZ$  for all $i$.  Hence, we have that $\boldsymbol{x} - \boldsymbol{x'}\in 2R \ZZ^d$.
 
Now we prove the reverse implication.  Since $\boldsymbol{x} - \boldsymbol{x'}\in 2R \ZZ^d$ holds, we have, for all $i$, that there exists $\widetilde{m}_i \in \ZZ$ such that $x'_i  = x_i +2\widetilde{m}_iR$ holds.  Consequently, we have that $B[(\widetilde{m}_1, \cdots, \widetilde{m}_d), \boldsymbol{x}] = B[(0, \cdots, 0), \boldsymbol{x'}]$.  The desired result now follows by Lemma~\ref{lemm:UnifExtrapolCompTess}.
\end{proof}

\noindent As an immediate corollary, we obtain, for a given $R>0$, that these families are uniquely determined by a point of the torus $\RR^d/(2R\ZZ^d)$.  Let $\boldsymbol{z} \in \RR^d/(2R\ZZ^d)$.  We call the family $\B_{\RR^d}(\boldsymbol{z}, R)$ a {\em complete tessellation (of $\RR^d$ by closed balls given by $\| \cdot \|_\infty$ of radius $R$)}, and it is uniquely determined by $\boldsymbol{z}$.  We refer to $\RR^d/(2R\ZZ^d)$ as the {\em parameter torus (for $R$)}.

\begin{lemm}\label{lemm:CompTesselBallTrans} Let $d \in \NN$, $R>0$, $\boldsymbol{z} \in\RR^d/(2R\ZZ^d)$, and $\boldsymbol{a} \in \RR^d/(2R\ZZ^d)$.  Then \[\B_{\RR^d}(\boldsymbol{z}+\boldsymbol{a}, R) = \{B + \boldsymbol{a} : B \in \B_{\RR^d}(\boldsymbol{z}, R)\}.\]\end{lemm} 

\begin{rema}
The sum $\boldsymbol{z}+\boldsymbol{a}$ is understood to be on the parameter torus, namely modulo $2R \ZZ^d$.  The sum $B + \boldsymbol{a}$ is as in (\ref{eqn:TranslationEuclid}).
\end{rema}

\begin{proof}
 The proof is analogous to that of Lemma~\ref{lemm:TesselBallTrans}.
\end{proof}

A {\em subtessellation} of a complete tessellation $\B$ (or tessellation $\B$) is a tessellation $\mS$ such that $\mS \subset \B$.  Note that all the elements in $\mS$ and in $\B$ have the same radius $R>0$.  Furthermore, note that, if $B \in \B$, then $\{B\}$ is a subtessellation of $\B$.  

\begin{lemm}\label{lemm:SubTessOfCompTess}
Let $d \in \NN$, $M \in \NN \cup \{0\}$, $R>0$, and $\boldsymbol{x}:=(x_1, \cdots, x_d)^t \in \RR^d$.   Then \begin{align*}\mS_M&(\boldsymbol{x}, R):= \\ &\left\{
\prod_{i=1}^d\left[{(2 m_i-1) }R + x_i, x_i+ (2 m_i+1)R \right] : (m_1, \cdots, m_d) \in \ZZ^d \cap [-M, M]^d \right\}  \end{align*} is a subtessellation of $\B_{\RR^d}(\boldsymbol{x}, R)$ and the $(2M+1)$-tessellation of $B(\boldsymbol{x}, (2M+1)R)$.
\end{lemm}

\begin{proof} By construction, we have $\mS_M(\boldsymbol{x}, R) \subset \B_{\RR^d}(\boldsymbol{x}, R)$.  We now show that $\mS_M(\boldsymbol{x}, R)$ is the $(2M+1)$-tessellation of $B(\boldsymbol{x}, (2M+1)R)$ (which is unique by Remark~\ref{rema:MTessUnique}).   Let $B \in \mS_M(\boldsymbol{x}, R)$.  Thus we have \[B = \prod_{i=1}^d\left[{(2 m_i-1) }R + x_i, x_i+ (2 m_i+1)R \right]\] for some $(m_1, \cdots, m_d) \in \ZZ^d \cap [-M, M]^d.$  Setting \begin{align}\label{eqn:SubTessInCompTessIndexBij}
 \widetilde{m_i} = m_i +M +1, \end{align} we have that \begin{align*}
B &= \prod_{i=1}^d\left[{2 (\widetilde{m_i}-1) -2M-1) }R+ x_i, x_i+ (2 \widetilde{m_i}-2M-1)R \right] \\ &= \prod_{i=1}^d\left[\left(\frac{2 (\widetilde{m_i}-1)}{2M+1} -1\right)(2M+1) R+ x_i, x_i+ \left(\frac{2 \widetilde{m_i}}{2M+1}-1\right)(2M+1)R \right].\end{align*}  Therefore, $B$ is an element of the $(2M+1)$-tessellation of $B(\boldsymbol{x}, (2M+1)R)$.

Note that (\ref{eqn:SubTessInCompTessIndexBij}) gives a bijection between $\ZZ^d \cap [-M, M]^d$ and $\NN^d \cap [1, (2M+1)]^d$.  Consequently, the inverse bijection allows us to take an element of the $(2M+1)$-tessellation of $B(\boldsymbol{x}, (2M+1)R)$ into an element of $\mS_M(\boldsymbol{x}, R)$ by allowing us to reverse the above calculation.  This shows that $\mS_M(\boldsymbol{x}, R)$ is the $(2M+1)$-tessellation of $B(\boldsymbol{x}, (2M+1)R)$ and completes the proof of the lemma.

\end{proof}

Let $\B$ be a tessellation.  Then there exists a closed ball $B$ given by $\|\cdot\|$ for which $\B$ is an $M$-tessellation of $B$ for some $M \in \NN$.  Let $B = B(\boldsymbol{x}, R)$ for some $\boldsymbol{x} \in \RR^d$ and $R>0$.  Consequently, $\B$ is the family given in (\ref{eqn:MTesselationDefn}).  If $M$ is odd, then set $m_i = (M+1)/2$ in \[\left[\left(\frac {2 (m_i-1) } M-1\right)R+x_{i}, x_{i}+ \left(\frac {2 m_i} M-1\right)R \right]\] for all $i$ to obtain the interval \[\left[- \frac R M +x_{i}, x_{i}+  \frac R M\right]\] and, thus, the element $B[(0, \cdots, 0), \boldsymbol{x}]$ of the complete tessellation $\B_{\RR^d}(\boldsymbol{x}, R/M)$.  Consequently, comparing (\ref{eqn:MTesselationDefn}) and the definition of $\B_{\RR^d}(\boldsymbol{x}, R/M)$, we have that $\B$ is a subtessellation of $\B_{\RR^d}(\boldsymbol{x}, R/M)$.  We refer to $\B_{\RR^d}(\boldsymbol{x}, R/M)$ as the {\em completion} of the tessellation $\B$.

If $M$ is even, then set $m_i = M/2$ in \[\left[\left(\frac {2 (m_i-1) } M-1\right)R+x_{i}, x_{i}+ \left(\frac {2 m_i} M-1\right)R \right]\] for all $i$ to obtain the interval \[\left[- \frac {2R} M +x_{i}, x_{i}\right]\] and, thus, the element $B[(0, \cdots, 0), \boldsymbol{x}- (R/M, \cdots, R/M)^t]$ of the complete tessellation $\B_{\RR^d}(\boldsymbol{x} - (R/M, \cdots, R/M)^t, R/M)$.  Consequently,  comparing (\ref{eqn:MTesselationDefn}) and the definition of $\B_{\RR^d}(\boldsymbol{x}- (R/M, \cdots, R/M)^t, R/M)$, we have that $\B$ is a subtessellation of $\B_{\RR^d}(\boldsymbol{x}- (R/M, \cdots, R/M)^t, R/M)$.  We refer to $\B_{\RR^d}(\boldsymbol{x}- (R/M, \cdots, R/M)^t, R/M)$ as the {\em completion} of the tessellation $\B$.  Remark~\ref{rmk:UniqCompleteTess} gives that the completion of a tessellation is unique.  

\begin{lemm}\label{lemm:UniqCompletion}
Let $\B$ be a tessellation.  Then $\B$ is a subtessellation of exactly one complete tessellation, namely the completion of $\B$.
\end{lemm}
\begin{proof}
Let $R>0$ and $\boldsymbol{x} \in \RR^d $ be such that $\B_{\RR^d}(\boldsymbol{x}, R)$ is the completion of $\B$.  Remark~\ref{rmk:UniqCompleteTess} gives that $\B_{\RR^d}(\boldsymbol{x}, R)$ is a unique family.  If $\B$ is a subtessellation of another complete tessellation $\B_{\RR^d}(\boldsymbol{x'}, R')$ for some $R'>0$ and $\boldsymbol{x'} \in \RR^d $, then we have that $\emptyset \neq \B \subset \B_{\RR^d} \cap \B_{\RR^d}'$.  Consequently, the elements in $\B$ have radius $R$ and radius $R'$, which gives that $R=R'$ because the elements of any tessellation have the same radius.  The result now follows from Lemma~\ref{lemm:UnifExtrapolCompTess}.
\end{proof}

\begin{rema}\label{rmk:lemm:UniqCompletion}  We use the notation $\overline{\B}$ to denote the unique completion of $\B$.
\end{rema}

\begin{coro}\label{coro:UniqCompletion}  Let $\mS, \mS'$ be subtessellations of a complete tessellation $\B$.  Then $\overline{\mS}= \overline{\mS'} = \B$.
 
\end{coro}
\begin{proof}
 This follows immediately from Lemma~\ref{lemm:UniqCompletion}.
\end{proof}

\begin{lemm}\label{lemm:ElementPartofTess} Let $d \in \NN$, $R' >0$, $\boldsymbol{x'} \in \RR^d$, $\B:= \B_{\RR^d}(\boldsymbol{x'}, R')$, $B \in \B$, and $\mS$ be a subtessellation of $\B$.  If \[\inte(B) \cap  \bigcup_{B'' \in \mS} B'' \neq \emptyset,\] then $B \in \mS$.
 
\end{lemm}

\begin{proof}  We have that $\inte(B) \cap B'' \neq \emptyset$ for some $B'' \in \mS$.  Lemma~\ref{lemm:TesselEuclidFund} implies that $B = B''$, which gives the desired result.
\end{proof}

\begin{lemm}\label{lemm:ChordSubtesselProperties1}  Let $d \in \NN$, $R' >0$, $\boldsymbol{x'} \in \RR^d$, $\B:= \B_{\RR^d}(\boldsymbol{x'}, R')$, and $\mS$ be a subtessellation of $\B$ and also an $M$-tessellation of $B(\boldsymbol{\widetilde{x}}, \widetilde{R})$ for some $M \in \NN$, $\boldsymbol{\widetilde{x}} \in \RR^d$, and $\widetilde{R}>0$. Then $M = \widetilde{R}/R'$.

\end{lemm}

\begin{proof}
 By the definition of subtessellation, $\mS$ is itself a tessellation, and, thus, there exists a closed ball $B:=B(\boldsymbol{\widetilde{x}}, \widetilde{R})$ given by $\| \cdot \|_\infty$ for some  $\boldsymbol{\widetilde{x}} \in \RR^d$ and $\widetilde{R}>0$ such that $\mS$ is an $M$-tessellation of $B$ for some $M \in \NN$.  Every element of $\mS$ has radius $R'$ because every element of $\B$ has radius $R'$.  Since every element of the $M$-tessellation has radius $\widetilde{R}/M$, we have that $\widetilde{R}/M = R'$, which yields the desired result.
\end{proof}

\begin{lemm}\label{lemm:ChordSubtesselProperties2}   Let $d \in \NN$, $i$ be an integer such that $1 \leq i \leq d$, $R' >0$, $\boldsymbol{x'} \in \RR^d$, $\B:= \B_{\RR^d}(\boldsymbol{x'}, R')$, and $\mS$ be a subtessellation of $\B$ and also an $M$-tessellation of $B:=B(\boldsymbol{\widetilde{x}}, \widetilde{R})$ for some $M \in \NN$, $\boldsymbol{\widetilde{x}} \in \RR^d$, and $\widetilde{R}>0$.  Let $\boldsymbol{y} \in \inte(B)$ and $\sL_i(\boldsymbol{y})$ be the  chord of $B$ parallel to the $i$-axis through the point $\boldsymbol{y}$.  Then we have that \begin{align}\label{eqn:lemm:ChordSubtesselProperties2} 
 \boldsymbol{e}\left(\sL_i(\boldsymbol{y})\right), \boldsymbol{e'}\left(\sL_i(\boldsymbol{y})\right) \in \bigcup_{B'' \in \mS} \left(B'' \backslash \inte(B'')\right) \subset \bigcup_{B'' \in \B} \left(B'' \backslash \inte(B'')\right). \end{align}

\end{lemm}
\begin{proof}
 By the definition of subtessellation, $\mS$ is also an $M$-tessellation of $B:=B(\boldsymbol{\widetilde{x}}, \widetilde{R})$ for some $M \in \NN$, $\boldsymbol{\widetilde{x}} \in \RR^d$, and $\widetilde{R}>0$.  By Lemma~\ref{lemm:TesselBallFund} , we  have that\begin{align*}\label{eqn:lemm:ChordSubtesselProperties2:tessint} 
  \bigcup_{B'' \in \mS} \inte(B'') \subset \inte(B), \end{align*} which, together with Lemmas~\ref{lemm:ChordsThruInteriorPointsChar} and~\ref{lemm:TesselBallFund}, yields \[\boldsymbol{e}\left(\sL_i(\boldsymbol{y})\right), \boldsymbol{e'}\left(\sL_i(\boldsymbol{y})\right) \in  B \backslash \inte(B) \subset \left(\bigcup_{B'' \in \mS} B'' \right)\backslash \bigcup_{B'' \in \mS} \inte(B'') = \bigcup_{B'' \in \mS} \left(B'' \backslash \inte(B'') \right).\]  Since $\mS \subset \B$, the set inclusion of (\ref{eqn:lemm:ChordSubtesselProperties2}) also follows.  This yields the desired result.
\end{proof}

Let $R >0$, $\boldsymbol{x} \in \RR^d$, and $N \in \NN$.  We define the {\em $N$-refinement of $B(\boldsymbol{x}, R)$}, denoted by $\mR_N(B(\boldsymbol{x}, R))$, to be the $N$-tessellation of $B(\boldsymbol{x}, R)$.  Note that $\mR_1(B(\boldsymbol{x}, R)) = \{B(\boldsymbol{x}, R)\}$.  Let $\B$ be a tessellation or a complete tessellation.   We define the {\em $N$-refinement of $\B$}, denoted by $\mR_N(\B)$, to be \begin{align*} \mR_N(\B) := \bigcup_{B \in \B }\mR_N(B)
\end{align*}  Note that $\mR_1(\B) = \B$ and $\mR_N(\{B(\boldsymbol{x}, R)\}) = \mR_N(B(\boldsymbol{x}, R))$.

\begin{lemm}\label{lemm:RefineTess}
Let $d,M, N \in \NN$, $R >0$, $\boldsymbol{x} \in \RR^d$, and $\B$ be the $M$-tessellation of $B(\boldsymbol{x}, R)$.  Then $\mR_N(\B)$ is the $(MN)$-tessellation of $B(\boldsymbol{x}, R)$.
\end{lemm}

\begin{proof}  

Let $B[(m_1, \cdots, m_d), \boldsymbol{x}] \in \B$.  Then \[B[(m_1, \cdots, m_d), \boldsymbol{x}]=\prod_{i=1}^d\left[x_i - R+ \frac {2 (m_i-1) R} M, x_{i}-R+ \frac {2 m_i R} M \right]\] for some integers $1 \leq m_1, \cdots, m_d \leq M$.  Note that $\rho(B[(m_1, \cdots, m_d), \boldsymbol{x}]) = R/M$, and thus we have that an arbitrary element of $\mR_N(B[(m_1, \cdots, m_d), \boldsymbol{x}])$ is \[B:=\prod_{i=1}^d\left[x_i - R+ \frac {2 (m_i-1) R} M + \frac {2 (n_i-1) \frac R M} N, x_i - R+ \frac {2 (m_i-1) R} M + \frac {2 n_i \frac R M} N\right]\] for some integers $1 \leq n_1, \cdots, n_d \leq N$.  Hence, we have \[B:=\prod_{i=1}^d\left[\left(2 \left(\frac{N(m_i-1) + n_i -1}{MN}\right) -1 \right)R +x_i, x_i + \left(2 \left(\frac{N(m_i-1) + n_i }{MN}\right) -1 \right)R \right].\]

Now the division algorithm gives that, for any integer $p$, we have $p = m N +n$ where integers $m$ and $0 \leq n \leq N-1$ are uniquely determined by $p$ and $N$.  Requiring $0 \leq p \leq MN-1$ gives that $0 \leq m \leq M-1$.

Moreover, for any integers $0 \leq m \leq M-1$ and $0 \leq n \leq N-1$, the integer $m N +n$ lies in the closed interval $[0, MN-1]$.  Consequently, $\{p \in \ZZ : 0 \leq p \leq MN-1\}$ is in bijection with $\{(m,n) \in \ZZ \times \ZZ:  0 \leq m \leq M-1, 0 \leq n \leq N-1\}$ and the bijection is given by $p = m N +n$.

Applying the proceeding to each coordinate $i$ by setting $m = m_i-1$ and $n = n_i-1$ gives that the integers $1 \leq p_i \leq MN$ can be written uniquely as $p_i=N(m_i-1) + n_i $ for integers $1\leq m_i \leq M$ and $1 \leq n_i \leq N$.  Consequently, we have that \[B:=\prod_{i=1}^d\left[\left(2 \left(\frac{p_i-1}{MN}\right) -1 \right)R +x_i, x_i + \left(2 \left(\frac{p_i }{MN}\right) -1 \right)R \right]\] for some integers $1 \leq p_1, \cdots, p_d \leq MN$. 

Since  $1 \leq m_1, \cdots, m_d \leq M$ and $1 \leq n_1, \cdots, n_d \leq N$ hold, applying the bijection from the division algorithm in each coordinate $i$ gives us the desired result.  \end{proof}

\begin{rema}
For $N=1$, Lemma~\ref{lemm:RefineTess} also follows by the definition of $1$-refinement and the uniqueness of an $M$-tessellation for a given $B(\boldsymbol{x}, R)$.

\end{rema}

\subsection{Constructing subtessellations using a uniform local coordinate system} Let $d \in \NN$, $i$ be an integer such that $1 \leq i \leq d$, $R' >0$, and $\boldsymbol{x'} \in \RR^d$.  Fix a complete tessellation $\B:= \B_{\RR^d}(\boldsymbol{x'}, R')$ and let \[\boldsymbol{y} \in \bigcup_{B'' \in \B} \inte(B'').\]  For $a_i < y_i < a'_i$, define the sets \[\Ls_i(\boldsymbol{y}):=\Ls_i(\boldsymbol{y}, a_i, a'_i):=\{y_1\} \times \cdots \times \{y_{i-1}\} \times [a_i, a'_i] \times \{y_{i+1}\} \times \cdots \times \{y_{d}\}\] and the points \begin{align*}
\boldsymbol{e}\left(\Ls_i(\boldsymbol{y})\right)&:=(y_1, \cdots, y_{i-1}, a_i, y_{i+1}, \cdots, y_{d}) \textrm { and } \\  \boldsymbol{e'}\left(\Ls_i(\boldsymbol{y})\right)&:=(y_1, \cdots, y_{i-1}, a'_i, y_{i+1}, \cdots, y_{d}).\end{align*}  We refer to the set $\Ls_i(\boldsymbol{y})$ as a {\em line segment parallel to the $i$-axis (through the point $\boldsymbol{y}$)} and the points $\boldsymbol{e}\left(\Ls_i(\boldsymbol{y})\right), \boldsymbol{e'}\left(\Ls_i(\boldsymbol{y})\right)$ as the {\em endpoints of $\Ls_i(\boldsymbol{y})$}.  Note that the length of any $\Ls_i(\boldsymbol{y})$ is positive by construction.  Finally, define the set \[\sP(\Ls_i(\boldsymbol{y})):=\sP(\Ls_i(\boldsymbol{y}), \B) := \Ls_i(\boldsymbol{y}) \cap \bigcup_{B'' \in \B} \left(B'' \backslash \inte(B'')\right).\]

\begin{lemm}\label{lemm:PartitionLineSeg}  Let $d \in \NN$, $i$ be an integer such that $1 \leq i \leq d$, $R' >0$, $\boldsymbol{x'} \in \RR^d$, $\B:= \B_{\RR^d}(\boldsymbol{x'}, R')$ and \[\boldsymbol{y}:=(y_1, \cdots, y_d)^t \in \bigcup_{B'' \in \B} \inte(B'').\]  Then \begin{align*}\sP(\Ls_i(\boldsymbol{y})) &= \left\{(y_1, \cdots, y_{i-1}, y'_i, y_{i+1}, \cdots y_d)^t : y'_i \in [a_i, a'_i] \cap \{(2 m-1)R' + x'_i: m \in \ZZ\}\right\}\end{align*}  is a finite, discrete set of points such that, for $|\sP(\Ls_i(\boldsymbol{y}))|\geq 2$, any two consecutive elements of $\sP(\Ls_i(\boldsymbol{y}))$ have distance exactly $2R'$.  Moreover, if $a'_i - a_i \geq 2nR'$ for some $n \in \NN$, then $|\sP(\Ls_i(\boldsymbol{y}))| \geq n$.
 
\end{lemm}

\begin{rema*}
Note that the set $\sP(\Ls_i(\boldsymbol{y}))$ from Lemma~\ref{lemm:PartitionLineSeg} could (but need not) be empty if $a'_i - a_i <2R'$.

\end{rema*}

\begin{proof}  Let $\boldsymbol{x'}:=(x'_1, \cdots, x'_d)^t$.  By the definition of $\B:= \B_{\RR^d}(\boldsymbol{x'}, R')$, we have that \[\boldsymbol{z}:=(z_1, \cdots, z_d)^t \in \bigcup_{B'' \in \B}\left(B'' \backslash \inte(B'') \right)=:\partial(\B)\] if and only if there exists an integer $j$ such that $1 \leq j \leq d$ for which $z_j = (2 m-1)R' + x'_j$ for some $m \in \ZZ$.  Since $\boldsymbol{y} \in \RR^d \backslash \partial(\B)$, we, thus, have, for  every integer $j$ such that $1 \leq j \leq d$, that $y_j \neq  (2 m-1)R' + x'_j$ for any $m \in \ZZ$.  Consequently, we have that \begin{align*}
\sP(\Ls_i(\boldsymbol{y})) &= \Ls_i(\boldsymbol{y}) \cap \partial(\B) \\  &= \left\{(y_1, \cdots, y_{i-1}, y'_i, y_{i+1}, \cdots y_d)^t : y'_i \in [a_i, a'_i] \cap \{(2 m-1)R' + x'_i: m \in \ZZ\}\right\}.\end{align*}  Thus, it follows that $\sP(\Ls_i(\boldsymbol{y}))$ is a discrete set with finite cardinality and, if it contains two or more elements, that the distance between consecutive elements is $2R'$. This proves the first assertion.

To prove the second assertion, we prove a stronger assertion.  Define the following subset of $\sP(\Ls_i(\boldsymbol{y}))$:  \begin{align*}
\widetilde{\sP}(\Ls_i(\boldsymbol{y})) := \left\{(y_1, \cdots, y_{i-1}, y'_i, y_{i+1}, \cdots y_d)^t : y'_i \in (a_i, a'_i] \cap \{(2 m-1)R' + x'_i: m \in \ZZ\}\right\}.\end{align*}  The stronger assertion that we prove is that, if $a'_i - a_i \geq 2nR'$ for some $n \in \NN$, then $|\widetilde{\sP}(\Ls_i(\boldsymbol{y}))| \geq n$.  We induct on $n$ to prove this stronger assertion.  We prove the initial case $n=1$.  Assume that the conclusion does not hold.  Then we have that $|\widetilde{\sP}(\Ls_i(\boldsymbol{y}))| <1$, and, consequently, $\widetilde{\sP}(\Ls_i(\boldsymbol{y}))$ is the empty set.  Thus, we have that \[(a_i, a'_i] \cap \{(2 m-1)R' + x'_i: m \in \ZZ\}= \emptyset,\] which implies that the interval $(a_i, a'_i]$ lies strictly between two consecutive elements of the set $ \{(2 m-1)R' + x'_i: m \in \ZZ\}$.  Thus, we have that $a'_i - a_i < 2R'$, which is a contradiction and thus proves the initial case.

To finish the induction proof, we assume that the result holds for $n-1$ and prove it for $n$.  Now, since we have $a'_i - a_i \geq 2nR'$, there exists a real number $a$ such that $a_i < a < a'_i$ and $a - a_i = 2(n-1)R'$.  Consequently, we have that $a'_i -a \geq 2R'$.  Now the induction hypothesis gives that \[\left| \left\{(y_1, \cdots, y_{i-1}, y'_i, y_{i+1}, \cdots y_d)^t : y'_i \in (a_i, a] \cap \{(2 m-1)R' + x'_i: m \in \ZZ\}\right\}\right| \geq n-1\] and the initial step gives that \[\left| \left\{(y_1, \cdots, y_{i-1}, y'_i, y_{i+1}, \cdots y_d)^t : y'_i \in (a, a'_i] \cap \{(2 m-1)R' + x'_i: m \in \ZZ\}\right\}\right| \geq 1,\] which together yields the stronger assertion.  This proves the second assertion and the completes the proof of the lemma.
\end{proof}

\noindent We refer to $\sP(\Ls_i(\boldsymbol{y}), \B) $ as the {\em partition of the line segment $\Ls_i(\boldsymbol{y})$ (according to the complete tessellation $\B$)}.  

Now let $\Ls_1(\boldsymbol{y}, a_1, a'_1), \cdots, \Ls_d(\boldsymbol{y}, a_d, a'_d)$ be line segments parallel to the $1$-axis, $\cdots$, $d$-axis, respectively, such that their lengths are all the same positive real number $D$.  Hence, $D = a'_i - a_i$ for all integers $i$ such that $1 \leq i \leq d$.  We call the family $\left\{\Ls_1(\boldsymbol{y}), \cdots, \Ls_d(\boldsymbol{y})\right\}$ a {\em uniform local coordinate system (at $\boldsymbol{y}$ of size $D$).}

\begin{lemm}\label{lemm:LocalCoordinates} Let $d \in \NN$, $R' >0$, $R'' >0$, $\boldsymbol{x'} \in \RR^d$, $\B:= \B_{\RR^d}(\boldsymbol{x'}, R')$, \[\boldsymbol{y} \in \bigcup_{B'' \in \B} \inte(B''),\] and $\left\{\Ls_1(\boldsymbol{y}), \cdots, \Ls_d(\boldsymbol{y})\right\}$ be a uniform local coordinate system at $\boldsymbol{y}$ of size $2R''$ such that \begin{align}\label{eqn:lemm:LocalCoordinates}
 \left\{\boldsymbol{e}\left(\Ls_1(\boldsymbol{y})\right), \boldsymbol{e'}\left(\Ls_1(\boldsymbol{y})\right), \cdots, \boldsymbol{e}\left(\Ls_d(\boldsymbol{y})\right), \boldsymbol{e'}\left(\Ls_d(\boldsymbol{y})\right) \right\} \subset \bigcup_{B'' \in \B} \left(B'' \backslash \inte(B'')\right) \end{align} holds.  Then there exists a unique subtessellation $\mS$ of $\B$ such that \begin{enumerate}
\item \[\Ls_1(\boldsymbol{y}) \cup  \cdots \cup \Ls_d(\boldsymbol{y}) \subset \bigcup_{B'' \in \mS} B''  \textrm{ and, } \]
\item if $\mS'$ is a subtessellation of $\B$ such that \[\Ls_1(\boldsymbol{y}) \cup  \cdots \cup \Ls_d(\boldsymbol{y}) \subset \bigcup_{B'' \in \mS'} B'',\] then $\mS' \supset \mS$.

\end{enumerate}

\noindent Moreover, we have that $(R''/R') \in \NN$ and $\mS$ is the $(R''/R')$-tessellation of $\prod_{i=1}^d[a_i, a'_i]$. 
\end{lemm} 

\begin{rema}\label{Rmk:lemm:LocalCoordinates}
Note that, in Lemma~\ref{lemm:LocalCoordinates}, we have that $a'_i -a_i =2R''$ for all $i$ and, hence, $\prod_{i=1}^d[a_i, a'_i]$ is a closed ball given by $\|\cdot\|_\infty$ of radius $R''$.
\end{rema}
\begin{proof}
Using the definition of $\B$ and (\ref{eqn:lemm:LocalCoordinates}), we have that \begin{align}\label{eqn:lemmLocalCoordinates}
 a_i &= (2 m_i-1) R' + x'_i,  \\\nonumber a'_i &= (2 m'_i+1) R' + x'_i  \end{align}for integers $m_i$ and $m'_i$ such that  $m_i \leq m'_i$.  Since $\left\{\Ls_1(\boldsymbol{y}), \cdots, \Ls_d(\boldsymbol{y})\right\}$ is a uniform local coordinate system at $\boldsymbol{y}$ of size $2R''$, we have that \begin{align}\label{eqn:lemmLocalCoordinates2}
 R'' = (m'_i - m_i+1)R' \end{align} for all $i$.  Since $m'_i - m_i + 1\in \NN$, we have that $(R''/R') \in \NN$.

We claim that  \begin{align*} \mS = 
\left\{\prod_{i=1}^d\left[{(2 n_i-1) }R' + x'_i, x'_i+ (2 n_i+1)R' \right] : (n_1, \cdots, n_d) \in \ZZ^d \cap\prod_{i=1}^d[m_i, m'_i] \right\}.\end{align*}  First, we show that $\mS$ is a subtessellation of $\B$.  By construction, we have that $\mS \subset \B$.  To complete the proof that $\mS$ is a subtessellation of $\B$, it suffices to show that $\mS$ is the $(R''/R')$-tessellation of $\prod_{i=1}^d[a_i, a'_i].$  By Remark~\ref{Rmk:lemm:LocalCoordinates}, we have that $\prod_{i=1}^d[a_i, a'_i]$ is a closed ball given by $\|\cdot\|_\infty$ of radius $R''$.  Setting $x_i:=a_i+ R''$, we have that the center of this ball is the point $\boldsymbol{x}:=(x_1, \cdots, x_d)^t$.  Consequently, we have that $B(\boldsymbol{x}, R'') = \prod_{i=1}^d[a_i, a'_i].$

Let $B[(j_1, \cdots, j_d), \boldsymbol{x}]$ be an element of the $(R''/R')$-tessellation of $B(\boldsymbol{x}, R'')$.  Then we have that \[B[(j_1, \cdots, j_d), \boldsymbol{x}] =\prod_{i=1}^d\left[\left(\frac {2 (j_i-1) } {R''/R'}-1\right){R''}+x_{i}, x_{i}+ \left(\frac {2 j_i} {R''/R'}-1\right){R''} \right] \] for some integers $1 \leq j_1, \cdots,j_d\leq R''/R'$.

Consequently, we have \begin{align*} B[(j_1, \cdots, j_d), \boldsymbol{x}] &=\prod_{i=1}^d\left[2 (j_i-1) R'  +x_{i} -{R''},  2 j_i R'+  x_i - R'' \right] \\ &= \prod_{i=1}^d\left[2 (j_i-1) R'  +a_{i},  2 j_i R'+  a_i \right]\\ &= \prod_{i=1}^d\left[\left(2 (m_i+j_i-1)-1\right) R'  +x'_{i}, x'_i+ \left(2 (m_i+j_i) - 1\right) R' \right]
  \end{align*}
  
 \noindent Here, the third equality follows by (\ref{eqn:lemmLocalCoordinates}).  Setting $n_i = m_i+j_i-1$ for all $i$ and applying (\ref{eqn:lemmLocalCoordinates2}) give that the elements of $\mS$ and the elements of the $(R''/R')$-tessellation of $\prod_{i=1}^d[a_i, a'_i]$ are the same.  This shows that $\mS$ is a subtessellation of $\B$.
 
Now, by Lemma~\ref{lemm:TesselBallFund}, we have that \begin{align}\label{eqn:LemmLocalCoordinates:SIsExact}
 \bigcup_{B'' \in \mS} B'' = \prod_{i=1}^d[a_i, a'_i], \end{align} which, using the definition of the $\Ls_i(\boldsymbol{y})$, proves assertion (1).

To prove assertion (2), we assume that the conclusion does not hold, namely that there exists $B^* \in \mS$ such that $B^* \notin \mS'$.  Since $\mS'$ is a subtessellation, there exists a unique closed ball $B:=B(\widetilde{\boldsymbol{x}}, \widetilde{R})$ given by $\|\cdot\|_\infty$ for some $\widetilde{\boldsymbol{x}}:= (\widetilde{x}_1, \cdots, \widetilde{x}_d)^t \in \RR^d$ and $\widetilde{R}>0$ such that \[B =  \bigcup_{B'' \in \mS'} B''.\]  Consequently, Lemma~\ref{lemm:TesselEuclidFund} implies that $\inte(B^*) \cap B = \emptyset$.  Thus (\ref{eqn:LemmLocalCoordinates:SIsExact}) and the second assertion of Lemma~\ref{lemm:TesselBallFund} imply that $B \not\supset  \prod_{i=1}^d[a_i, a'_i]$.  

Now, for all integers $i$ such that $1 \leq i \leq d$, consider the chord $\sL_i(\boldsymbol{y}, B)$ of $B$.  Since $\Ls_i(\boldsymbol{y}) \subset B$, we have that $\Ls_i(\boldsymbol{y}) \subset \sL_i(\boldsymbol{y}, B)$, which implies that $-\widetilde{R}+ \widetilde{x}_i \leq a_i < a'_i \leq \widetilde{x}_i+ \widetilde{R}$.  Since this holds for all $i$, we have that $B \supset  \prod_{i=1}^d[a_i, a'_i]$, a contradiction.   This proves that assertion (2) holds.  If $\mS''$ a subtessellation for which assertions (1) and (2) hold, then assertion (2) gives that $\mS'' = \mS$.  Thus $\mS$ is unique.  This proves the lemma.
\end{proof}

\subsection{Minimal and maximal tessellations}\label{subsec:MinTess}

Let $d \in \NN$, $R \geq R' >0$, and $\boldsymbol{x}, \boldsymbol{x'} \in \RR^d$.  Fix a complete tessellation $\B:= \B_{\RR^d}(\boldsymbol{x'}, R')$.  All subtessellations in this section, Section~\ref{subsec:MinTess}, refer to subtessellations of $\B$ unless otherwise stated.  Note that the collection \[\col(\B):=\{\mS : \mS \textrm{ is a subtessellations of } \B\} \cup \{ \B\} \] is partially ordered by inclusion.  For a closed ball $B(\boldsymbol{x}, R)$ given by $\| \cdot \|_\infty$, we say $\mS \in \col(\B)$ {\em covers $B(\boldsymbol{x}, R)$} if \begin{align}\label{eqn:CompTessMeetsBall}
 B(\boldsymbol{x}, R) \subset \bigcup_{B \in \mS} B\end{align} holds and {\em is contained in $B(\boldsymbol{x}, R)$} if \begin{align}\label{eqn:CompTessMeetsBallInner}
 B(\boldsymbol{x}, R) \supset \bigcup_{B \in \mS} B\end{align} holds.  Define the subcollections \begin{align*}
\col(\B, B(\boldsymbol{x}, R))&:= \{\mS \in \col(\B): \mS \textrm{ covers } B(\boldsymbol{x}, R)\} \\ \col^*(\B, B(\boldsymbol{x}, R))&:= \{\mS \in \col(\B): \mS \textrm{ is contained in } B(\boldsymbol{x}, R)\}, \end{align*} and note, since $\B$ covers every $B(\boldsymbol{x}, R)$, we have that $\B \in \col(\B, B(\boldsymbol{x}, R))$ for every $B(\boldsymbol{x}, R)$.  Furthermore, we say an element $\mS \in \col(\B, B(\boldsymbol{x}, R))$ is a {\em minimal (outer) tessellation for $B(\boldsymbol{x}, R)$} if, for any $\mS' \in \col(\B, B(\boldsymbol{x}, R))$ such that $\mS'\subset \mS$, we have $\mS=\mS'$, and an element $\mS \in \col^*(\B, B(\boldsymbol{x}, R))$ is a {\em maximal (inner) tessellation for $B(\boldsymbol{x}, R)$} if, for any $\mS' \in \col^*(\B, B(\boldsymbol{x}, R))$ such that $\mS'\supset \mS$, we have $\mS=\mS'$.  Note that a minimal tessellation for $B(\boldsymbol{x}, R)$ is a minimal element for $\col(\B, B(\boldsymbol{x}, R))$ ordered by inclusion.\footnote{One could show that any (non-empty) descending chain has a lower bound, which implies the existence of minimal tessellations for $B(\boldsymbol{x}, R)$ by Zorn's lemma.  However, we do not do this because it is more convenient and useful for us to explicitly construct these minimal tessellations (see Remark~\ref{rmk:prop:RepsMinTess} and Section~\ref{subsubsec:ProofSecAssertProp:MinTess}).}  Analogously, a maximal tessellation for $B(\boldsymbol{x}, R)$ is a maximal element for $\col^*(\B, B(\boldsymbol{x}, R))$ (when nonempty) ordered by inclusion.  All collections $\col^*(\B, B(\boldsymbol{x}, R))$ under consideration in this paper will be nonempty.

\begin{lemm}\label{lemm:MaxTessNonEmptyCond}
Let $d \in \NN$, $R, R' >0$ such that $R \geq 2R'$, $\boldsymbol{x}, \boldsymbol{x'} \in \RR^d$, and $\B:=\B_{\RR^d}(\boldsymbol{x'}, R')$.  Then $\col^*(\B, B(\boldsymbol{x}, R)) \neq \emptyset$ and $\B \notin \col^*(\B, B(\boldsymbol{x}, R))$.
\end{lemm}
\begin{proof}
 By Lemma~\ref{lemm:TesselEuclidFund}, there exists a $\widetilde{B} \in \B$ such that $\boldsymbol{x} \in \widetilde{B}$.  Now, since $\rho(\widetilde{B})=R'$, we have that $\diam(\widetilde{B}) =2R'$.  Let $\boldsymbol{z} \in \widetilde{B}$.  Then $\|\boldsymbol{z}-\boldsymbol{x}\|_\infty \leq 2R'$, which gives that $\boldsymbol{z} \in B(\boldsymbol{x}, 2R') \subset B(\boldsymbol{x}, R)$.  Thus, $\{\widetilde{B}\} \in \col^*(\B, B(\boldsymbol{x}, R))$, yielding the first assertion.
 
Now if we assume the conclusion of the second assertion is false, then Lemma~\ref{lemm:TesselEuclidFund} yields that $\RR^d \subset  B(\boldsymbol{x}, R)$, which is a contradiction.  This proves the lemma.
\end{proof}

 Finally, we say $B(\boldsymbol{x}, R)$ is {\em representable (in $\B$)} if there exists an $\mS \in \col(\B, B(\boldsymbol{x}, R))$ for which we have equality in (\ref{eqn:CompTessMeetsBall}), or, equivalently, if there exists an $\mS \in \col^*(\B, B(\boldsymbol{x}, R))$ for which we have equality in (\ref{eqn:CompTessMeetsBallInner}).  We refer to this $\mS$ as its {\em representation (in $\B$)}.  Note that representations are always subtessellations of $\B$ and never $\B$ itself.
 
 \begin{lemm}\label{lemm:RepresentUniqSubTess}  Let $d \in \NN$, $R \geq R' >0$, and $\boldsymbol{x}, \boldsymbol{x'} \in \RR^d$.  If $B(\boldsymbol{x}, R)$ is representable in $\B:= \B_{\RR^d}(\boldsymbol{x'}, R')$, then its representation $\mS$ is a unique element of $\col(\B, B(\boldsymbol{x}, R))$ and a unique element of $\col^*(\B, B(\boldsymbol{x}, R))$. 
 
\end{lemm}

\begin{proof}
We first consider $\col(\B, B(\boldsymbol{x}, R))$.  Assume that the conclusion does not hold.  Let $\mS_1$ and $\mS_2$ be two distinct elements of $\col(\B, B(\boldsymbol{x}, R))$ that are representations of $B(\boldsymbol{x}, R)$.  Consequently, we have that \[\bigcup_{B \in \mS_1} B= B(\boldsymbol{x}, R) =  \bigcup_{B \in \mS_2} B.\]  Without loss of generality, we may assume that there exists a $B \in \mS_1 \backslash \mS_2$.   Hence, we have \[\inte(B) \subset \bigcup_{B \in \mS_2} B,\] from which it follows that $\inte(B) \cap \inte(B') \neq \emptyset$ for some $B'  \in \mS_2$.  Since $B \neq B'$, Lemma~\ref{lemm:TesselEuclidFund} yields a contraction.

The proof for $\col^*(\B, B(\boldsymbol{x}, R))$ is analogous.  This proves the lemma.
\end{proof}

We have two characterizations of representability.

\begin{lemm}\label{lemm:CharRepresentableSubTess} Let $d \in \NN$, $R \geq R' >0$, and $\boldsymbol{x}, \boldsymbol{x'} \in \RR^d$.  Then $B(\boldsymbol{x}, R)$ is representable in $\B:=\B_{\RR^d}(\boldsymbol{x'}, R')$ if and only if $\begin{cases}  R/R' \textrm{ is an odd integer and } \boldsymbol{x} \equiv \boldsymbol{x'} \mod 2R'\ZZ^d\\ R/R' \textrm{ is an even integer and } \boldsymbol{x} \equiv \boldsymbol{x'}+ (R', \cdots, R')^t\mod 2R'\ZZ^d \end{cases}.$  Also, when $B(\boldsymbol{x}, R)$ is representable in $\B$, its representation is the $(R/R')$-tessellation of $B(\boldsymbol{x}, R)$.
 
\end{lemm}
 \begin{proof}  We first prove the forward implication.  By Lemma~\ref{lemm:RepresentUniqSubTess}, the representation $\mS$ of $B(\boldsymbol{x}, R)$ is a unique element of $\col(\B, B(\boldsymbol{x}, R))$.  Consequently, $\mS$ is a subtessellation of $\B$ and, thus, is an $M$-tessellation of a closed ball given by $\|\cdot\|_\infty$ for some $M \in \NN$.  By Lemma~\ref{lemm:TesselBallFund}, that closed ball given by $\|\cdot\|_\infty$ is $B(\boldsymbol{x}, R)$.  Since $\mS$ is an $M$-tessellation, every element of $\mS$ has radius $R/M$.  Since $\overline{\mS}= \B$, every element of $\mS$ has radius $R'$.  Thus, $R/M = R'$ and $R/R' \in \NN$.
 
 If $R/R'$ is odd, we have that $\overline{\mS} = \B_{\RR^d}(\boldsymbol{x}, R')$ by construction of the completion and Lemma~\ref{lemm:UniqCompletion}.  Since $\overline{\mS}= \B$, Lemma~\ref{lemm:CompTessExtrapol} gives that $\boldsymbol{x} \equiv \boldsymbol{x'}\mod 2R'\ZZ^d$.
 
 Otherwise, $R/R'$ is even, and we have that  $\overline{\mS} = \B_{\RR^d}(\boldsymbol{x}- (R', \cdots, R')^t, R')$ by construction of the completion and Lemma~\ref{lemm:UniqCompletion}.  Since $\overline{\mS}= \B$, Lemma~\ref{lemm:CompTessExtrapol} gives that $\boldsymbol{x} \equiv \boldsymbol{x'}+ (R', \cdots, R')^t\mod 2R'\ZZ^d$.  This proves the forward implication.
 
 We now prove the reverse implication.  As $R \geq R' >0$, we have that $R/R' \geq 1$.  Let $M:=R/R' \in \NN$.  Then Lemma~\ref{lemm:RefineTess} gives that $\mR_M(B(\boldsymbol{x}, R))$ is the $M$-tessellation of $B(\boldsymbol{x}, R)$.   If $M$ is odd, then $\overline{\mR_M(B(\boldsymbol{x}, R))} = \B_{\RR^d}(\boldsymbol{x}, R/M)= \B_{\RR^d}(\boldsymbol{x}, R')$ by construction of the completion.  When $M$ is odd, we have the assumption that $\boldsymbol{x} \equiv \boldsymbol{x'} \mod 2R'\ZZ^d$, using which Lemma~\ref{lemm:CompTessExtrapol} gives that $\overline{\mR_M(B(\boldsymbol{x}, R))}= \B_{\RR^d}(\boldsymbol{x'}, R')=\B$.  Now Lemma~\ref{lemm:UniqCompletion} gives that $\mR_M(B(\boldsymbol{x}, R)$ is a subtessellation of $\B$, and Lemma~\ref{lemm:TesselBallFund} implies that $\mR_M(B(\boldsymbol{x}, R))$ is the representation of $B(\boldsymbol{x}, R)$.  Consequently, $B(\boldsymbol{x}, R)$ is representable, and $\mR_M(B(\boldsymbol{x}, R))$ is its representation.

 If $M$ is even, then $\overline{\mR_M(B(\boldsymbol{x}, R))} = \B_{\RR^d}(\boldsymbol{x}- (R', \cdots, R')^t, R')$ by construction of the completion.  When $M$ is even, we have the assumption that $\boldsymbol{x} \equiv \boldsymbol{x'}+ (R', \cdots, R')^t \mod 2R'\ZZ^d$, using which Lemma~\ref{lemm:CompTessExtrapol} gives that $\overline{\mR_M(B(\boldsymbol{x}, R))}= \B_{\RR^d}(\boldsymbol{x'}, R')=\B$.  As in the case for $M$ odd, we have that $B(\boldsymbol{x}, R)$ is representable, and $\mR_M(B(\boldsymbol{x}, R))$ is its representation.  This proves the reverse implication.

For the final assertion, we have shown that, when $B(\boldsymbol{x}, R)$ is representable, its representation is the $(R/R')$-tessellation of $B(\boldsymbol{x}, R)$.  This completes the proof the lemma.
 \end{proof}

\begin{lemm}\label{lemm:ChordCharRepresentable} Let $d \in \NN$, $R \geq R' >0$, and $\boldsymbol{x}, \boldsymbol{x'} \in \RR^d$.  Then $B(\boldsymbol{x}, R)$ is representable in $\B:=\B_{\RR^d}(\boldsymbol{x'}, R')$ if and only if there exists a \begin{align}\label{eqn:lemm:ChordCharRepresentable}
 \boldsymbol{y} \in \inte(B(\boldsymbol{x}, R)) \cap \bigcup_{B'' \in \B} \inte(B'') \end{align} such that the family $\{\sL_1(\boldsymbol{y}, B(\boldsymbol{x}, R)), \cdots, \sL_d(\boldsymbol{y}, B(\boldsymbol{x}, R)) \}$ is a uniform local coordinate system at $\boldsymbol{y}$ satisfying (\ref{eqn:lemm:LocalCoordinates}).
\end{lemm}
\begin{proof}  Let $\sL_i(\boldsymbol{y}):=\sL_i(\boldsymbol{y}, B(\boldsymbol{x}, R))$ for all $i$, and let $\lambda(\cdot)$ denote the Lebesgue measure on $\RR^d$.  We first prove the forward implication. Since \[\lambda\left(\RR^d \backslash \bigcup_{B'' \in \B} \inte(B'')\right)=0,\] we have that \[\lambda\left(\inte(B(\boldsymbol{x}, R)) \cap \bigcup_{B'' \in \B} \inte(B'')\right) = (2R)^d >0.\]  Consequently, we can and do pick a $\boldsymbol{y} \in \RR^d$ such that (\ref{eqn:lemm:ChordCharRepresentable}) holds.

Now recall that the chords of $B(\boldsymbol{x}, R)$ all have length $2R$, and, thus, the family $\{\sL_1(\boldsymbol{y}), \cdots, \sL_d(\boldsymbol{y})\}$ is a uniform local coordinate system at $\boldsymbol{y}$ of size $2R$.   Since $B(\boldsymbol{x}, R)$ is representable, we have that there exists $\mS \in \col(\B, B(\boldsymbol{x}, R))$ such that \[B(\boldsymbol{x}, R) =  \bigcup_{B'' \in \mS} B''\] and, hence, we have that \[\bigcup_{B'' \in \mS} \inte(B'') \subset \inte(B(\boldsymbol{x}, R)).\]  Thus, we have that \[B(\boldsymbol{x}, R) \backslash \inte(B(\boldsymbol{x}, R)) = \left(\bigcup_{B'' \in \mS} B''\right) \backslash \inte(B(\boldsymbol{x}, R))  \subset \bigcup_{B'' \in \mS} B'' \backslash \inte(B'').\]  Since $(\ref{eqn:lemm:ChordCharRepresentable})$ holds, Lemma~\ref{lemm:ChordsThruInteriorPointsChar} gives that $ \boldsymbol{e}\left(\sL_i(\boldsymbol{y})\right), \boldsymbol{e'}\left(\sL_i(\boldsymbol{y})\right) \in B(\boldsymbol{x}, R) \backslash \inte(B(\boldsymbol{x}, R))$ for all $i$. Thus, we obtain that the uniform local coordinate system $\{\sL_1(\boldsymbol{y}), \cdots, \sL_d(\boldsymbol{y}) \}$ satisfies (\ref{eqn:lemm:LocalCoordinates}).  This proves the forward implication.

We now prove the reverse implication.  Since the chords of $B(\boldsymbol{x}, R)$ all have length $2R$, the family $\{\sL_1(\boldsymbol{y}), \cdots, \sL_d(\boldsymbol{y})\}$ is a uniform local coordinate system at $\boldsymbol{y}$ of size $2R$.  Also, by (\ref{eqn:EndPointsChordDefn}), we have that \begin{align*}\boldsymbol{e}\left(\sL_i(\boldsymbol{y})\right) &= (y_1, \cdots, y_{i-1}, x_i-R, y_{i+1}, \cdots, y_{d}) \textrm { and } \\  \boldsymbol{e'}\left(\sL_i(\boldsymbol{y})\right)&=(y_1, \cdots, y_{i-1}, x_i+R, y_{i+1}, \cdots, y_{d}) \end{align*} for all $i$.  Since (\ref{eqn:lemm:LocalCoordinates}) holds, Lemma~\ref{lemm:LocalCoordinates} gives a unique subtessellation $\mS$ of $\B$ such that \[\sL_1(\boldsymbol{y}) \cup  \cdots \cup \sL_d(\boldsymbol{y}) \subset \bigcup_{B'' \in \mS} B''  \] holds and that $\mS$ is the $(R/R')$-tessellation of $\prod_{i=1}^d [ -R+ x_i, x_i+R] = B(\boldsymbol{x}, R)$.  By Lemma~\ref{lemm:TesselBallFund}, we have that \[\bigcup_{B'' \in \mS} B'' = B(\boldsymbol{x}, R).\]  Thus $B(\boldsymbol{x}, R)$ is representable as desired.  This proves the lemma.
\end{proof}

\begin{prop}\label{prop:RepsMinTess} Let $d \in \NN$, $R \geq R' >0$, and $\boldsymbol{x}, \boldsymbol{x'} \in \RR^d$.  If \begin{itemize}
\item $B(\boldsymbol{x}, R)$ is representable in $\B:=\B_{\RR^d}(\boldsymbol{x'}, R')$, then the representation $\mS$ of $B(\boldsymbol{x}, R)$ is the unique minimal tessellation for $B(\boldsymbol{x}, R)$, $R/R' \in \NN$, and $|\mS|= (R/R')^d$.
\item $B(\boldsymbol{x}, R)$ is not representable in $\B:=\B_{\RR^d}(\boldsymbol{x'}, R')$, then there exist finitely many minimal tessellations $\mS$ for $B(\boldsymbol{x}, R)$, all of which satisfy \[\begin{cases} |\mS| = (\lceil R/R' \rceil)^d \text{ or }  |\mS| = (\lceil R/R' \rceil+1)^d  & \text{ if } R/R' \notin \NN \\  |\mS| =  (R/R' +1)^d & \text{ if } R/R' \in \NN\end{cases}.\]\end{itemize}

\end{prop}
\begin{rema}\label{rmk:prop:RepsMinTess}
When $B(\boldsymbol{x}, R)$ is not representable in $\B$, it is also possible for the minimal tessellation of $B(\boldsymbol{x}, R)$ to be unique, but this is not guaranteed.  In our proof of the second assertion (see Section~\ref{subsubsec:ProofSecAssertProp:MinTess}), we explicitly construct a minimal tessellation of $B(\boldsymbol{x}, R)$ and show that any minimal tessellation of $B(\boldsymbol{x}, R)$ is constructed in a similar way by making various choices (see~ (\ref{eqn:RepsMinTess:ChoicebCaseAi}), \ref{eqn:RepsMinTess:ChoicebCaseB}) and the end of Section~\ref{subsubsec:ProofSecAssertProp:MinTess}).  Consequently, our proof of the second assertion gives explicit constructions for any minimal tessellation of $B(\boldsymbol{x}, R)$.

\end{rema}

\begin{rema}\label{rmk:proofPropMinTess2ndPart} The maximum distance with respect to $\|\cdot\|_\infty$ of any point in a minimal tessellation $\mS$ of $B(\boldsymbol{x}, R)$ to some point in $B(\boldsymbol{x}, R)$ is less than $4R'$.  If $B(\boldsymbol{x}, R)$ is representable, this follows by (\ref{eqn:CompTessMeetsBall}) and the definition of representable.  Otherwise, this follows by the proof of the second assertion of Proposition~\ref{prop:RepsMinTess} (see Section~\ref{subsubsec:ProofSecAssertProp:MinTess}): precisely, it follows by (\ref{eqn:prop:RepsMinTess:kappasize}, \ref{eqn:prop:RepsMinTess:kappasizeCase3}, \ref{eqn:prop:RepsMinTess:LDefn}), that the two possible values of $L_i$ are $L$ and $L-2R'$ (namely Cases (I) and (II) in Section~\ref{subsubsec:ProofSecAssertProp:MinTess}), and that any minimal tessellation comes from the construction in Section~\ref{subsubsec:ProofSecAssertProp:MinTess} (namely, as noted at the end of Section~\ref{subsubsec:ProofSecAssertProp:MinTess}, a choice of $\widehat{\boldsymbol{b}}$ or $\widecheck{\boldsymbol{b}}$ for each coordinate $i$).  
\end{rema}

\begin{prop}\label{prop:RepsMaxTess} Let $d \in \NN$, $R, R' >0$ such that \begin{align}\label{eqn:prop:RepsMaxTess:CondRRPrime}
 R \geq 2R', \end{align} $\boldsymbol{x}, \boldsymbol{x'} \in \RR^d$, and $\B:=\B_{\RR^d}(\boldsymbol{x'}, R')$.  If \begin{itemize}
\item $B(\boldsymbol{x}, R)$ is representable in $\B$, then the representation $\mS$ of $B(\boldsymbol{x}, R)$ is the unique maximal tessellation for $B(\boldsymbol{x}, R)$, $R/R' \in \NN$, and $|\mS|= (R/R')^d$.
\item $B(\boldsymbol{x}, R)$ is not representable in $\B$, then there exist finitely many maximal tessellations $\mS$ for $B(\boldsymbol{x}, R)$, all of which satisfy \[\begin{cases} |\mS| = \left(\lfloor R/R' \rfloor\right)^d \text{ or }  |\mS| = \left(\lfloor R/R' \rfloor-1\right)^d  & \text{ if } R/R' \notin \NN \\  |\mS| =  \left(R/R' -1\right)^d & \text{ if } R/R' \in \NN\end{cases}.\]\end{itemize}

\end{prop}
\begin{rema}\label{rmk:prop:RepsMaxTess}
When $B(\boldsymbol{x}, R)$ is not representable in $\B$, it is also possible for the maximal tessellation of $B(\boldsymbol{x}, R)$ to be unique, but this is not guaranteed.

\end{rema}

\subsubsection{Proof of the first assertion of Proposition~\ref{prop:RepsMinTess}}\label{subsubsec:ProofFirstAssertProp:MinTess},
 Let $\mS' \in \col(\B, B(\boldsymbol{x}, R))$ such that $\mS'\subset \mS$.  Since every element of $\mS'$ is also an element of $\mS$, we have that \[B(\boldsymbol{x}, R) \subset \bigcup_{B \in \mS'} B \subset \bigcup_{B \in \mS} B =  B(\boldsymbol{x}, R),\] which implies that $\mS'$ is another representation of $B(\boldsymbol{x}, R)$.  Now Lemma~\ref{lemm:RepresentUniqSubTess}  gives that $\mS' = \mS$, which implies that $\mS$ is a minimal tessellation of $B(\boldsymbol{x}, R)$.  This proves the existence of a minimal tessellation of $B(\boldsymbol{x}, R)$ for the first assertion.

We now prove the uniqueness of the minimal tessellation of $B(\boldsymbol{x}, R)$ for the first assertion.  Let $\mS''$ be a minimal tessellation of $B(\boldsymbol{x}, R)$.  Since $\mS''$ covers $B(\boldsymbol{x}, R)$ and $\mS$ is a representation of $B(\boldsymbol{x}, R)$, we have that \[\bigcup_{B \in \mS} B = B(\boldsymbol{x}, R) \subset \bigcup_{B \in \mS''} B.\]  We claim that $\mS \subset \mS''$.  We now prove the claim.  Assume that the conclusion does not hold.  Then there exists an element $\widehat{B} \in \mS \backslash \mS''$.  Thus, we have that \[\inte(\widehat{B}) \subset B(\boldsymbol{x}, R)\subset \bigcup_{B \in \mS''} B,\] which implies that $\inte(\widehat{B}) \cap \widecheck{B} \neq \emptyset$ for some  $\widecheck{B} \in \mS''$.  Since both $\mS$ and $\mS''$ are subtessellations of $\B$, Lemma~\ref{lemm:TesselEuclidFund} yields a contraction and proves the claim.  Since $\mS''$ is a minimal tessellation of $B(\boldsymbol{x}, R)$, we have $\mS = \mS''$.  This proves the uniqueness of the minimal tessellation of $B(\boldsymbol{x}, R)$ for the first assertion.  By Lemma~\ref{lemm:CharRepresentableSubTess}, we have that $R/R' \in \NN$ and $\mS$ is the $(R/R')$-tessellation of $B(\boldsymbol{x}, R)$, which gives that $|\mS|= (R/R')^d$.  This completes the proof of the first assertion.

\subsubsection{Proof of the second assertion of Proposition~\ref{prop:RepsMinTess}}\label{subsubsec:ProofSecAssertProp:MinTess}
Let \[\partial(\B):=\bigcup_{B'' \in \B}\left(B'' \backslash \inte(B'') \right), \quad \B^{\circ} := \bigcup_{B'' \in \B} \inte(B'')\] and $ \boldsymbol{x} =: (x_1, \cdots, x_d)^t$.  Pick a \begin{align}\label{eqn:prop:RepsMinTess:Defny}
 \boldsymbol{y}:=(y_1, \cdots, y_d)^t \in \inte(B(\boldsymbol{x}, R)) \cap \B^{\circ},\end{align} and let $i$ be an integer such that $1 \leq i \leq d$.  Since $\sL_i(\boldsymbol{y}):=\sL_i(\boldsymbol{y}, B(\boldsymbol{x}, R))$ is the chord of $B(\boldsymbol{x}, R)$ parallel to the $i$-axis through the point $\boldsymbol{y}$, we have that the length of $\sL_i(\boldsymbol{y})$ is $2R$ for all $i$.  Consequently, the family $\{\sL_1(\boldsymbol{y}), \cdots, \sL_d(\boldsymbol{y}) \}$ is a uniform local coordinate system at $\boldsymbol{y}$ of size $2R$, which, by Lemma~\ref{lemm:ChordCharRepresentable}, does not satisfy (\ref{eqn:lemm:LocalCoordinates}).  Hence, there exists an integer $j$ such that $1 \leq j \leq d$ for which $\boldsymbol{e}\left(\sL_j(\boldsymbol{y})\right) \in  \B^{\circ}$ or $\boldsymbol{e'}\left(\sL_j(\boldsymbol{y})\right) \in  \B^{\circ}$.  Note that, by (\ref{eqn:EndPointsChordDefn}), we have \begin{align}\label{eqn:prop:RepsMinTess:EndPointsChords}
\boldsymbol{e}\left(\sL_j(\boldsymbol{y})\right) &= (y_1, \cdots, y_{j-1}, x_j-R, y_{j+1},\cdots, y_d)^t, \\\nonumber \boldsymbol{e'}\left(\sL_j(\boldsymbol{y})\right) &= (y_1, \cdots, y_{j-1}, x_j+R, y_{j+1},\cdots, y_d)^t.   \end{align}  There are three cases \begin{enumerate}
\item $\boldsymbol{e}\left(\sL_j(\boldsymbol{y})\right) \in  \B^{\circ}$ and $\boldsymbol{e'}\left(\sL_j(\boldsymbol{y})\right) \in  \partial(\B)$,
\item $\boldsymbol{e}\left(\sL_j(\boldsymbol{y})\right) \in \partial(\B)$ and $\boldsymbol{e'}\left(\sL_j(\boldsymbol{y})\right) \in  \B^{\circ}$,
\item $\boldsymbol{e}\left(\sL_j(\boldsymbol{y})\right) \in  \B^{\circ}$ and $\boldsymbol{e'}\left(\sL_j(\boldsymbol{y})\right) \in   \B^{\circ}$
\end{enumerate} to consider. 

First, we consider Case~(1).  Let us extend $\sL_j(\boldsymbol{y})$ to the superset \[\widetilde{\sL_j}(\boldsymbol{y}):=\{y_1\} \times \cdots \times \{y_{j-1}\} \times [-3R+x_j, x_j+R] \times \{y_{j+1}\} \times \cdots \times \{y_{d}\}.\]  Now define the line segment \[\widehat{\sL_j}(\boldsymbol{y}):=\{y_1\} \times \cdots \times \{y_{j-1}\} \times [-3R+x_j, x_j-R] \times \{y_{j+1}\} \times \cdots \times \{y_{d}\}\] and note that $\widetilde{\sL_j}(\boldsymbol{y}) = \sL_j(\boldsymbol{y}) \cup \widehat{\sL_j}(\boldsymbol{y})$ and $\sL_j(\boldsymbol{y}) \cap \widehat{\sL_j}(\boldsymbol{y})= \boldsymbol{e}\left(\sL_j(\boldsymbol{y})\right)$.  Now, since $\boldsymbol{e}\left(\sL_j(\boldsymbol{y})\right) \in  \B^{\circ}$, we have that \begin{align}\label{eqn:prop:RepsMinTess}
 \boldsymbol{e}\left(\sL_j(\boldsymbol{y})\right) \notin \sP(\sL_j(\boldsymbol{y})) \cup \sP(\widehat{\sL_j}(\boldsymbol{y})). \end{align}  Applying Lemma~\ref{lemm:PartitionLineSeg} to these three line segments yields that \begin{align}\label{eqn:prop:RepsMinTessDisjoint}
 \sP(\widetilde{\sL_j}(\boldsymbol{y})) =\sP(\widehat{\sL_j}(\boldsymbol{y})) \sqcup \sP(\sL_j(\boldsymbol{y})) \end{align} and that the distance between any two consecutive elements is $2R'$.  Now since the length of $\widehat{\sL_j}(\boldsymbol{y}))$ is equal to $2R$ and $R\geq R'$, the second assertion of Lemma~\ref{lemm:PartitionLineSeg} gives that $|\sP(\widehat{\sL_j}(\boldsymbol{y}))|\geq 1$.  Moreover, since $\boldsymbol{e'}\left(\sL_j(\boldsymbol{y})\right) \in  \partial(\B)$, we have that $|\sP(\sL_j(\boldsymbol{y}))|\geq 1$.

Now let \[\boldsymbol{b'}:=(y_1, \cdots, y_{j-1}, b_j', y_{j+1},\cdots, y_d)^t \in \sP(\sL_j(\boldsymbol{y}))\] be the unique element of $\sP(\sL_j(\boldsymbol{y}))$ nearest to $\boldsymbol{e}\left(\sL_j(\boldsymbol{y})\right)$.  By (\ref{eqn:prop:RepsMinTess}) and the definitions of $\boldsymbol{e}(\cdot)$ and $\sP(\cdot)$, we have $x_j-R < b_j'$. Note also that the uniqueness of $\boldsymbol{b'}$ follows from the definitions of $\boldsymbol{e}(\cdot)$ and $\sP(\cdot)$.

Now let \[\boldsymbol{b}:=(y_1, \cdots, y_{j-1}, b_j, y_{j+1},\cdots, y_d)^t \in \sP(\widehat{\sL_j}(\boldsymbol{y}))\] be the unique element of $\sP(\widehat{\sL_j}(\boldsymbol{y}))\}$ nearest to $\boldsymbol{e}\left(\sL_j(\boldsymbol{y})\right)$.  Since $\boldsymbol{e'}\left(\widehat{\sL_j}(\boldsymbol{y})\right)= \boldsymbol{e}\left(\sL_j(\boldsymbol{y})\right)$, we have, by  (\ref{eqn:prop:RepsMinTess}) and the definitions of $\boldsymbol{e'}(\cdot)$ and $\sP(\cdot)$, that $b_j< x_j-R$.  Note that the uniqueness of $\boldsymbol{b'}$ follow from the definitions of $\boldsymbol{e'}(\cdot)$ and $\sP(\cdot)$.  

By (\ref{eqn:prop:RepsMinTessDisjoint}), we have that $\boldsymbol{b}$ and $\boldsymbol{b'}$ are distinct elements of $\sP(\widetilde{\sL_j}(\boldsymbol{y}))$.  We claim that $\boldsymbol{b}$ and $\boldsymbol{b'}$ are consecutive elements of $\sP(\widetilde{\sL_j}(\boldsymbol{y}))$.  We now prove the claim.  Assume that the conclusion does not hold.  Then there exists \[\boldsymbol{b''}:=(y_1, \cdots, y_{j-1}, b_j'', y_{j+1},\cdots, y_d)^t \in \sP(\widetilde{\sL_j}(\boldsymbol{y}))\] such that $b_j<  b_j'' < b_j'$.  We have already shown that $b_j<  x_j-R < b_j'$.  By (\ref{eqn:prop:RepsMinTess}), we have that either $x_j-R < b_j''$ or $b_j'' < x_j-R$.  If $x_j-R < b_j''$ holds, then $\boldsymbol{b''}$ is an element of $\sP(\sL_j(\boldsymbol{y}))$ for which $\|\boldsymbol{e}\left(\sL_j(\boldsymbol{y})\right) -\boldsymbol{b''}\|_\infty < \|\boldsymbol{e}\left(\sL_j(\boldsymbol{y})\right) -\boldsymbol{b'}\|_\infty$, yielding a contradiction.  If $b_j'' < x_j-R$ holds, we also obtain a contradiction in analogous way.  This proves the claim.

Lemma~\ref{lemm:PartitionLineSeg} gives that consecutive elements of $\sP(\widetilde{\sL_j}(\boldsymbol{y}))$ have distance $2R'$ and, consequently, we have shown that \begin{align}\label{eqn:prop:RepsMinTess:BndsForxj}
 b_j < x_j -R < b_j' \quad \textrm{ and }\quad  b_j'-b_j= 2R' \end{align} hold.  Now form the line segment \[\widecheck{\sL}_j(\boldsymbol{y}):= \{y_1\} \times \cdots \times \{y_{j-1}\} \times [b_j, x_j+R] \times \{y_{j+1}\} \times \cdots \times \{y_{d}\}.\]  Thus we have that $\boldsymbol{e}\left(\widecheck{\sL}_j(\boldsymbol{y})\right), \boldsymbol{e'}\left(\widecheck{\sL}_j(\boldsymbol{y})\right) \in \partial(\B)$. Setting $\kappa_j:=x_j-R - b_j$, we have that $\widecheck{\sL}_j(\boldsymbol{y})$ has length \begin{align}\label{eqn:prop:RepsMinTess:Lj}
 L_j:= 2R + \kappa_j  \end{align}and \begin{align}\label{eqn:prop:RepsMinTess:kappasize}
 0<\frac{\kappa_j}{2R'} < 1 \end{align} holds.

Now, since $\boldsymbol{e}\left(\widecheck{\sL}_j(\boldsymbol{y})\right), \boldsymbol{e'}\left(\widecheck{\sL}_j(\boldsymbol{y})\right) \in \sP(\widetilde{\sL_j}(\boldsymbol{y}))$, we have, by Lemma~\ref{lemm:PartitionLineSeg}, that $L_j = 2 n R'$ for some $n \in \NN$.  Consequently, for Case~(1), $R/R'$ can not be an integer because, if it were, then $\kappa_j$ would be an integer multiple of $2R'$, contradicting (\ref{eqn:prop:RepsMinTess:kappasize}).  Finally, we claim that \begin{align}\label{eqn:prop:RepsMinTess:kappasizeCase1LRBnd} \frac{L_j} {2 R'} = \left\lceil \frac {R}{R'} \right\rceil
  \end{align} holds.  We now prove this claim.  By (\ref{eqn:prop:RepsMinTess:Lj}, \ref{eqn:prop:RepsMinTess:kappasize}), we have that $R/R'<L_j/2R' < R/R'+1$.  Since $L_j/2R'$ is an integer and $R/R'$ is not an integer, the claim follows.

Next, we consider Case~(2).  Define the line segment \[\widehat{\sL_j}(\boldsymbol{y}):=\{y_1\} \times \cdots \times \{y_{j-1}\} \times [R+x_j, x_j+3R] \times \{y_{j+1}\} \times \cdots \times \{y_{d}\}.\]  By the analog of the proof of Case~(1), we have that there exists an unique element \[\boldsymbol{b'}:=(y_1, \cdots, y_{j-1}, b_j', y_{j+1},\cdots, y_d)^t\] of $\sP(\widehat{\sL_j}(\boldsymbol{y}))$ nearest to $\boldsymbol{e'}\left(\sL_j(\boldsymbol{y})\right)$.  Now form the line segment \[\widecheck{\sL}_j(\boldsymbol{y}):= \{y_1\} \times \cdots \times \{y_{j-1}\} \times [-R+x_j, b_j'] \times \{y_{j+1}\} \times \cdots \times \{y_{d}\}\] and set $\kappa_j:= b_j'- (x_j+R)$.  By the analog of the proof of Case~(1), we have that (\ref{eqn:prop:RepsMinTess:Lj}, \ref{eqn:prop:RepsMinTess:kappasize}, \ref{eqn:prop:RepsMinTess:kappasizeCase1LRBnd} ) hold.  For the same reason as in Case~(1), $R/R'$ can not be an integer for Case~(2).

 Finally, we consider Case~(3).  Since $R \geq R'$, Lemma~\ref{lemm:PartitionLineSeg} gives that $|\sP({\sL_j}(\boldsymbol{y}))| \geq 1$.  Let $\boldsymbol{c}:=(c_1, \cdots, c_d)^t \in \sP({\sL_j}(\boldsymbol{y}))$.  Replacing $\sL_j(\boldsymbol{y})$ with \[\sL_j^{-}(\boldsymbol{y}):= \{y_1\} \times \cdots \times \{y_{j-1}\} \times [-R+x_j, c_j] \times \{y_{j+1}\} \times \cdots \times \{y_{d}\}\] and $\widetilde{\sL_j}(\boldsymbol{y})$ with \[\widetilde{\sL_j}^{-}(\boldsymbol{y}):=\{y_1\} \times \cdots \times \{y_{j-1}\} \times [-3R+x_j, c_j] \times \{y_{j+1}\} \times \cdots \times \{y_{d}\}\] in the proof of Case~(1) yields the uniqueness and existence of two consecutive elements \begin{align*}
\boldsymbol{b}:=&(y_1, \cdots, y_{j-1}, b_j, y_{j+1},\cdots, y_d)^t, \\ \boldsymbol{b'}:=&(y_1, \cdots, y_{j-1}, b_j', y_{j+1},\cdots, y_d)^t  \end{align*} of $\sP\left(\widetilde{\sL_j}^{-}(\boldsymbol{y})\right)$ such that (\ref{eqn:prop:RepsMinTess:BndsForxj}) holds.  Form the line segment \[\widecheck{\sL}^-_j(\boldsymbol{y}):= \{y_1\} \times \cdots \times \{y_{j-1}\} \times [b_j, c_j] \times \{y_{j+1}\} \times \cdots \times \{y_{d}\}\] and set $\kappa^-_j:=x_j-R - b_j$.

Now replacing $\sL_j(\boldsymbol{y})$ with \[\sL_j^{+}(\boldsymbol{y}):= \{y_1\} \times \cdots \times \{y_{j-1}\} \times [c_j, x_j+R] \times \{y_{j+1}\} \times \cdots \times \{y_{d}\}\] and $\widetilde{\sL_j}(\boldsymbol{y})$ with \[\widetilde{\sL_j}^{+}(\boldsymbol{y}):=\{y_1\} \times \cdots \times \{y_{j-1}\} \times [c_j, x_j+3R ] \times \{y_{j+1}\} \times \cdots \times \{y_{d}\}\] in the proof of Case~(2) yields the uniqueness and existence of two consecutive elements \begin{align*}
\boldsymbol{b''}:=&(y_1, \cdots, y_{j-1}, b_j'', y_{j+1},\cdots, y_d)^t, \\ \boldsymbol{b'''}:=&(y_1, \cdots, y_{j-1}, b_j''', y_{j+1},\cdots, y_d)^t  \end{align*} of $\sP\left(\widetilde{\sL_j}^{+}(\boldsymbol{y})\right)$ such that  \begin{align}\label{eqn:prop:RepsMinTess:BndsForxj2}
 b''_j < x_j +R < b_j''' \quad \textrm{ and }\quad  b_j'''-b_j''= 2R' \end{align} hold.  Form the line segment \[\widecheck{\sL}^+_j(\boldsymbol{y}):= \{y_1\} \times \cdots \times \{y_{j-1}\} \times [c_j, b_j'''] \times \{y_{j+1}\} \times \cdots \times \{y_{d}\}\] and set $\kappa^+_j:=b_j''' - (x_j+R)$.  Using (\ref{eqn:prop:RepsMinTess:BndsForxj}, \ref{eqn:prop:RepsMinTess:BndsForxj2}), we have that the line segment \[\widecheck{\sL}_j(\boldsymbol{y}):=\widecheck{\sL}^+_j(\boldsymbol{y}) \cup \widecheck{\sL}^-_j(\boldsymbol{y}) =  \{y_1\} \times \cdots \times \{y_{j-1}\} \times [b_j, b_j'''] \times \{y_{j+1}\} \times \cdots \times \{y_{d}\}\] has length \begin{align}\label{eqn:prop:RepsMinTess:LjCase3}
 L_j:= 2R + \kappa^-_j+\kappa^+_j \end{align} and \begin{align}\label{eqn:prop:RepsMinTess:kappasizeCase3}
 0<\frac{\kappa^-_j}{2R'} < 1, \quad 0<\frac{\kappa^+_j}{2R'} < 1\end{align} hold.  Now, since $\boldsymbol{e}\left(\widecheck{\sL}_j(\boldsymbol{y})\right), \boldsymbol{e'}\left(\widecheck{\sL}_j(\boldsymbol{y})\right) \in \sP\left(\widetilde{\sL_j}^-(\boldsymbol{y}) \cup \widetilde{\sL_j}^+(\boldsymbol{y})\right)$, we have, by Lemma~\ref{lemm:PartitionLineSeg}, that $L_j = 2 n R'$ for some $n \in \NN$.

 Unlike for Cases~(1) and~(2), $R/R'$ could be either an integer or not.  Therefore, Case~(3) splits into two further cases: \begin{enumerate}
\item[(3a)] $\boldsymbol{e}\left(\sL_j(\boldsymbol{y})\right) \in  \B^{\circ}$, $\boldsymbol{e'}\left(\sL_j(\boldsymbol{y})\right) \in   \B^{\circ}$, and $R/R' \in \NN$,
\item[(3b)] $\boldsymbol{e}\left(\sL_j(\boldsymbol{y})\right) \in  \B^{\circ}$,  $\boldsymbol{e'}\left(\sL_j(\boldsymbol{y})\right) \in   \B^{\circ}$, and $R/R' \notin \NN$.
\end{enumerate}  First, we consider Case~(3a).  We claim that \begin{align}\label{eqn:prop:RepsMinTess:kappasizeCase3aBnd} \kappa^-_j + \kappa^+_j = 2R' \quad \textrm{ and } \quad \frac{L_j} {2 R'} =  \frac {R}{R'} +1.
  \end{align}  We now prove the claim.  Since $R/R' \in \NN$, we have that $\kappa^-_j + \kappa^+_j$ is an integer multiple of $2R'$ by (\ref{eqn:prop:RepsMinTess:LjCase3}), then the claim follows by (\ref{eqn:prop:RepsMinTess:kappasizeCase3}).

Next, we consider Case~(3b).  We claim that \begin{align}\label{eqn:prop:RepsMinTess:kappasizeCase3bBnd}  \frac{L_j} {2 R'} = \left\lceil \frac {R}{R'} \right\rceil \quad \textrm{ or }  \quad\frac{L_j} {2 R'} = \left\lceil \frac {R}{R'} \right\rceil+1
  \end{align}    We now prove this claim.  By (\ref{eqn:prop:RepsMinTess:LjCase3}, \ref{eqn:prop:RepsMinTess:kappasizeCase3}), we have that $R/R'<L_j/2R' < R/R'+2$.  Since $R/R'$ is not an integer, the open interval $(R/R', R/R'+2)$ contains exactly two integers $\left\lceil R/R' \right\rceil$ and $\left\lceil R/R' \right\rceil+1$.  Since $L_j/2R'$ is an integer, the claim follows.  This completes Case~(3).

We now construct a uniform local coordinate system at $\boldsymbol{y}$ which will determine a desired minimal tessellation.  To use in applying Lemma~\ref{lemm:PartitionLineSeg}, we extend, for each $i = 1, \cdots, d$, the chord $\sL_i(\boldsymbol{y})$ to \[\sL_i^D(\boldsymbol{y}):=\{y_1\} \times \cdots \times \{y_{i-1}\} \times [-D+x_i, x_i+D] \times \{y_{i+1}\} \times \cdots \times \{y_{d}\}\] where $D>R$ is a real number.  Let \[\eJ := \{j \in \{1, \cdots, d\} : \boldsymbol{e}\left(\sL_j(\boldsymbol{y})\right) \in  \B^{\circ} \text{ or } \boldsymbol{e'}\left(\sL_j(\boldsymbol{y})\right) \in  \B^{\circ} \}.\] First, note that, for all three cases above, $\widecheck{\sL}_j(\boldsymbol{y})$ satisfies the minimality condition precisely formulated in the following lemma.
 \begin{lemm} \label{lemm:prop:RepsMinTess} Let $j \in \eJ$.  If ${\Ls}^\#_j(\boldsymbol{y})$ is a line segment parallel to the $j$-axis through the point $\boldsymbol{y}$ with endpoints $\boldsymbol{e}\left({\Ls}^\#_j(\boldsymbol{y})\right),  \boldsymbol{e'}\left({\Ls}^\#_j(\boldsymbol{y})\right) \in \partial(\B)$ such that ${\sL}_j(\boldsymbol{y}) \subset {\Ls}^\#_j(\boldsymbol{y})$, then $\widecheck{\sL}_j(\boldsymbol{y}) \subset {\Ls}^\#_j(\boldsymbol{y})$.
\end{lemm}
\begin{proof} Note that the above proof gives that $|\sP\left({\sL_j}(\boldsymbol{y})\right)|\geq 1$.  Let $L_j^\#$ be the length of ${\Ls}^\#_j(\boldsymbol{y})$ and set $D: = 5 L_j^\#$.  (Note the choice of $D$, as long as it is large enough, is immaterial to this proof.)  Let\begin{align*}
 \widetilde{\boldsymbol{c}}&:=(y_1, \cdots, y_{j-1}, \widetilde{c_j}, y_{j+1},\cdots, y_d)^t:= \boldsymbol{e}\left(\sL_j(\boldsymbol{y})\right), \\ \boldsymbol{c^\#}&:=(y_1, \cdots, y_{j-1}, c_j^\#, y_{j+1},\cdots, y_d)^t:= \boldsymbol{e}\left({\Ls}^\#_j(\boldsymbol{y})\right).\end{align*}  If $\widetilde{\boldsymbol{c}}\in  \B^{\circ}$, then the above proof gives that there exists consecutive elements \[\boldsymbol{b}:=(y_1, \cdots, y_{j-1}, b_j, y_{j+1},\cdots, y_d)^t \quad \textrm { and } \quad \boldsymbol{b'}:=(y_1, \cdots, y_{j-1}, b_j', y_{j+1},\cdots, y_d)^t\] of $\sP\left(\sL_j^D(\boldsymbol{y})\right)$ such that  $\boldsymbol{b} \notin \sL_j(\boldsymbol{y})$ and $\boldsymbol{b'} \in  \sL_j(\boldsymbol{y})$ with $b_j < \widetilde{c_j} < b_j'$ and that $ \boldsymbol{e}\left(\widecheck{\sL}_j(\boldsymbol{y})\right) = \boldsymbol{b}$.  Since ${\sL}_j(\boldsymbol{y}) \subset {\Ls}^\#_j(\boldsymbol{y})$ holds, we have that $\widetilde{\boldsymbol{c}}, \boldsymbol{b'} \in {\Ls}^\#_j(\boldsymbol{y})$.  Since  $\boldsymbol{c^\#} \in \partial(\B)$ holds, we have that $\boldsymbol{c^\#} \in \sP\left(\sL_j^D(\boldsymbol{y})\right)$.  Consequently, we have that $c_j^\# \leq \widetilde{c_j}$. Since $\widetilde{\boldsymbol{c}} \in  \B^{\circ}$ holds, we must have that $c_j^\# \leq b_j$.  This implies that $\boldsymbol{e}\left(\widecheck{\sL}_j(\boldsymbol{y})\right) = \boldsymbol{b} \in {\Ls}^\#_j(\boldsymbol{y})$.  

If $\widetilde{\boldsymbol{c}} \notin  \B^{\circ}$, then $\widetilde{\boldsymbol{c}} \in \partial(\B)$ holds and the above proof gives that $\boldsymbol{e}\left(\widecheck{\sL}_j(\boldsymbol{y})\right) =\widetilde{\boldsymbol{c}}$.  Since ${\sL}_j(\boldsymbol{y}) \subset {\Ls}^\#_j(\boldsymbol{y})$ holds, we have that $\boldsymbol{e}\left(\widecheck{\sL}_j(\boldsymbol{y})\right) =\widetilde{\boldsymbol{c}} \in {\Ls}^\#_j(\boldsymbol{y})$.

The proof that $\boldsymbol{e'}\left(\widecheck{\sL}_j(\boldsymbol{y})\right)\in {\Ls}^\#_j(\boldsymbol{y})$ holds is analogous to the proceeding.  Since $\boldsymbol{e}\left(\widecheck{\sL}_j(\boldsymbol{y})\right), \boldsymbol{e'}\left(\widecheck{\sL}_j(\boldsymbol{y})\right)\in {\Ls}^\#_j(\boldsymbol{y})$ holds, we have that $\widecheck{\sL}_j(\boldsymbol{y}) \subset {\Ls}^\#_j(\boldsymbol{y})$, which is the desired result.  This proves the lemma.
\end{proof}

\noindent Now set $D = 5R$.  (Note the choice of $D$, as long as it is large enough, is immaterial to this proof.)  By the above proof, we have that $\eJ \neq \emptyset$, and, for each $j \in \eJ$, we have $\widecheck{\sL}_j(\boldsymbol{y})$ is of length $L_j$.  Let \begin{align}\label{eqn:prop:RepsMinTess:LDefn}
 L:= \sup_{j \in \eJ} L_j . \end{align}    There are two cases \begin{enumerate}[label=(\Alph*)]
\item $R/R' \notin \NN$,
\item $R/R' \in \NN$,
\end{enumerate} to consider.  First, we consider Case~(A).  We claim that $\eJ = \{1, \cdots, d\}$.  We now prove the claim.  Assume that the conclusion of the claim does not hold.  Then there exists an integer $m$ such that $1 \leq m \leq d$ for which $\boldsymbol{e}\left(\sL_m(\boldsymbol{y})\right),  \boldsymbol{e'}\left(\sL_m(\boldsymbol{y})\right) \in \partial (\B)$ holds.  Consequently, we have that $\boldsymbol{e}\left(\sL_m(\boldsymbol{y})\right),  \boldsymbol{e'}\left(\sL_m(\boldsymbol{y})\right) \in \sP\left(\sL_m^D(\boldsymbol{y})\right)$, and, thus by Lemma~\ref{lemm:PartitionLineSeg}, the distance between $\boldsymbol{e}\left(\sL_m(\boldsymbol{y})\right)$ and $\boldsymbol{e'}\left(\sL_m(\boldsymbol{y})\right)$ is an integer multiple of $2R'$.  Hence, $2R$ is an integer multiple of $2R'$, which yields a contradiction.  This proves the claim.

Case~(A) has two further cases:  \begin{enumerate}
\item [(Ai)]  there exists a $k \in \eJ$ such that $L_k/2R' = \left\lceil R/R' \right\rceil+1$,
\item [(Aii)] there does not exist a $k \in \eJ$ such that $L_k/2R'  = \left\lceil R/R' \right\rceil+1$. 
\end{enumerate}  We consider Case~(Ai) first.  For every $j \in \eJ$ such that $L_j/2R'  = \left\lceil R/R' \right\rceil+1$, define ${\sL'}_j(\boldsymbol{y}):=\widecheck{\sL}_j(\boldsymbol{y})$ and denote its length by $L'_j:=L_j=2 R'\left(\left\lceil R/R' \right\rceil+1\right)$.  For any $j \in \eJ$ such that $L_j/2R' \neq \left\lceil R/R' \right\rceil+1$, we have that $L_j/2R' =  \left\lceil R/R' \right\rceil$ by the above proof.  Since $\boldsymbol{e}\left(\widecheck{\sL}_j(\boldsymbol{y})\right),  \boldsymbol{e'}\left(\widecheck{\sL}_j(\boldsymbol{y})\right) \in \sP\left(\sL_j^D(\boldsymbol{y})\right)$, we have two choices \begin{align}\label{eqn:RepsMinTess:ChoicebCaseAi}
 \widehat{\boldsymbol{b}}:= (y_1, \cdots, y_{j-1}, \widehat{b_j}, y_{j+1},\cdots, y_d)^t \quad \textrm{and} \quad \widecheck{\boldsymbol{b}}:=(y_1, \cdots, y_{j-1}, \widecheck{b_j}, y_{j+1},\cdots, y_d)^t \end{align} of elements of $\sP\left(\sL_j^D(\boldsymbol{y})\right)$ that are consecutive to $\boldsymbol{e}\left(\widecheck{\sL}_j(\boldsymbol{y})\right)$ or consecutive to $\boldsymbol{e'}\left(\widecheck{\sL}_j(\boldsymbol{y})\right)$ for which $\{\widehat{\boldsymbol{b}}, \widecheck{\boldsymbol{b}}\} \cap \widecheck{\sL}_j(\boldsymbol{y}) = \emptyset$.  Without loss of generality, we may assume that $ \widehat{b_j} < \widecheck{b_j}$.  We can choose either $\widehat{\boldsymbol{b}}$ or $\widecheck{\boldsymbol{b}}$ for our construction.  A different choice would yield a different minimal tessellation of $B(\boldsymbol{x}, R)$ but is immaterial for the proof of the existence of a minimal tessellation.  For definiteness, we choose $\widehat{\boldsymbol{b}}$.  From Lemma~\ref{lemm:PartitionLineSeg}, it follows that the distance between $\widehat{\boldsymbol{b}}$ and $\boldsymbol{e}\left(\widecheck{\sL}_j(\boldsymbol{y})\right)$ is $2R'$, and let us set $\boldsymbol{c'}:=(y_1, \cdots, y_{j-1}, c_j', y_{j+1},\cdots, y_d)^t:= \boldsymbol{e'}\left(\widecheck{\sL}_j(\boldsymbol{y})\right)$.  Now extend $\widecheck{\sL}_j(\boldsymbol{y})$ to \[{\sL'}_j(\boldsymbol{y}):=\{y_1\} \times \cdots \times \{y_{j-1}\} \times [\widehat{b_j}, c_j'] \times \{y_{j+1}\} \times \cdots \times \{y_{d}\},\] which has length $L'_j := 2R'\left(\left\lceil R/R' \right\rceil+1\right)$.  Note that the above proof gives that \begin{align}\label{eqn:prop:RepsMinTess:SegmentIncCaseAi}
 \sL_j(\boldsymbol{y}) \subset \widecheck{\sL}_j(\boldsymbol{y})\subset {\sL'}_j(\boldsymbol{y})\end{align} for all integers $j$ such that $1 \leq j \leq d$.  Consequently, for Case~(Ai), we have constructed a uniform local coordinate system $\{{\sL'}_1(\boldsymbol{y}), \cdots, {\sL'}_d(\boldsymbol{y})\}$ at $\boldsymbol{y}$ of size $L= 2R'\left(\left\lceil R/R' \right\rceil+1\right)$ satisfying (\ref{eqn:lemm:LocalCoordinates}).  Thus, by Lemma~\ref{lemm:LocalCoordinates}, we have that there exists a unique subtessellation $\mS$ of $\B$ such that\begin{align}\label{eqn:prop:RepsMinTess:ConstructedULCSaty}
  {\sL'}_1(\boldsymbol{y}) \cup  \cdots \cup{\sL'}_d(\boldsymbol{y}) \subset \bigcup_{B'' \in \mS} B'' \end{align} and, moreover, $\mS$ is the $\left(\left\lceil R/R' \right\rceil+1\right)$-tessellation of \[\prod_{j=1}^d\left[\pi_j\left(\boldsymbol{e}\left({\sL'}_j(\boldsymbol{y})\right)\right), \pi_j\left(\boldsymbol{e'}\left({\sL'}_j(\boldsymbol{y})\right)\right)\right] \supset \prod_{j=1}^d\left[\pi_j\left(\boldsymbol{e}\left({\sL}_j(\boldsymbol{y})\right)\right), \pi_j\left(\boldsymbol{e'}\left({\sL}_j(\boldsymbol{y})\right)\right)\right] = B(\boldsymbol{x}, R),\] where $\pi_j(\cdot)$ is defined in (\ref{eqn:DefnProjMaps}) and the superset and equality follow by (\ref{eqn:prop:RepsMinTess:SegmentIncCaseAi}, \ref{eqn:prop:RepsMinTess:EndPointsChords}).  Hence, $\mS$ covers $B(\boldsymbol{x}, R)$.
 
Now we claim that $\mS$ is a minimal tessellation of $B(\boldsymbol{x}, R)$.  We now prove the claim.  Assume the conclusion does not hold.  Then there exists an $\mS^* \in \col(\B, B(\boldsymbol{x}, R))$ such that $\mS^*$ is a proper subcollection of $\mS$.  We note that $\mS^* \neq \B$.  Thus $\mS^*$ must be a subtessellation and, hence, is also a tessellation.  Thus, there exists a closed ball $B(\boldsymbol{\widetilde{x}}, \widetilde{R})$ given by $\| \cdot \|_\infty$ for some  $\boldsymbol{\widetilde{x}} \in \RR^d$ and $\widetilde{R}>0$ such that $\mS^*$ is the $M$-tessellation of $B(\boldsymbol{\widetilde{x}}, \widetilde{R})$ for some $M \in \NN$.  Since $\mS$ is an $\left(\left\lceil R/R' \right\rceil+1\right)$-tessellation, we have that $|\mS| = \left(\left\lceil R/R' \right\rceil+1\right)^d$.  Thus, $|\mS^*|= M^d < \left(\left\lceil R/R' \right\rceil+1\right)^d$, which implies that \begin{align}\label{eqn:propMinTess:CaseAi:MBnd}
 M \leq \left\lceil R/R' \right\rceil. \end{align}

Since $\mS^*$ covers $B(\boldsymbol{x}, R)$, we have that \begin{align}\label{eqn:propMinTess:SubTessCovers}
 B(\boldsymbol{x}, R)) \subset \bigcup_{B'' \in \mS^*} B''  = B(\boldsymbol{\widetilde{x}}, \widetilde{R}) \end{align} holds by Lemma~\ref{lemm:TesselBallFund}, which gives that $\boldsymbol{y} \in  \inte\left( B(\boldsymbol{\widetilde{x}}, \widetilde{R})\right)$.  Let $\sL^*_j(\boldsymbol{y})$ be the  chord of $B(\boldsymbol{\widetilde{x}}, \widetilde{R})$ parallel to the $j$-axis through the point $\boldsymbol{y}$.  Since (\ref{eqn:propMinTess:SubTessCovers}) holds, we have that  $\sL_j(\boldsymbol{y}) \subset \sL^*_j(\boldsymbol{y})$, and, by Lemma~\ref{lemm:ChordSubtesselProperties2}, we have that \[
 \boldsymbol{e}\left(\sL^*_j(\boldsymbol{y})\right), \boldsymbol{e'}\left(\sL^*_j(\boldsymbol{y})\right) \in \bigcup_{B'' \in \mS^*} \left(B'' \backslash \inte(B'')\right) \subset \partial(\B)\] for all integers $j$ such that $1 \leq j \leq d$.

 In particular for $k$, Lemma~\ref{lemm:prop:RepsMinTess} implies that $\sL'_k(\boldsymbol{y}) \subset \sL^*_k(\boldsymbol{y})$.  Now, since $\sL^*_k(\boldsymbol{y})$ is a chord of $B(\boldsymbol{\widetilde{x}}, \widetilde{R})$, it has length $2\widetilde{R} = 2MR' \leq 2R'  \left\lceil R/R' \right\rceil$, where the equality follows from Lemma~\ref{lemm:ChordSubtesselProperties1} and the inequality follows from (\ref{eqn:propMinTess:CaseAi:MBnd}).  However, $\sL'_k(\boldsymbol{y})$ has length $L'_k = 2 R'\left(\left\lceil R/R' \right\rceil+1\right)$, which yields a contradiction.  This completes the proof of the claim and shows that $\mS$ is a minimal tessellation of $B(\boldsymbol{x}, R)$.  Since $\mS$ is also a $\left(\left\lceil R/R' \right\rceil+1\right)$-tessellation, we have $|\mS| = \left(\left\lceil R/R' \right\rceil+1\right)^d$ as desired.  This completes Case~(Ai).

Next, we consider Case~(Aii).  Since there does not exist a $j \in \eJ$ such that $L_j/2R'  = \left\lceil R/R' \right\rceil+1$, we have, by the above proof, that $L_j/2R' =  \left\lceil R/R' \right\rceil$ for all $j \in \eJ$.  Define ${\sL'}_j(\boldsymbol{y}):=\widecheck{\sL}_j(\boldsymbol{y})$ and denote its length by $L'_j:=L_j=2 R'\left(\left\lceil R/R' \right\rceil\right)$.  Using a proof analogous to the proof of Case~(Ai), we obtain (\ref{eqn:prop:RepsMinTess:SegmentIncCaseAi}) and that the uniform local coordinate system $\{{\sL'}_1(\boldsymbol{y}), \cdots, {\sL'}_d(\boldsymbol{y})\}$ of size $L=2 R'\left(\left\lceil R/R' \right\rceil\right)$ determines the unique subtessellation $\mS$ of $\B$ satisfying (\ref{eqn:prop:RepsMinTess:ConstructedULCSaty}).  The analogous proof also gives that $\mS$ covers $B(\boldsymbol{x}, R)$ and is also an $\left(\left\lceil R/R' \right\rceil\right)$-tessellation.

Now we claim that $\mS$ is a minimal tessellation of $B(\boldsymbol{x}, R)$. The proof of the claim is analogous to the proof in Case~(Ai) except that (\ref{eqn:propMinTess:CaseAi:MBnd}) is replaced by \[ M \leq \left\lceil R/R' \right\rceil - 1\] and the particular $k \in \eJ$ can be replaced by any $j \in \eJ$ as the length of $\sL'_j(\boldsymbol{y})$ is $L=L'_j$ for all $j \in \eJ$.  (Note that $\left\lceil R/R' \right\rceil - 1 \geq 1$ for the case under consideration.)  Since $\mS$ is also a $\left(\left\lceil R/R' \right\rceil\right)$-tessellation, we have $|\mS| = \left(\left\lceil R/R' \right\rceil\right)^d$ as desired.  This completes Case~(Aii) and, thus, Case~(A).

Next, we consider Case~(B).   Let $j \in \eJ$.  By the above proof, $\widecheck{\sL}_j(\boldsymbol{y})$ has length $L_j=2 R'\left( R/R' +1\right)$.  Define ${\sL'}_j(\boldsymbol{y}):=\widecheck{\sL}_j(\boldsymbol{y})$ and denote its length by $L'_j:=L_j=2 R'\left( R/R' +1\right)$.  

Now let $j \in \{1, \cdots, d\} \backslash \eJ$.  Then $\boldsymbol{e}\left(\sL_j(\boldsymbol{y})\right),  \boldsymbol{e'}\left(\sL_j(\boldsymbol{y})\right)$ are in $\partial (\B)$ and, therefore, also in $\sP\left(\sL_j^D(\boldsymbol{y})\right)$.  Let $\widecheck{\sL}_j(\boldsymbol{y}):= \sL_j(\boldsymbol{y})$.  Consequently, we have two choices \begin{align}\label{eqn:RepsMinTess:ChoicebCaseB} \widehat{\boldsymbol{b}}:= (y_1, \cdots, y_{j-1}, \widehat{b_j}, y_{j+1},\cdots, y_d)^t \quad \textrm{and} \quad \widecheck{\boldsymbol{b}}:=(y_1, \cdots, y_{j-1}, \widecheck{b_j}, y_{j+1},\cdots, y_d)^t \end{align} of elements of $\sP\left(\sL_j^D(\boldsymbol{y})\right)$ that are consecutive to $\boldsymbol{e}\left({\sL}_j(\boldsymbol{y})\right)$ or consecutive to $\boldsymbol{e'}\left({\sL}_j(\boldsymbol{y})\right)$ for which $\{\widehat{\boldsymbol{b}}, \widecheck{\boldsymbol{b}}\} \cap {\sL}_j(\boldsymbol{y}) = \emptyset$.  Analogous to the proof in Case (Ai), we choose $\widehat{\boldsymbol{b}}$ to define \begin{align}\label{eqn:prop:RepsMinTess:CaseBConstruct}
 {\sL'}_j(\boldsymbol{y}):=\{y_1\} \times \cdots \times \{y_{j-1}\} \times [\widehat{b_j}, x_j+R] \times \{y_{j+1}\} \times \cdots \times \{y_{d}\}, \end{align} which has length $L'_j := 2R'(R/R' + 1)$ by Lemma~\ref{lemm:PartitionLineSeg} and (\ref{eqn:prop:RepsMinTess:EndPointsChords}).  Similar to Case (Ai), a different choice would yield a different minimal tessellation of $B(\boldsymbol{x}, R)$.   Note that it may be possible for $\{1, \cdots, d\} \backslash \eJ = \emptyset$ in which case we do not need to define (\ref{eqn:prop:RepsMinTess:CaseBConstruct}) for our proof.

Using a proof analogous to the proof of Case~(Ai), we obtain (\ref{eqn:prop:RepsMinTess:SegmentIncCaseAi})  for all integers $j$ such that $1 \leq j \leq d$ and that the uniform local coordinate system $\{{\sL'}_1(\boldsymbol{y}), \cdots, {\sL'}_d(\boldsymbol{y})\}$ of size $L= 2R'(R/R' + 1)$ determines the unique subtessellation $\mS$ of $\B$ satisfying (\ref{eqn:prop:RepsMinTess:ConstructedULCSaty}).  The analogous proof also gives that $\mS$ covers $B(\boldsymbol{x}, R)$ and is also an $R/R' + 1$-tessellation.  

Now we claim that $\mS$ is a minimal tessellation of $B(\boldsymbol{x}, R)$. The proof of the claim is analogous to the proof in Case~(Ai) except that the particular $k \in \eJ$ can be replaced by any $j \in \eJ$ as the length of $\sL'_j(\boldsymbol{y})$ is $L=L'_j$ for all $j \in \eJ$.  Since $\mS$ is also a $(R/R' +1)$-tessellation, we have $|\mS| = (R/R' +1)^d$ as desired.  This completes Case~(B).  This shows that we can always construct a minimal tessellation $\mS$ of $B(\boldsymbol{x}, R)$ with the desired cardinalities.

We now claim that every minimal tessellation of $B(\boldsymbol{x}, R)$ comes from choosing either $\widehat{\boldsymbol{b}}$ or $\widecheck{\boldsymbol{b}}$ from (\ref{eqn:RepsMinTess:ChoicebCaseAi}) in Case~(Ai) and choosing either $\widehat{\boldsymbol{b}}$ or $\widecheck{\boldsymbol{b}}$ from (\ref{eqn:RepsMinTess:ChoicebCaseB}) in Case~(B) in the above construction of the minimal tessellation $\mS$.  (Note that Case~(Aii) does not involve $\widehat{\boldsymbol{b}}$ or $\widecheck{\boldsymbol{b}}$, and, hence, there is no choice to make for Case~(Aii).)   We now prove this claim.  Let $\mS^\#$ be a minimal tessellation of $B(\boldsymbol{x}, R)$.  If $\mS^\# = \B$, then, as $\mS \subset \B$, we have that $\mS= \B$, a contradiction.  Otherwise, we have that $\mS^\#$ is a subtessellation.   Thus, $\mS^\#$ is also a tessellation and, hence, there exists a closed ball $B(\boldsymbol{x^\#}, R^\#)$ for some $\boldsymbol{x^\#} \in \RR^d$ and $R^\#>0$ such that $\mS^\#$ is the $M^\#$-tessellation of $B(\boldsymbol{x^\#}, R^\#)$ for some $M^\# \in \NN$.  Thus, by Lemma~\ref{lemm:TesselBallFund} and (\ref{eqn:prop:RepsMinTess:Defny}), we have that $\boldsymbol{y} \in \inte(B(\boldsymbol{x^\#}, R^\#))$.  Let $\sL^\#_i(\boldsymbol{y})$ be the  chord of $B(\boldsymbol{x^\#}, R^\#)$ parallel to the $i$-axis through the point $\boldsymbol{y}$.  Since $\mS^\#$ is a minimal tessellation of $B(\boldsymbol{x}, R)$, we have that $\sL_i(\boldsymbol{y}) \subset \sL^\#_i(\boldsymbol{y})$.  Also, by Lemma~\ref{lemm:ChordSubtesselProperties2}, we have that \[
 \boldsymbol{e}\left(\sL^\#_i(\boldsymbol{y})\right), \boldsymbol{e'}\left(\sL^\#_i(\boldsymbol{y})\right) \in \bigcup_{B'' \in \mS^\#} \left(B'' \backslash \inte(B'')\right) \subset \partial(\B),\] which, applying Lemma~\ref{lemm:prop:RepsMinTess} for $i \in \eJ$, gives \begin{align}\label{eqn:prop:RepsMinTess:SegmentIncAllMinTess}
 \sL_i(\boldsymbol{y}) \subset \widecheck{\sL}_i(\boldsymbol{y})\subset \sL^\#_i(\boldsymbol{y})\end{align} for all integers $i$ such that $1 \leq i \leq d$.  (Note that, if $ \{1, \cdots, d \} \backslash \eJ \neq \emptyset$, then we have $\widecheck{\sL}_i(\boldsymbol{y})= \sL_i(\boldsymbol{y})$ for $i \in \{1, \cdots, d \} \backslash \eJ$ in the above proof.)  Thus, $2R^\# \geq L$.  By the above construction, we have two cases \begin{enumerate}[label=(\Roman*)]
\item $L_i = L$ 
\item $L_i = L - 2R'$
\end{enumerate} to consider.  Note that, in both Cases~(Ai) and~(B), we have $L=2 R'\left(\left\lceil R/R' \right\rceil+1\right)$ and Cases~(I) occurs for some $i$ (at least once such $i$) and Case~(II) occurs for other $i$ (possibly no such $i$); in Case~(Aii), we have $L=2 R'\left(\left\lceil R/R' \right\rceil\right)$ and only Case~(I) occurs for all $i$.  For Case~(I), we define $\sL^{**}_i(\boldsymbol{y}):= \widecheck{\sL}_i(\boldsymbol{y})$.  Thus, the $\sL^{**}_i(\boldsymbol{y})$ all have length $L$ for Case~(I).  This completes Case~(I).

We now consider Case~(II).  Let $C:= 10R^\#$.  (The constant $C$, as long as it is large enough, is immaterial to this proof.)  Now we have that  \[ \boldsymbol{e}\left(\widecheck{\sL}_i(\boldsymbol{y})\right), \boldsymbol{e'}\left(\widecheck{\sL}_i(\boldsymbol{y})\right), \boldsymbol{e}\left(\sL^\#_i(\boldsymbol{y})\right), \boldsymbol{e'}\left(\sL^\#_i(\boldsymbol{y})\right) \in \sP\left(\sL_i^{C}(\boldsymbol{y})\right).\]  By Lemma~\ref{lemm:PartitionLineSeg}, either $\widehat{\boldsymbol{b}}$ or $\widecheck{\boldsymbol{b}}$ lie in $\sL^\#_i(\boldsymbol{y})$.  Define \[\boldsymbol{c''}:= (y_1, \cdots, y_{i-1}, c_i'', y_{i+1},\cdots, y_d)^t:= \boldsymbol{e}\left(\widecheck{\sL}_i(\boldsymbol{y})\right)\] and  \[\sL^{**}_i(\boldsymbol{y}):= \begin{cases}  \sL'_i(\boldsymbol{y})  & \textrm{ if } \widehat{\boldsymbol{b}} \in  \sL^\#_i(\boldsymbol{y})\\ \{y_1\} \times \cdots \times \{y_{i-1}\} \times [c_i'', \widecheck{b}_j] \times \{y_{i+1}\} \times \cdots \times \{y_{d}\}& \textrm{ if } \widehat{\boldsymbol{b}} \notin  \sL^\#_i(\boldsymbol{y})\end{cases}.\]  Note that, by the definitions of $\widehat{\boldsymbol{b}}$ and $\widecheck{\boldsymbol{b}}$ and Lemma~\ref{lemm:PartitionLineSeg}, the $\sL^{**}_i(\boldsymbol{y})$ all have length $L$ for Case~(II).  This completes Case~(II).
 
 Consequently,  we have that \begin{align}\label{eqn:prop:RepsMinTess:SegmentIncAllMinTess2}
 \sL_i(\boldsymbol{y}) \subset \widecheck{\sL}_i(\boldsymbol{y})\subset \sL^{**}_i(\boldsymbol{y})\subset \sL^\#_i(\boldsymbol{y})\end{align} holds for all integers $i$ such that $1 \leq i \leq d$ and that $\{\sL^{**}_1(\boldsymbol{y}), \cdots, \sL^{**}_1(\boldsymbol{y})\}$ is a uniform local coordinate system at $\boldsymbol{y}$ of size $L$ satisfying (\ref{eqn:lemm:LocalCoordinates}).  Thus, by Lemma~\ref{lemm:LocalCoordinates}, we have there exists an unique subtessellation $\widetilde{\mS}$ of $\B$ such that\begin{align}\label{eqn:prop:RepsMinTess:ConstructedULCSaty2}
  \sL^{**}_1(\boldsymbol{y}) \cup  \cdots \cup \sL^{**}_d(\boldsymbol{y}) \subset \bigcup_{B'' \in \widetilde{\mS}} B'' \end{align} and, moreover, such that $\widetilde{\mS}$ is the $(L/2R')$-tessellation of \[\prod_{j=1}^d\left[\pi_j\left(\boldsymbol{e}\left(\sL^{**}_j(\boldsymbol{y})\right)\right), \pi_j\left(\boldsymbol{e'}\left(\sL^{**}_j(\boldsymbol{y})\right)\right)\right] \supset \prod_{j=1}^d\left[\pi_j\left(\boldsymbol{e}\left({\sL}_j(\boldsymbol{y})\right)\right), \pi_j\left(\boldsymbol{e'}\left({\sL}_j(\boldsymbol{y})\right)\right)\right] = B(\boldsymbol{x}, R),\] where $\pi_j(\cdot)$ is defined in (\ref{eqn:DefnProjMaps}) and the superset and equality follow by (\ref{eqn:prop:RepsMinTess:SegmentIncAllMinTess2}, \ref{eqn:prop:RepsMinTess:EndPointsChords}).  Hence, $\widetilde{\mS}$ covers $B(\boldsymbol{x}, R)$.

 Now, since \[B(\boldsymbol{x^\#}, R^\#) = \prod_{j=1}^d\left[\pi_j\left(\boldsymbol{e}\left(\sL^\#_j(\boldsymbol{y})\right)\right), \pi_j\left(\boldsymbol{e'}\left(\sL^\#_j(\boldsymbol{y})\right)\right)\right] \supset \prod_{j=1}^d\left[\pi_j\left(\boldsymbol{e}\left(\sL^{**}_j(\boldsymbol{y})\right)\right), \pi_j\left(\boldsymbol{e'}\left(\sL^{**}_j(\boldsymbol{y})\right)\right)\right]\] holds by (\ref{eqn:prop:RepsMinTess:SegmentIncAllMinTess2}), we have, for any $B \in \widetilde{\mS}$, that \[B \subset B(\boldsymbol{x^\#}, R^\#)\] by Lemma~\ref{lemm:TesselBallFund}.  Since $\mS^\#$ is the $M^\#$-tessellation of $B(\boldsymbol{x^\#}, R^\#)$, we have $B \in \mS^\#$ by Lemma~\ref{lemm:ElementPartofTess}.  Consequently, we have that $\widetilde{\mS} \subset \mS^\#$.  Since $\mS^\#$ is a minimal tessellation of $B(\boldsymbol{x}, R)$, we have $\mS^\# = \widetilde{\mS}$ and completes the proof of the claim.
 
 Finally, since, by the claim, any minimal tessellation of $B(\boldsymbol{x}, R)$ comes from, possibly, a choice of $\widehat{\boldsymbol{b}}$ or $\widecheck{\boldsymbol{b}}$ for each coordinate $i$, we have that there are a  finite number of minimal tessellations of $B(\boldsymbol{x}, R)$.  This completes the proof of the proposition.

\subsubsection{Proof of Proposition~\ref{prop:RepsMaxTess}}\label{sec:ProofPropRepsMaxTess}  The proof of the first assertion is analogous to the proof in Section~\ref{subsubsec:ProofFirstAssertProp:MinTess}.  We now prove the second assertion.  For this proof, we adapt the proof in Section~\ref{subsubsec:ProofSecAssertProp:MinTess} and, for conciseness, only give the necessary changes.  All our notation comes from that proof. Note, by Lemma~\ref{lemm:MaxTessNonEmptyCond}, we have that $\col^*(\B, B(\boldsymbol{x}, R)) \neq \emptyset$ and, thus, that there exists a $\widetilde{B} \in \B$ such that \begin{align}\label{eqn:prop:RepsMaxTess:SeedTess} \widetilde{B} \subset B(\boldsymbol{x}, R).\end{align} Replace the $\boldsymbol{y}$ from (\ref{eqn:prop:RepsMinTess:Defny}) with \begin{align}\label{eqn:prop:RepsMaxTess:Defny}
 \boldsymbol{y}:=(y_1, \cdots, y_d)^t \in \inte(\widetilde{B}) \cap \B^{\circ}.\end{align}  Let \[\widetilde{\sL_j}(\boldsymbol{y}):=\{y_1\} \times \cdots \times \{y_{j-1}\} \times [-(R+R')+x_j, x_j+R+R'] \times \{y_{j+1}\} \times \cdots \times \{y_{d}\}.\]  Note that $\sL_j(\boldsymbol{y}) \cap \widetilde{B}$ is the chord of $\widetilde{B}$ parallel to the $j$-axis through the point $\boldsymbol{y}$, and, by (\ref{eqn:prop:RepsMaxTess:SeedTess}) and Lemma~\ref{lemm:ChordsThruInteriorPointsChar}, we have that $\boldsymbol{e}(\sL_j(\boldsymbol{y}) \cap \widetilde{B}), \boldsymbol{e'}(\sL_j(\boldsymbol{y}) \cap \widetilde{B}) \in  \sP(\sL_j(\boldsymbol{y})) \subset \sP(\widetilde{\sL_j}(\boldsymbol{y}))$.   Using Lemma~\ref{lemm:PartitionLineSeg}, we may set \[\boldsymbol{a}:= (y_1, \cdots, y_{j-1}, a_j, y_{j+1},\cdots, y_d)^t \textrm{ and } \boldsymbol{a'}:= (y_1, \cdots, y_{j-1}, a_j', y_{j+1},\cdots, y_d)^t,\] where $a_j < a_j'$, to be the two elements of $\sP(\sL_j(\boldsymbol{y}))$ whose distance $D_j$ is greatest.  By (\ref{eqn:prop:RepsMinTess:EndPointsChords}), we have that \begin{align}\label{eqn:prop:RepsMaxTess:Dist}
 x_j-R\leq  a_j < a_j' \leq x_j+R.  \end{align}   Form the line segment \[\widecheck{\sL}_j(\boldsymbol{y}):= \{y_1\} \times \cdots \times \{y_{j-1}\} \times [a_j, a_j'] \times \{y_{j+1}\} \times \cdots \times \{y_{d}\}\] and set \[\kappa^+_j:= x_j+R- a_j' \quad \textrm{ and }\quad \kappa^-_j:= a_j - (x_j-R).\]

As in Section~\ref{subsubsec:ProofSecAssertProp:MinTess}, we have Cases~(1), (2), and~(3).  Let us first consider Case~(1).  We claim that $\boldsymbol{e'}(\sL_j(\boldsymbol{y})) = \boldsymbol{a'}$.  We now prove this claim.  Assume that the conclusion is false.  Then, using (\ref{eqn:prop:RepsMaxTess:Dist}), we have $a_j' < x_j+R$ and, thus, $\|\boldsymbol{e'}(\sL_j(\boldsymbol{y}))-\boldsymbol{a}\|_\infty = x_j+R - a_j > a_j'-a_j = D_j$, yielding two elements of $\sP(\sL_j(\boldsymbol{y}))$ whose distance is strictly larger than $D_j$.  This is a contradiction and proves the claim.  The claim gives that $\kappa^+_j = 0$.   Since $\boldsymbol{e}(\sL_j(\boldsymbol{y})) \in \B^\circ$, we have that $ x_j-R< a_j$.  Using Lemma~\ref{lemm:PartitionLineSeg}, set $\boldsymbol{\widehat{a}}:= (y_1, \cdots, y_{j-1}, \widehat{a}_j, y_{j+1},\cdots, y_d)^t$ to be the element of $\sP(\widetilde{\sL_j}(\boldsymbol{y}))$ consecutive to $\boldsymbol{a}$ such that $\widehat{a}_j < a_j$, and note that $a_j - \widehat{a}_j = 2R'$.

Next, we claim that $\widehat{a}_j < x_j-R$.  We now prove this claim.  Assume that the conclusion does not hold.  Then, we have that  $\widehat{a}_j \geq x_j-R$, which implies that $\boldsymbol{\widehat{a}} \in \sP(\sL_j(\boldsymbol{y}))$.  Consequently, we have that $\|\boldsymbol{a'}- \boldsymbol{\widehat{a}}\|_\infty = a_j' - \widehat{a}_j = a_j' - a_j + 2R' = D_j +2 R'$.  This yields two elements of $\sP(\sL_j(\boldsymbol{y}))$ with distance strictly greater than $D_j$, which is a contradiction.  This proves the claim.  Consequently, we have \[\widehat{a}_j < x_j-R< a_j \quad \textrm{ and } \quad a_j -\widehat{a}_j = 2R',\] that $\widecheck{\sL}_j(\boldsymbol{y})$ has length \begin{align}\label{eqn:prop:RepsMaxTess:Lj}
 L_j:= 2R - \kappa_j^-,  \end{align} and that \begin{align}\label{eqn:prop:RepsMaxTess:kappasize}
 0<\frac{\kappa_j^-}{2R'} < 1 \end{align} holds.  Analogously to the proof in Section~\ref{subsubsec:ProofSecAssertProp:MinTess}, we have, for Case~(1), that \begin{align}\label{eqn:prop:RepsMaxTess:kappasizeCase1LRBnd} \frac{L_j} {2 R'} = \left\lfloor \frac {R}{R'} \right\rfloor
  \end{align} holds and $R/R'$ can not be an integer.  Note that in Case~(1), we have $\boldsymbol{e'}(\sL_j(\boldsymbol{y})) = \boldsymbol{e'}(\widecheck{\sL}_j(\boldsymbol{y}))$.
  
For Case~(2), the analogous proof to Case~(1) gives that $\widecheck{\sL}_j(\boldsymbol{y})$ has length $L_j$ where (\ref{eqn:prop:RepsMaxTess:kappasizeCase1LRBnd}) holds and $R/R'$ can not be an integer.  Note that in Case~(2), we have $\boldsymbol{e}(\sL_j(\boldsymbol{y})) = \boldsymbol{e}(\widecheck{\sL}_j(\boldsymbol{y}))$.

Finally, for Case~(3), we combine the proofs of Cases~(1) and~(2) to yield that $\widecheck{\sL}_j(\boldsymbol{y})$ has length \begin{align}\label{eqn:prop:RepsMaxTess:LjCase3}
 L_j:= 2R - \kappa^-_j-\kappa^+_j \end{align} and (\ref{eqn:prop:RepsMinTess:kappasizeCase3}) holds.  Case~(3) splits into Cases~(3a) and~(3b).  For Case~(3a), the analogous proof from Section~\ref{subsubsec:ProofSecAssertProp:MinTess} yields that \begin{align}\label{eqn:prop:RepsMaxTess:kappasizeCase3aBnd} \kappa^-_j + \kappa^+_j = 2R' \quad \textrm{ and } \quad \frac{L_j} {2 R'} =  \frac {R}{R'} -1
  \end{align} and that $R/R' \in \NN$.  Since for Proposition~\ref{prop:RepsMaxTess} the additional condition (\ref{eqn:prop:RepsMaxTess:CondRRPrime}) holds, we have that the integer $R/R' \geq 2$ and, thus, $L_j >0$ for Case~(3a).  For Case~(3b), the analogous proof from Section~\ref{subsubsec:ProofSecAssertProp:MinTess} yields that \begin{align}\label{eqn:prop:RepsMaxTess:kappasizeCase3bBnd}  \frac{L_j} {2 R'} = \left\lfloor \frac {R}{R'} \right\rfloor \quad \textrm{ or }  \quad\frac{L_j} {2 R'} = \left\lfloor \frac {R}{R'} \right\rfloor -1
  \end{align} and that $R/R' \notin \NN$.  Thus (\ref{eqn:prop:RepsMaxTess:CondRRPrime}) gives that $R/R' \geq 2$ and, thus, $L_j >0$ for Case~(3b).  This completes Case~(3).

 We now construct a uniform local coordinate system at $\boldsymbol{y}$ which will determine a desired maximal tessellation.

 \begin{lemm} \label{lemm:prop:RepsMaxTess} Let $j \in \eJ$.  If ${\Ls}^\#_j(\boldsymbol{y})$ is a line segment parallel to the $j$-axis through the point $\boldsymbol{y}$ with endpoints $\boldsymbol{e}\left({\Ls}^\#_j(\boldsymbol{y})\right),  \boldsymbol{e'}\left({\Ls}^\#_j(\boldsymbol{y})\right) \in \partial(\B)$ such that ${\sL}_j(\boldsymbol{y}) \supset {\Ls}^\#_j(\boldsymbol{y})$, then $\widecheck{\sL}_j(\boldsymbol{y}) \supset {\Ls}^\#_j(\boldsymbol{y})$.
\end{lemm}
\begin{proof}  The proof is analogous to that of Lemma~\ref{lemm:prop:RepsMinTess}.

\end{proof}

Continuing to adapt the proof in Section~\ref{subsubsec:ProofSecAssertProp:MinTess}, we replace $L$ from (\ref{eqn:prop:RepsMinTess:LDefn}) with \begin{align*}
 L:= \inf_{j \in \eJ} L_j . \end{align*}  As in Section~\ref{subsubsec:ProofSecAssertProp:MinTess}, we have two cases \begin{enumerate}[label=(\Alph*)]
\item $R/R' \notin \NN$ and $R/R'\geq 2$,
\item $R/R' \in \NN$ and $R/R'\geq 2$,
\end{enumerate} to consider.  Case~(A) has two further cases:  \begin{enumerate}
\item [(Ai)]  there exists a $k \in \eJ$ such that $L_k/2R' = \left\lfloor R/R' \right\rfloor-1$,
\item [(Aii)] there does not exist a $k \in \eJ$ such that $L_k/2R'  =  \left\lfloor R/R' \right\rfloor-1$. 
\end{enumerate}  We consider Case~(Ai) first.   For every $j \in \eJ$ such that $L_j/2R'  = \left\lfloor R/R' \right\rfloor-1$, define ${\sL'}_j(\boldsymbol{y}):=\widecheck{\sL}_j(\boldsymbol{y})$, and denote its length by $L'_j:=L_j$.  For any $j \in \eJ$ such that $L_j/2R' \neq  \left\lfloor R/R' \right\rfloor-1$, we have that $L_j/2R' =   \left\lfloor R/R' \right\rfloor$ by the above proof.  Since \begin{align}\label{eqn:prop:RepsMaxTess:CaseAiDefnC}
\boldsymbol{c'''}&:=(y_1, \cdots, y_{j-1}, c_j''', y_{j+1},\cdots, y_d)^t:=\boldsymbol{e}\left(\widecheck{\sL}_j(\boldsymbol{y})\right), \\\nonumber \boldsymbol{c'}&:=(y_1, \cdots, y_{j-1}, c_j', y_{j+1},\cdots, y_d)^t:=\boldsymbol{e'}\left(\widecheck{\sL}_j(\boldsymbol{y})\right )\end{align} are in $\sP\left(\widecheck{\sL}_j(\boldsymbol{y})\right)$, we have by Lemma~\ref{lemm:PartitionLineSeg} and the definitions of $\boldsymbol{e}(\cdot), \boldsymbol{e'}(\cdot)$, a unique element $\widehat{\boldsymbol{a}}:=(y_1, \cdots, y_{j-1}, \widehat{a}_j, y_{j+1},\cdots, y_d)^t $ of $\sP\left(\widecheck{\sL}_j(\boldsymbol{y})\right)$ that is least in distance to $\boldsymbol{c'''}$ and a unique element $\widecheck{\boldsymbol{a}}:=(y_1, \cdots, y_{j-1}, \widecheck{a}_j, y_{j+1},\cdots, y_d)^t $ of $\sP\left(\widecheck{\sL}_j(\boldsymbol{y})\right)$ that is least in distance to $\boldsymbol{c'}$.  Then define \begin{align*}
\widehat{\Ls}_j &:= \{y_1\} \times \cdots \times \{y_{i-1}\} \times [ \widehat{a}_j, c_j'] \times \{y_{i+1}\} \times \cdots \times \{y_{d}\} \\
\widecheck{\Ls}_j &:= \{y_1\} \times \cdots \times \{y_{i-1}\} \times [c_j''',  \widecheck{a}_j] \times \{y_{i+1}\} \times \cdots \times \{y_{d}\},\end{align*} and note that, by Lemma~\ref{lemm:PartitionLineSeg}, the lengths of both are equal to $2R'\left(\left\lfloor R/R' \right\rfloor-1\right)$.  Note that either $\widehat{\Ls}_j$ or $\widecheck{\Ls}_j$ (or both) contain $\boldsymbol{y}$.  Define\footnote{Note that, if $\boldsymbol{y}$ is contained in both, we could make the other choice, but this is immaterial for the proof of the existence of a maximal tessellation and we make this choice for definiteness.} \begin{align}\label{eqn:prop:RepsMaxTess:CaseB:Defnlprime}
 {\sL'}_j(\boldsymbol{y}):=\begin{cases} 
\widehat{\Ls}_j &\text{if }   \boldsymbol{y} \in\widehat{\Ls}_j\\ \widecheck{\Ls}_j &\text{if }   \boldsymbol{y} \notin\widehat{\Ls}_j\end{cases}, \end{align} and denote its length by $L'_j:=2R'\left(\left\lfloor R/R' \right\rfloor-1\right)$.  Analogous to the proof in Section~\ref{subsubsec:ProofSecAssertProp:MinTess}, we have there exists a subtessellation  $\mS$ of $\B$ such that $\mS$ is the $\left(\left\lfloor R/R' \right\rfloor-1\right)$-tessellation of \[\prod_{j=1}^d\left[\pi_j\left(\boldsymbol{e}\left({\sL'}_j(\boldsymbol{y})\right)\right), \pi_j\left(\boldsymbol{e'}\left({\sL'}_j(\boldsymbol{y})\right)\right)\right] \subset \prod_{j=1}^d\left[\pi_j\left(\boldsymbol{e}\left({\sL}_j(\boldsymbol{y})\right)\right), \pi_j\left(\boldsymbol{e'}\left({\sL}_j(\boldsymbol{y})\right)\right)\right] = B(\boldsymbol{x}, R),\] where $\pi_j(\cdot)$ is defined in (\ref{eqn:DefnProjMaps}).  Hence, $\mS$ is contained in $B(\boldsymbol{x}, R)$.
 
 Now we claim that $\mS$ is a maximal tessellation of $B(\boldsymbol{x}, R)$.  By Lemma~\ref{lemm:MaxTessNonEmptyCond}, we have that $\B \notin \col^*(\B, B(\boldsymbol{x}, R))$.  The proof of the claim is analogous to that in Section~\ref{subsubsec:ProofSecAssertProp:MinTess}.  Note that $|\mS| = \left(\left\lfloor R/R' \right\rfloor-1\right)^d$ for Case~(Ai).  This completes Case~(Ai).  The proof for Case~(Aii) is also analogous to that in Section~\ref{subsubsec:ProofSecAssertProp:MinTess} and the constructed maximal tessellation $\mS$ of $B(\boldsymbol{x}, R)$ is such that  $|\mS| = \left(\left\lfloor R/R' \right\rfloor\right)^d$.  This completes Case~(A).

Next, we consider Case~(B).   Let $j \in \eJ$.  By the above proof, $\widecheck{\sL}_j(\boldsymbol{y})$ has length $L_j=2 R'\left( R/R' -1\right)$.  Define ${\sL'}_j(\boldsymbol{y}):=\widecheck{\sL}_j(\boldsymbol{y})$ and denote its length by $L'_j:=L_j=2 R'\left( R/R' -1\right)$.  If $ \{1, \cdots, d\} \backslash \eJ = \emptyset$, we have constructed $\{{\sL'}_1(\boldsymbol{y}), \cdots, {\sL'}_d(\boldsymbol{y})\}$.

Otherwise, let $j \in \{1, \cdots, d\} \backslash \eJ$.  As in Section~\ref{subsubsec:ProofSecAssertProp:MinTess}, set $\widecheck{\sL}_j(\boldsymbol{y}):= \sL_j(\boldsymbol{y})$. Replacing the $\boldsymbol{c'''}, \boldsymbol{c'}$ in (\ref{eqn:prop:RepsMaxTess:CaseAiDefnC}) with \begin{align*}\label{eqn:prop:RepsMaxTess:CaseBDefnC}
\boldsymbol{c'''}&:=(y_1, \cdots, y_{j-1}, c_j''', y_{j+1},\cdots, y_d)^t:=\boldsymbol{e}\left({\sL}_j(\boldsymbol{y})\right), \\\nonumber \boldsymbol{c'}&:=(y_1, \cdots, y_{j-1}, c_j', y_{j+1},\cdots, y_d)^t:=\boldsymbol{e'}\left({\sL}_j(\boldsymbol{y})\right )\end{align*} we define ${\sL'}_j(\boldsymbol{y})$ analogously as in (\ref{eqn:prop:RepsMaxTess:CaseB:Defnlprime}).  Note that the length of ${\sL'}_j(\boldsymbol{y})$ is $L'_j:=2 R'\left( R/R' -1\right)$.  Consequently, we have constructed a uniform local coordinate system $\{{\sL'}_1(\boldsymbol{y}), \cdots, {\sL'}_d(\boldsymbol{y})\}$ at $\boldsymbol{y}$ of size of size $L$.  Analogous to the proof in Section~\ref{subsubsec:ProofSecAssertProp:MinTess}, the collection $\{{\sL'}_1(\boldsymbol{y}), \cdots, {\sL'}_d(\boldsymbol{y})\}$ yields a maximal tessellation $\mS$ of $B(\boldsymbol{x}, R)$ such that $|\mS| = (R/R' -1)^d$. This completes Case~(B).  This shows that we can always construct a maximal tessellation $\mS$ of $B(\boldsymbol{x}, R)$ with the desired cardinalities.

Next, we claim that any maximal tessellation $\mS^\#$ of $B(\boldsymbol{x}, R)$ has desired cardinalities.  We now prove this claim.  Let $B^\# \in \mS^\#$ and $\boldsymbol{z} \in \inte(B^\#).$  Replacing $\boldsymbol{y}$ with $\boldsymbol{z}:=(z_1, \cdots, z_d)^t$ in the above proof yields the maximal tessellation $\mS$ of $B(\boldsymbol{x}, R)$.  Note that $\mS$ has the desired cardinalities by the above proof and that \[\boldsymbol{z} \in \bigcup_{B'' \in \mS} B''.\]  Since $\mS^\#$ is a subtessellation (by Lemma~\ref{lemm:MaxTessNonEmptyCond}), it is also a tessellation and, hence, there exists a closed ball $B(\boldsymbol{x^\#}, R^\#)$ for some $\boldsymbol{x^\#} \in \RR^d$ and $R^\#>0$ such that $\mS^\#$ is an $M^\#$-tessellation of $B(\boldsymbol{x^\#}, R^\#)$ for some $M^\# \in \NN$.  Thus, by Lemma~\ref{lemm:TesselBallFund}, we have that $\boldsymbol{z} \in \inte(B(\boldsymbol{x^\#}, R^\#))$.  Let $\sL^\#_i(\boldsymbol{z})$ be the  chord of $B(\boldsymbol{x^\#}, R^\#)$ parallel to the $i$-axis through the point $\boldsymbol{z}$.  By Lemma~\ref{lemm:ChordSubtesselProperties1}, $\sL^\#_i(\boldsymbol{z})$ has length $L_j^\#:=2 R^\# = 2 M^\# R' \geq 2R'$.  Since $\mS^\#$ is a maximal tessellation of $B(\boldsymbol{x}, R)$, we have that $\sL_i(\boldsymbol{z}) \supset \sL^\#_i(\boldsymbol{z})$.  Applying Lemmas~\ref{lemm:ChordSubtesselProperties2} and~\ref{lemm:prop:RepsMaxTess} for $i \in \eJ$ and, we have that \begin{align*}
 \sL_i(\boldsymbol{z}) \supset \widecheck{\sL}_i(\boldsymbol{z})\supset \sL^\#_i(\boldsymbol{z})\end{align*} for all integers $i$ such that $1 \leq i \leq d$.  (Note that, if $ \{1, \cdots, d \} \backslash \eJ \neq \emptyset$, then we have $\widecheck{\sL}_i(\boldsymbol{z})= \sL_i(\boldsymbol{z})$ for $i \in \{1, \cdots, d \} \backslash \eJ$ in the above proof.)

 Thus, \begin{align} \label{eqn:prop:RepsMaxTess:CompareSegs}
 2R'\leq L_j^\# \leq L. \end{align}  By the above construction, we have two cases \begin{enumerate}[label=(\Roman*)]
\item $L_i = L$ 
\item $L_i = L + 2R'$
\end{enumerate} to consider.  Note that, in both Cases~(Ai) and~(B), we have $L=2 R'\left(\left\lfloor R/R' \right\rfloor-1\right)$ and Cases~(I) occurs for some $i$ (at least once such $i$) and Case~(II) occurs for other $i$ (possibly no such $i$); in Case~(Aii), we have $L=2 R'\left(\left\lfloor R/R' \right\rfloor\right)$ and only Case~(I) occurs for all $i$.  For Case~(I), we define $\sL^{**}_i(\boldsymbol{z}):= \widecheck{\sL}_i(\boldsymbol{z})$.  Thus, the $\sL^{**}_i(\boldsymbol{z})$ all have length $L$ for Case~(I).  This completes Case~(I).
 
We now consider Case~(II).  Only Cases~(Ai) and~(B) in the proof above applies to Case~(II).  Let $\widehat{\Ls}_j, \widecheck{\Ls}_i$ be as in Cases~(Ai) and~(B).  Then Lemma~\ref{lemm:PartitionLineSeg} gives that both $\widehat{\Ls}_i, \widecheck{\Ls}_i$ have length $L$.  Let \begin{align*}
\widehat{\Ls}^-_i &:= \{z_1\} \times \cdots \times \{z_{i-1}\} \times [ c_i''', \widehat{a}_i) \times \{z_{i+1}\} \times \cdots \times \{z_{d}\} \\
\widecheck{\Ls}^+_i &:= \{z_1\} \times \cdots \times \{z_{i-1}\} \times (\widecheck{a}_i, c_i'] \times \{z_{i+1}\} \times \cdots \times \{z_{d}\}.\end{align*}  Also, by the above proof, we have that \[\widehat{\Ls}_i \cup \widecheck{\Ls}_i = \widehat{\Ls}^-_i \sqcup \widehat{\Ls}_i=\widecheck{\Ls}_i \sqcup \widecheck{\Ls}^+_i = \widecheck{\sL}_i(\boldsymbol{z}).\]  We assert that either $\sL^\#_i(\boldsymbol{z}) \subset \widehat{\Ls}_i$ or $\sL^\#_i(\boldsymbol{z}) \subset \widecheck{\Ls}_i$.  We now prove the assertion.  Assume that the conclusion does not hold.  Then we have that $\widecheck{\Ls}^+_i \cap \sL^\#_i(\boldsymbol{z})$ and $\widehat{\Ls}^-_i \cap \sL^\#_i(\boldsymbol{z})$ are both nonempty sets.  Since, by Lemma~\ref{lemm:ChordSubtesselProperties2}, we have  $\boldsymbol{e}\left(\sL^\#_i(\boldsymbol{z})\right), \boldsymbol{e'}\left(\sL^\#_i(\boldsymbol{z})\right) \in\sP\left(\sL_i(\boldsymbol{z})\right)$, it follows by Lemma~\ref{lemm:PartitionLineSeg} that $\boldsymbol{c'''},\boldsymbol{c'} \in \sL^\#_i(\boldsymbol{z})$, which implies $L_j^\#\geq L_i= L + 2R'$.  This contradicts (\ref{eqn:prop:RepsMaxTess:CompareSegs}) and proves the assertion.  Now by construction both $\widehat{\Ls}_i$ and $\widecheck{\Ls}_i$ are subsets of $\widecheck{\sL}_j(\boldsymbol{y})$.  Define \[\sL^{**}_i(\boldsymbol{z}):= \begin{cases}  \widehat{\Ls}_i & \text{ if }  \sL^\#_i(\boldsymbol{z}) \subset \widehat{\Ls}_i \\ \widecheck{\Ls}_i & \text{ if }  \sL^\#_i(\boldsymbol{z}) \not\subset \widehat{\Ls}_i\end{cases}.\]  Note that $\sL^{**}_i(\boldsymbol{z})$ has length $L$ for Case~(II).  This completes Case~(II).  Consequently, we have that \begin{align*}
 \sL_i(\boldsymbol{y}) \supset \widecheck{\sL}_i(\boldsymbol{y})\supset \sL^{**}_i(\boldsymbol{y})\supset \sL^\#_i(\boldsymbol{y})\end{align*} holds for all integers $i$ such that $1 \leq i \leq d$.  Now, using the analogous proof in Section~\ref{subsubsec:ProofSecAssertProp:MinTess}, the claim that any maximal tessellation of $B(\boldsymbol{x}, R)$ has the desired cardinalities follows.

Finally, we show that there are only a finitely-many maximal tessellations of $B(\boldsymbol{x}, R)$.  By Lemma~\ref{lemm:TesselEuclidFund}, every maximal tessellations of $B(\boldsymbol{x}, R)$ must lie in a minimal tessellation $\mS^+$ of $B(\boldsymbol{x}, R)$, which, by Proposition~\ref{prop:RepsMinTess}, has finite cardinality and, thus, finitely-many subcollections.  Since every maximal  tessellation of $B(\boldsymbol{x}, R)$ is a subcollection of $\mS^+$, we have that desired result.  This completes the proof the proposition.

\section{An upper bound for the Hausdorff dimension of ubiquitously losing sets}\label{sec:UpBndHDLosingSets}

In this section, we prove Theorem~\ref{thm:UpBndHDLosingSets} (see Section~\ref{subsec:ProofThm:UpBndHDLosingSets}).  To do this, we use the upper box dimension to provide an upper bound on the Hausdorff dimension.  Recall a definition of upper box dimension (see, for example,~\cite{Fac03} for an introduction to Hausdorff and box dimensions) is as follows.   Let $F$ be a nonempty bounded subset of $\RR^d$.  We say $F$ {\em is covered by} a collection $\mC$ of subsets of $\RR^d$ or, alternatively, $\mC$ {\em covers} $F$ if we have that \[F \subset \bigcup_{A \in \mC} A.\]  Let $\delta >0$ and $N_\delta(F)$ be the smallest number of closed balls of radius $\delta$ that cover $F$.  Then the {\em upper box dimension} of $F$, denoted by $\overline{\dim}_B (F)$, is given by \[\overline{\dim}_B (F) = \limsup_{\delta \rightarrow 0} \frac{\log N_\delta(F)}{-\log \delta}.\]

\subsection{Proof of Theorem~\ref{thm:UpBndHDLosingSets}}\label{subsec:ProofThm:UpBndHDLosingSets}  All the balls that we consider in this proof are given by $\|\cdot\|_\infty$.  We prove (\ref{eqn:thm:UpBndHDLosingSets2}) first and, thus, restrict $1/\beta$ to be an integer.  We will play infinitely-many $(1/j, \beta; S)$-accelerated games for Bob on $\RR^d$.  Bob picks a closed ball $B^*_1$ for a $(1/j, \beta; S)$-game on $\RR^d$.  Let $\rho_1 := \rho(B^*_1)$.  Lemmas~\ref{lemm:TesselEuclidFund} and~\ref{lemm:UniqCompletion} and Remark~\ref{rmk:lemm:UniqCompletion} give that \[\RR^d = \bigcup_{B_1' \in  \overline{\{B^*_1\}}} B_1'.\]  Note that each element in $ \overline{\{B^*_1\}}$ is a possible choice of Bob's initial ball for some $(1/j, \beta; S)$-game.  For each $B_1'\in  \overline{\{B^*_1\}}$, we obtain an upper bound for $\dim_H(S \cap B_1')$ using the proof that follows.  Since we obtain the same upper bound for each $B_1' \in  \overline{\{B^*_1\}}$, we have, by countable stability of Hausdorff dimension, the same upper bound for $\dim_H(S)$.  Consequently, we may restrict, without loss of generality, to considering one element $B_1$ of $ \overline{\{B^*_1\}}$.  

\subsubsection{A distinguished set}\label{subsubsec:ADistingSet}  In this section, we show some properties of the set defined in (\ref{eqn:ProofUpperBndHD:DistingSetDefn}).  Let $\hh_0:=\mC_0 := \{B_1\}$ and $F_0(B_1):= B_1$.  Now, since $j \geq 2$ is an integer, we have that \begin{align}\label{eqn:thm:UpBndHDLosingSets:Refinement}
 \mR_j\left(B_1\right) =: \left\{A_{1,(m_1)}: m_1 \in \NN \textrm{ such that } 1\leq m_1 \leq j^d\right\} \end{align} is the $j$-tessellation of $B_1$.  Recall that the radius of any element of $\mR_j\left(B_1\right)$ is $\frac 1 j \rho(B_1) = \rho_1/j$.  Consequently, for all $m_1$, we have that \[A_{1,(m_1)} \in \left(B_1 \right)^{1/j}\] and, thus, that each $A_{1,(m_1)}$ is a possible choice for Alice in a $(1/j, \beta; S)$-game.  Now, for each $m_1$, Bob chooses \[B_{2,(m_1)} \in \left(A_{1,(m_1)}\right)^\beta\] according to a $(1/j, \beta; S)$-winning strategy for Bob with initial ball $B_1$.  Let \[\B_{1,(m_1)}:=\mR_{j^s \beta^{-s}}\left(B_{2,(m_1)}\right).\]  Note that the complete tessellation $\overline{\B_{1,(m_1)}}$ is comprised of closed balls of radius \[R'_1:=\frac{\rho\left(B_{2,(m_1)} \right) }{j^s \beta^{-s} }= \frac{\rho_1 \beta^{s+1}}{j^{s+1}}.\]  Since $\B_{1,(m_1)}$ is a subtessellation of $\overline{\B_{1,(m_1)}}$ and the $\left(j^s \beta^{-s}\right)$-tessellation of $B_{2,(m_1)}$, we have that $B_{2,(m_1)}$ is representable in $\overline{\B_{1,(m_1)}}$.  Define \begin{align}\label{eqn:UpperBndDimCount}
 N:= \left(j^s \beta^{-s-1} +1\right)^d - j^{sd} \beta^{-sd}. \end{align} By Propositions~\ref{prop:RepsMinTess} and~\ref{prop:RepsMaxTess}, we have that the set \begin{align}\label{eq:subsubsec:ADistingSetA1RemoveB2}
 A_{1,(m_1)}  \backslash B_{2,(m_1)} \end{align} is covered by a collection \[\mC_{1,(m_1)}:=\left\{B_{1,(m_1, m_2)}\in \overline{\B_{1,(m_1)}}:m_2 \in \NN \textrm{ such that } 1 \leq m_2 \leq N \right\}\] of closed balls of radius $R'_1$.  Consequently, by Lemma~\ref{lemm:TesselBallFund}, we have that the set \begin{align}\label{eqn:UppBndHD:DesChainFsInitCase}
 F_1(B_1):=F_0(B_1) \big\backslash \bigcup_{m_1=1}^{j^d}B_{2,(m_1)} \end{align} is covered by  the collection \[\mC_1:=  \bigcup_{m_1=1}^{j^d}\mC_{1,(m_1)}=\left\{B_{1,(m_1, m_2)}:m_1, m_2 \in \NN \textrm{ such that } 1\leq m_1 \leq j^d, 1 \leq m_2 \leq N \right\}\] of closed balls of radius $R'_1$.  Note that the collection of balls  that have been removed from $B_1$ is \[\hh^*_1(B_1) := \left\{ B_{2,(m_1)} : m_1 \in \NN \textrm{ such that } 1\leq m_1 \leq j^d\right\}.\]

Now, for each $B_{1,(m_1, m_2)} \in \mC_1$, we can repeat the above with $B_1$ replaced by $B_{1,(m_1, m_2)}$.  Precisely, we proceed as follows.  Bob picks $B_{1,(m_1, m_2)}$ as his initial choice of ball for a $(1/j, \beta; S)$-game on $\RR^d$.  We have the analog of (\ref{eqn:thm:UpBndHDLosingSets:Refinement}), namely \begin{align*}
 \mR_j\left(B_{1,(m_1, m_2)}\right) =: \left\{A_{1,(m_1, m_2, m_3)}: m_3 \in \NN \textrm{ such that } 1\leq m_3 \leq j^d\right\}, \end{align*} whose elements are possible choices for Alice in a $(1/j, \beta; S)$-game in which Bob has picked $B_{1,(m_1, m_2)}$ as his initial choice of ball.  Now, for each $m_3$, Bob chooses \[B_{2,(m_1, m_2, m_3)} \in \left(A_{1,(m_1, m_2, m_3))}\right)^\beta\] according to a $(1/j, \beta; S)$-winning strategy for Bob with initial ball $B_{1,(m_1, m_2)}$.  Let \[\B_{2,(m_1, m_2, m_3)}:= \mR_{j^s \beta^{-s}}\left(B_{2,(m_1, m_2, m_3)}\right).\]  Using the analogous proof to that above, we have that \[B_{1,(m_1, m_2)} \big\backslash \bigcup_{m_3=1}^{j^d}B_{2,(m_1, m_2, m_3)}\] is covered by  the collection \[\mC_{2,(m_1, m_2)}:= \left\{B_{1,(m_1, m_2, m_3, m_4)}:m_3, m_4 \in \NN \textrm{ such that } 1\leq m_3 \leq j^d, 1 \leq m_4 \leq N \right\}\] and, consequently, that \[F_2(B_1):=F_1(B_1) \big\backslash \bigcup_{m_1=1}^{j^d}\bigcup_{m_2=1}^{N} \bigcup_{m_3=1}^{j^d}B_{2,(m_1, m_2, m_3)}\] is covered by  the collection \begin{align*}
\mC_2&:= \bigcup_{m_1=1}^{j^d}\bigcup_{m_2=1}^{N} \mC_{2,(m_1, m_2)}= \\ &\left\{B_{1,(m_1, m_2, m_3, m_4)}:m_1, m_2, m_3, m_4 \in \NN \textrm{ such that } 1\leq m_1, m_3 \leq j^d, 1 \leq m_2, m_4 \leq N \right\}  \end{align*} of closed balls of radius \[R'_2:=\frac{\rho\left(B_{2,(m_1, m_2, m_3)} \right) }{j^s \beta^{-s} }= \frac{\rho_1 \beta^{2s+2}}{j^{2s+2}}.\]  The collection of balls  that have been removed from $F_1(B_1)$ is \[\hh^*_2(B_1) := \left\{ B_{2,(m_1, m_2, m_3)} : m_1, m_2, m_3 \in \NN \textrm{ such that } 1\leq m_1, m_3 \leq j^d, 1 \leq m_2 \leq N\right\}.\]

Let $t \in \NN$ such that $t \geq 2$.  Continuing thus by recursion, we obtain for the $t$-th step that the set \begin{align}\label{eqn:UppBndHD:DesChainFs}
 F_t(B_1):=F_{t-1}(B_1) \big\backslash \bigcup_{m_1=1}^{j^d}\bigcup_{m_2=1}^{N} \cdots \bigcup_{m_{2t-3}=1}^{j^d}\bigcup_{m_{2t-2}=1}^{N}\bigcup_{m_{2t-1}=1}^{j^d}B_{2,(m_1, \cdots, m_{2t-1})} \end{align} is covered by  the collection \begin{align*}
\mC_t :=\left\{B_{1,(m_1, \cdots, m_{2t})}:m_1,\cdots, m_{2t} \in \NN \textrm{ such that } 1\leq m_1, m_3, \cdots, m_{2t-1} \leq j^d, 1 \leq m_2, m_4, \cdots, m_{2t} \leq N \right\}  \end{align*} of closed balls of radius \[R'_t:=\frac{\rho\left(B_{2,(m_1, \cdots, m_{2t-1})}\right) }{j^s \beta^{-s} }= \frac{\rho_1 \beta^{ts+t}}{j^{ts+t}}.\]  The collection of balls  that have been removed from $F_{t-1}(B_1)$ is \[\hh^*_t(B_1) := \left\{B_{2,(m_1, \cdots, m_{2t-1})}: m_1, \cdots, m_{2t-1} \in \NN \textrm{ such that } 1\leq m_1, \cdots m_{2t-1} \leq j^d, 1 \leq m_2, \cdots, m_{2t-2} \leq N\right\}.\]

Consequently, we have a descending chain \[F_0(B_1) \supset F_1(B_1) \supset \cdots \supset F_t(B_1) \supset \cdots.\]  Let \begin{align}\label{eqn:ProofUpperBndHD:DistingSetDefn}
 F(B_1) := \bigcap_{t=0}^\infty F_t(B_1). \end{align}  Then $F(B_1)$ is covered by $\mC_t$ for any $t \in \NN$.

\begin{lemm}\label{lemm:UppHDEstBndUbitLosing}  We have that \[\dim_H\left(F(B_1)\right) \leq\frac{\log\left(j^d N\right)}{\log\left(j^{s+1}\beta^{-(s+1)} \right)}= \frac{d \log(j) + \log\left(\left(j^s \beta^{-s-1} +1\right)^d - j^{sd} \beta^{-sd}\right)}{\log\left(j^{s+1}\beta^{-(s+1)} \right)} =: \zeta\]
 
\end{lemm}

\begin{proof}  Let $\delta>0$ be such that $R'_t \leq \delta < R'_{t-1}$.  Since $F(B_1)$ is covered by $\mC_t$, we have that \[N_\delta\left(F(B_1)\right) \leq j^{dt} N^t\] where $N$ is given by (\ref{eqn:UpperBndDimCount}).  Hence, we have that \begin{align*}\overline{\dim}_BF(B_1) \leq  \limsup_{t \rightarrow \infty} \frac{t\log \left(j^dN \right)}{-\log (\rho_1) + (t-1) \log\left(j^{s+1}\beta^{-(s+1)}\right)} = \frac{\log\left(j^d N\right)}{\log\left(j^{s+1}\beta^{-(s+1)} \right)},
  \end{align*} from which the desired result follows.
\end{proof}

Now define the countable collections \[\hh(B_1) := \bigcup_{t \in \NN} \hh^*_t(B_1) \quad \textrm{ and } \quad \mK(B_1) := \bigcup_{t \in \NN} \mC_t\] and set $\hh_1:= \hh(B_1)$ and $\mK_1:= \mK(B_1)$.  Note that $\hh(B_1)$ is the collection of all balls removed.

\begin{lemm}\label{lemm:UppHDEstFBComp}  

We have that \[B_1 \cap \bigcup_{B \in \hh(B_1)}B = B_1 \cap\left(F(B_1)\right)^c.\]

\end{lemm}
\begin{proof}
By (\ref{eqn:UppBndHD:DesChainFsInitCase}), we have that \[\left(F_1(B_1)\right)^c = \left(F_0(B_1)\right)^c \cup \bigcup_{B \in \hh^*_1(B_1)} B\] and, by (\ref{eqn:UppBndHD:DesChainFs}), we have, more generally, for any $t \in \NN$, that \[\left(F_t(B_1)\right)^c = \left(F_{t-1}(B_1)\right)^c \cup \bigcup_{B \in \hh^*_t(B_1)} B = \left(F_0(B_1)\right)^c \cup \bigcup_{B \in \hh^*_1(B_1)} B \cup \cdots \cup \bigcup_{B \in \hh^*_t(B_1)} B\] where the second equality follows by recursion.

Consequently, we have that \[\left(F(B_1)\right)^c = \bigcup_{t=1}^\infty \left(F_t(B_1)\right)^c  = \left(F_0(B_1)\right)^c \cup \bigcup_{B \in \hh(B_1)}B,\] which implies the desired result.

\end{proof}

\subsubsection{Constructing finite plays and plays that are winning for Bob}\label{subsubsec:ConstructWinningPlays}  In this section, we construct the elements of $\mW_\infty$ (see \ref{eqn:DeftWInfty:subsubsec:ConstructWinningPlays}) by recursively, via outer and inner recursions, constructing finite plays that are winning for Bob and plays that are winning for Bob.  To do this, we must first construct additional possible choices of Bob's initial ball for accelerated games as follows.  Let \[\mR_3(B_1) =: \left\{D^{[\ell]} : \ell \in \NN \textrm{ such that } 1 \leq \ell \leq 3^d\right\}\] where the ordering is chosen so that $D^{[1]}$ has the same center as $B_1$.  For all $1 \leq \ell \leq 3^d$, let $B_1^{[\ell]}$ be the closed ball given by $\|\cdot\|_\infty$ centered at the center of $D^{[\ell]}$ such that $\rho\left(B_1^{[\ell]}\right) = \rho(B_1)$.  Note that \begin{align}\label{eqn:subsubsec:ConstructWinningPlays:BOneIsBOne}
 B_1^{[1]} = B_1. \end{align}  Let \[\mW_0 := \left\{B_1^{[\ell]} : \ell \in \NN  \textrm{ such that }1 \leq \ell \leq 3^d\right\}.\]  We can regard $\mW_0$ as a collection of $1$-finite plays.  Recall that any $1$-finite play is winning for Bob.  Thus, we can further regard $\mW_0$ as a collection of $1$-finite plays that are winning for Bob.

Step~$1$ of the outer recursion is the proof in Section~\ref{subsubsec:ADistingSet}, the collection of balls to consider is $\hh_0$, and the collection of finite plays that are winning for Bob to consider is $\mW_0$.  The collection of balls removed in Step~$1$ is $\hh_1$.  We now construct finite plays that are winning for Bob using the elements of $\mW_0$ and $\hh_1$ as follows.

\begin{lemm}\label{lemm:AllHolesContainedBOneVariants}
Let $B \in \hh_1 \cup \mK_1$.  Then $B \subset B_1^{[\ell]}$ for some integer $\ell$ such that $1 \leq \ell \leq 3^d$.
\end{lemm}
\begin{proof}

In the proof in Section~\ref{subsubsec:ADistingSet}, we have that the elements of $\hh_1^*(B_1)$ are in $B_1^{[1]}$.  Moreover, the elements of $\mC_t$ come from Proposition~\ref{prop:RepsMinTess} (and Proposition~\ref{prop:RepsMaxTess}) for which Remark~\ref{rmk:proofPropMinTess2ndPart} applies.  Consequently, by the proof in Section~\ref{subsubsec:ADistingSet}, the maximum distance with respect to $\|\cdot\|_\infty$ of a point of an element of $\mC_1$ to some point of $B_1$ is less than $4R_1'\leq \frac{\rho_1}{4}$.  By recursion, the maximum distance with respect to $\|\cdot\|_\infty$ of a point of an element of $\mC_t$ to  some point of $B_1$ is less than \[\frac{\rho_1}{4} + \frac{\rho_1}{4^3} + \cdots + \frac{\rho_1}{4^{2t-1}}.\]  Now, by the proof in Section~\ref{subsubsec:ADistingSet}, any element of $\hh_{t+1}^*(B_1)$ is contained in an element of $\mC_t$.  Consequently, by the geometric series, the maximum distance with respect to $\|\cdot\|_\infty$ of a point of $B$ to some point of $B_1$ is less than $\frac{4 \rho_1}{15}$.  Since the distance function is continuous and the relevant sets are compact, there is a point $x \in B_1$ where this maximum distance is achieved.  By Lemma~\ref{lemm:TesselBallFund}, we have that $x \in D^{[\ell]}$ and, thus, we have that $B \subset B_1^{[\ell]}$ for some $\ell \in \{1, \cdots, 3^d\}$.

 \end{proof}

\begin{rema}
Note that, in the lemma, $\ell$ need not be unique for $B$.
\end{rema}

We will, via an inner recursion, construct finite plays that are winning for Bob.  The collection $\mW_0$ is Iterate~$0$ of the inner recursion.  Let us denote the collection of all finite plays that are winning for Bob that we have constructed at the end of Iterate~$0$ by $\mP_{1,0}$.  Thus, $\mP_{1,0} =  \mW_0$.  By recursion on $t$, we now construct finite plays that are winning for Bob using the elements of $\hh^*_t(B_1)$.  Consider Iterate~$t=1$.  Fix a $B \in \hh^*_1(B_1)$.  By Lemma~\ref{lemm:AllHolesContainedBOneVariants}, $B$ is appendable to some element of $\mP_{1,0}$.  Thus, each $B_1^{[\ell]} \in \mP_{1,0}$ for which $B$ is appendable yields, by Lemma~\ref{lemm:InducedWinningPlayAppend} (with $B^* = B_1$), a $2$-finite play $B_1^{[\ell]} \supset B$ that is winning for Bob.  Note also that every element of $ \hh^*_1(B_1)$ is the end ball of at least one of these finite plays that is winning for Bob.  Also note, in particular, this construction yields, by the proof in Section~\ref{subsubsec:ADistingSet}, that every $2$-finite play $B_1 \supset B_{2,(m_1)}$ is winning for Bob (but the construction may also yield other $2$-finite plays that are winning for Bob using the elements of $\hh^*_1(B_1)$).   Now, let $\mP_{1,1}$ be the collection of all the finite plays constructed at the end of Iterate~$t=1$ and the elements of $\mP_{1,0}$.

Consider Iterate~$t=2$.  Fix a $B_{2,(m_1, m_2, m_3)} \in \hh^*_2(B_1)$.  By Lemma~\ref{lemm:AllHolesContainedBOneVariants}, $B_{2,(m_1, m_2, m_3)}$ is appendable to some element of $\mP_{1,1}$.  (Note that no element of $\hh^*_2(B_1)$ is insertable due to the radii of the relevant balls.)  Thus, by Lemma~\ref{lemm:InducedWinningPlayAppend} (with $B^* =B_{1,(m_1, m_2)}$  as the proof in Section~\ref{subsubsec:ADistingSet} shows), each $\widetilde{B}_1 \supset \cdots \supset  \widetilde{B}_k \in \mP_{1,1}$ for which $B_{2,(m_1, m_2, m_3)}$ is appendable yields a finite play \[\widetilde{B}_1 \supset \cdots \supset  \widetilde{B}_k \supset B_{2,(m_1, m_2, m_3)}\] that is winning for Bob.  Note that every element of $ \hh^*_2(B_1)$ is the end ball of at least one of these finite plays that is winning for Bob.  Now, let $\mP_{1,2}$ be the collection of all the finite plays constructed at the end of Iterate~$t=2$ and the elements of $\mP_{1,1}$.

Continuing thus by recursion, we obtain, at Iterate~$t$, finite plays \[\widetilde{B}_1 \supset \cdots \supset  \widetilde{B}_{\widetilde{k}} \supset B_{2,(m_1, \cdots, m_{2t-1})}\] that are winning for Bob.  Note  that every element of $ \hh^*_t(B_1)$ is the end ball of at least one of these finite plays that is winning for Bob.  Also, note that the elements of $\hh^*_t(B_1)$  for a fixed $t \in \NN$ have the same radius by the proof in Section~\ref{subsubsec:ADistingSet} and, thus, Remark~\ref{rmk:DistinctPlaysWhenDiffBallsSameRad} applies.  Let $\mP_{1,t}$ be the collection of all the finite plays constructed at the end of Iterate~$t$ and the elements of $\mP_{1, t-1}$.

Define \[\mW_1 := \bigcup_{t =0}^\infty \mP_{1,t}\] and note that, by construction, we have $\mW_1 \supset \mW_0$.

Next, for Step~$2$ of the outer recursion, the collection of balls to consider is $\hh_1$, and the collection of finite plays that are winning for Bob and plays that are winning for Bob to consider is $\mW_1$.   Let us denote the collection of all finite plays that are winning for Bob and plays that are winning for Bob that we have constructed at the end of Step~$1$ of the outer recursion by $\mP_{2,0}$.  Thus, $\mP_{2,0} =  \mW_1$.  Step~$2$ of the outer recursion is as follows.  Let $t  \in \NN$.  Now an element $B_{2,(m_1, \cdots, m_{2t-1})} \in \hh^*_t(B_1)$ is Bob's $2$-nd choice of ball for a $(1/j, \beta; S)$-game for which Bob's first choice of ball is $B_{1,(m_1, \cdots, m_{2t-2})} \in \mC_{t-1}$ if $t \geq 2$ or $B_1\in \mC_0$ if $t=1$.  We repeat the proof in Section~\ref{subsubsec:ADistingSet} with $B_{2,(m_1, \cdots, m_{2t-1})}$ in place of $B_1$ to obtain Bob's next choice of ball $B_3$ for this $(1/j, \beta; S)$-game according to a $(1/j, \beta; S)$-winning strategy with Bob's initial choice of ball as $B_{1,(m_1, \cdots, m_{2t-2})}$ if $t \geq 2$ or $B_1$ if $t=1$.  As in the proof in Section~\ref{subsubsec:ADistingSet}, the collection of all such $B_3$ is denoted by $\hh^*_1\left(B_{2,(m_1, \cdots, m_{2t-1})}\right)$.  Furthermore, the proof in Section~\ref{subsubsec:ADistingSet}, gives that \[|\hh^*_1\left( B_{2,(m_1, \cdots, m_{2t-1})}\right)|  = j^d\] and that $\hh\left( B_{2,(m_1, \cdots, m_{2t-1})}\right)$ is the collection of balls removed.  Also, we have that the analog of Lemma~\ref{lemm:UppHDEstBndUbitLosing} gives that $\dim_H\left(F(B_{2,(m_1, \cdots, m_{2t-1})})\right) \leq \zeta$ and the analog of Lemma~\ref{lemm:UppHDEstFBComp} holds. 

Thus we have \begin{align}\label{eqn:NestedHoles2}  \hh_2 := \bigcup_{B \in \hh_1} \hh\left(B\right) \quad \textrm{ and } \quad \hh'_2 := \bigcup_{B \in \hh_1}\hh^*_1(B)
  \end{align} are both countable collections.  Note that $\hh_2$ is the collection of balls removed in Step~$2$ of the outer recursion and will be the collection of balls considered in Step~$3$ of the outer recursion.  Its subcollection $\hh'_2$ is the collection of Bob's next choices of balls made in Step~$2$ of the outer recursion for the various $(1/j, \beta; S)$-games.  By countable stability of Hausdorff dimension, we have that \begin{align*}\dim_H\left(\bigcup_{B\in \hh_1}F(B)\right) \leq \zeta.
  \end{align*}

Now partition $\hh_2$ into subsets $\hh_{2,r}$ where the elements have the same radius $r>0$.  Since $\hh_2$ is countable, the possible values of $r$ are countable.  Let $\{r_i\}_{i=1}^\infty$ denote these possible values of $r$ in descending order.  The iterates of the inner recursion of Step~$2$ of the outer recursion will be over $i$.  Consider Iterate~$i=1$.  By Lemma~\ref{lemm:AllHolesContainedBOneVariants2}, every $\widehat{B} \in \hh_{2,r_1}$ is insertable or appendable to some element of $\mP_{2,0}$.  For each such element of $\mP_{2,0}$, construct the $\widehat{B}$-induced finite play or $\widehat{B}$-induced play, which, by Lemmas~\ref{lemm:InducedWinningPlayInsert},~\ref{lemm:InducedWinningPlayInsert2} and~\ref{lemm:InducedWinningPlayAppend}, are all winning for Bob.  Now, let $\mP_{2,1}$ be the collection of all the finite plays or plays constructed at the end of Iterate~$1$ of the inner recursion and the elements of $\mP_{2,0}$.  Note that, since every element of $\hh_{2,r_1}$ is appendable to some element of $\mP_{2,0}$, then our construction yields that every element of $\hh_{2,r_1}$ is the end ball of some element of $\mP_{2,1}$.
  
  Continuing thus by recursion, we have that every $\widehat{B} \in \hh_{2,r_i}$ is insertable or appendable to some element of $\mP_{2,i-1}$, that $\mP_{2,i}$ is the collection of all the finite plays or plays constructed (in the way analogous to that of Iterate $i=1$) at the end of Iterate~$i$ of the inner recursion and the elements of $\mP_{2,i-1}$.  Our construction yields that every element of $\hh_{2,r_i}$ is the end ball of some element of $\mP_{2,i}$.

  Define \[\mW_2 := \bigcup_{t =0}^\infty \mP_{2,t}\] and note that, by construction, we have $\mW_2 \supset \mW_1 \supset \mW_0$.  
  
 \begin{rema}\label{rmk:WinningPlaysAtStep2}
We note, especially, the following distinguished elements of $\mW_2$, constructed by the above proof as follows.   For any $\widetilde{B} \in \hh_1$, Step~$1$ of the outer recursion gives that there exists at least one element $\widetilde{B}_1 \supset \cdots \supset  \widetilde{B}_k \in \mW_1$ such that $\widetilde{B}_k=\widetilde{B}$. Consequently, since $|\hh_1^*(\widetilde{B})|=j^d$, each such $\widetilde{B}_1 \supset \cdots \supset  \widetilde{B}_k$ generates $j^d$ finite plays that are winning for Bob.  Thus, every element of $\hh'_2$ is the end ball of at least one of these finite plays that is winning for Bob.\footnote{\label{footnote:CentralGames}While many of these finite plays that are winning for Bob are for proper accelerated games for Bob, there are some finite plays that are winning for Bob for the $(1/j, \beta; S)$-game.  In particular, we have the following $3$-finite plays that are winning for Bob for the $(1/j, \beta; S)$-game:  \[B_1 \supset B_{2,(m_1)} \supset B_3,\]  where $B_3 \in \hh_1^*( B_{2,(m_1)})$.  There are $j^{2d}$ of these $3$-finite plays and they are in $\mW_2$.} 

\end{rema}

  Let $p \in \NN$ such that $p \geq 3$.  Continuing thus by recursion, we have Step~$p$ of the outer recursion for which the collection of balls to consider is $\hh_{p-1}$ and the collection of finite plays that are winning for Bob and plays that are winning for Bob to consider is $\mW_{p-1}$.  As in the previous steps of the outer recursion, for every element $B_{2} \in \hh_{p-1}$, we have that \[|\hh^*_1\left( B_{2}\right)|  = j^d,\] $\hh\left( B_{2}\right)$ is the collection of balls removed, the analog of Lemma~\ref{lemm:UppHDEstBndUbitLosing} gives $\dim_H\left(F(B_{2})\right) \leq \zeta$, and the analog of Lemma~\ref{lemm:UppHDEstFBComp} holds.

Thus we have \begin{align}\label{eqn:NestedHolesp}  \hh_p := \bigcup_{B \in \hh_{p-1}} \hh\left(B\right) \quad \textrm{ and } \quad \hh'_p := \bigcup_{B \in \hh_{p-1}}\hh^*_1(B)
  \end{align} are both countable collections.  Note that $\hh_p$ is the collection of balls removed in Step~$p$ of the outer recursion and will be the collection of balls considered in Step~$p+1$ of the outer recursion.  Its subcollection $\hh'_p$ is the collection of Bob's next choices of balls from Step~$p$ of the outer recursion for the various $(1/j, \beta; S)$-games.  By countable stability of Hausdorff dimension, we have that \begin{align*}\dim_H\left(\bigcup_{B\in \hh_{p-1}}F(B)\right) \leq \zeta.
  \end{align*}  We also have that $\mW_p$ is the collection of all the finite plays or plays constructed (in the way analogous to that of Step $p=2$ of the outer recursion) at the end of Step~$p$ of the outer recursion and the elements of $\mW_{p-1}$.  Note that the analog of Remark~\ref{rmk:WinningPlaysAtStep2} holds.

Thus, by countable stability of Hausdorff dimension, we have \begin{align}\label{eqn:MainHDBndUbitLosSets}
\dim_H\left(\bigcup_{p =1}^\infty \bigcup_{B\in \hh_{p-1}}F(B)\right) \leq \zeta.
  \end{align}  We also define \begin{align}\label{eqn:DeftWInfty:subsubsec:ConstructWinningPlays}
 \overline{\hh} :=\bigcup_{p =1}^\infty \bigcup_{B\in \hh_{p-1}} B \quad \textrm{ and } \quad \mW_\infty := \bigcup_{p =0}^\infty \mW_p \end{align} and note that $\mW_0 \subset \mW_1 \subset \cdots$.  An observation similar to that in Footnote~\ref{footnote:CentralGames} holds for $\mW_\infty$.
  
 \begin{rema} \label{rmk:HHpMaxRadDectoZero} Observe that the elements of $\hh_p$ have a maximum radius and this maximum radius decreases to $0$ as $p \rightarrow \infty$.

\end{rema}  

Finally, we state and prove the analog of Lemma~\ref{lemm:AllHolesContainedBOneVariants}, used above in our construction of finite plays and plays winning for Bob.

\begin{lemm}\label{lemm:AllHolesContainedBOneVariants2}
Let $p \in \NN$ and $B \in \hh_p$.  Then $B \subset B_1^{[\ell]}$ for some integer $\ell$ such that $1 \leq \ell \leq 3^d$.
\end{lemm}
 \begin{proof} Lemma~\ref{lemm:AllHolesContainedBOneVariants} gives the desired result for $p=1$.  Let $p = 2$.  By the above proof for Step~2 of the outer recursion, we have chosen a $\widetilde{B} \in \hh_1$ for which to repeat proof in Section~\ref{subsubsec:ADistingSet}.  Thus, there exists a $t_1 \in \NN \cup \{0\}$ for which $\widetilde{B} \in \hh^*_{t_1+1}(B_1)$.   By the proof in Section~\ref{subsubsec:ADistingSet}, we further have that $\widetilde{B}$ is contained in an element of $\mC_{t_1}$ and, thus, \begin{align}\label{Eqn1:lemm:AllHolesContainedBOneVariants2}
 \rho(\widetilde{B}) = \frac \beta j \left( \frac{\rho_1 \beta^{t_1s+t_1}}{j^{t_1s+t_1}}\right) \leq \rho_1/4^{2t_1+1}. \end{align}  Applying the analog of the proof of Lemma~\ref{lemm:AllHolesContainedBOneVariants}, we have that the maximum distance with respect to $\|\cdot\|_\infty$ of a point of an element of $\hh^*_{t_2+1}(\widetilde{B})$ to  some point of $\widetilde{B}$ is less than \begin{align}\label{Eqn3:lemm:AllHolesContainedBOneVariants2}
 \frac{\rho(\widetilde{B})}{4} + \frac{\rho(\widetilde{B})}{4^3} + \cdots + \frac{\rho(\widetilde{B})}{4^{2t_2-1}} \leq \frac{\rho_1}{4^{2t_1+2}} + \frac{\rho_1}{4^{2t_1+4}} + \cdots + \frac{\rho_1}{4^{2(t_1+t_2)}}. \end{align}  Consequently, since $\widetilde{B}$ is contained in an element of $\mC_{t_1}$,  the maximum distance with respect to $\|\cdot\|_\infty$ of a point of an element of $\hh^*_{t_2+1}(\widetilde{B})$ to some point of $B_1$ is less than\begin{align}\label{Eqn2:lemm:AllHolesContainedBOneVariants2}
  \frac{\rho_1}{4} + \frac{\rho_1}{4^3} + \cdots + \frac{\rho_1}{4^{2t_1-1}}+ \frac{\rho_1}{4^{2t_1+1}} + \frac{\rho_1}{4^{2t_1+3}} + \cdots + \frac{\rho_1}{4^{2(t_1+t_2)-1}}. \end{align}  Here, we have that $t_2 \in \NN \cup \{0\}$.  Consequently, by the geometric series, the maximum distance with respect to $\|\cdot\|_\infty$ of a point of an element of $\hh_2$ to some point of $B_1$ is less than $\frac{4 \rho_1}{15}$ and the desired result for $p=2$ follows as it does in Lemma~\ref{lemm:AllHolesContainedBOneVariants}.

 Let $\widetilde{B}_{[0]} := B_1$, $p \geq 3$, and $\tau$ be an integer such that $1 \leq \tau \leq p$.  Analogous to the proof for $p=2$, we have, for $p \geq 3$, that $\widetilde{B}_{[\tau]} \in \hh_\tau$ and $t_\tau \in  \NN \cup \{0\}$ for which \[\widetilde{B}_{[\tau]} \in \hh^*_{t_{\tau}+1}\left(\widetilde{B}_{[\tau-1]}\right).\]  For $1 \leq \tau \leq p-1$, we claim that the following assertions both hold.  \begin{description}
\item[Assertion~(1)] We have that $\rho(\widetilde{B}_{[\tau]}) \leq \rho_1/4^{2(t_1+ \cdots + t_{\tau})+\tau}$.
\item[Assertion~(2)] The maximum distance with respect to $\|\cdot\|_\infty$ of a point of $\widetilde{B}_{[\tau+1]}$ to some point of $B_1$ is less than\begin{align*}
  \sum_{n=1}^{t_1+t_2+ \cdots + t_{\tau+1}} \frac{\rho_1}{4^{2n-1}}.\end{align*} 
\end{description}  We prove the claim by induction on $\tau$ for $1 \leq \tau \leq p-1$.  The initial case $\tau =1$ is given in (\ref{Eqn1:lemm:AllHolesContainedBOneVariants2}, \ref{Eqn2:lemm:AllHolesContainedBOneVariants2}).  Assume that the two assertions hold for $\tau -1$.  The the analog of the proof of (\ref{Eqn1:lemm:AllHolesContainedBOneVariants2}) gives \begin{align*}
 \rho(\widetilde{B}_{[\tau]}) \leq \rho\left( \widetilde{B}_{[\tau-1]}\right)/4^{2t_\tau+1}, \end{align*} which, together with the induction hypothesis for Assertion~(1), yields Assertion~(1).  
 
 Likewise, the analog of the proof of (\ref{Eqn3:lemm:AllHolesContainedBOneVariants2}) gives the maximum distance with respect to $\|\cdot\|_\infty$ of a point of $\widetilde{B}_{[\tau+1]}$ to  some point of $\widetilde{B}_{[\tau]}$ is less than  \begin{align}\label{Eqn4:lemm:AllHolesContainedBOneVariants2}
 \frac{\rho(\widetilde{B}_{[\tau]})}{4} + \frac{\rho(\widetilde{B}_{[\tau]})}{4^3} + \cdots + \frac{\rho(\widetilde{B}_{[\tau]})}{4^{2t_{\tau+1}-1}} &\leq \frac{\rho\left( \widetilde{B}_{[\tau-1]}\right)}{4^{2t_\tau+2}} + \frac{\rho\left( \widetilde{B}_{[\tau-1]}\right)}{4^{2t_\tau+4}} + \cdots + \frac{\rho\left( \widetilde{B}_{[\tau-1]}\right)}{4^{2(t_\tau+t_{\tau+1})}}\\\nonumber
 &\leq \frac{\rho_1}{4^{2(t_1+ \cdots + t_{\tau})+\tau+1}} + \frac{\rho_1}{4^{2(t_1+ \cdots + t_{\tau})+\tau+3}}+ \cdots + \frac{\rho_1}{4^{2(t_1+ \cdots + t_{\tau+1})+\tau-1}}\\\nonumber
 &\leq \frac{\rho_1}{4^{2(t_1+ \cdots + t_{\tau})+1}} + \frac{\rho_1}{4^{2(t_1+ \cdots + t_{\tau})+3}}+ \cdots + \frac{\rho_1}{4^{2(t_1+ \cdots + t_{\tau+1})-1}}.\end{align}  Here, the second inequality follows by the induction hypothesis for Assertion~(1).  Applying (\ref{Eqn4:lemm:AllHolesContainedBOneVariants2}) and the induction hypothesis for Assertion~(2), yields Assertion~(2).  Now the analog of the $p=2$ proof above yields the desired result for $p$.  This proves the lemma.

\end{proof}

 \subsubsection{A cover for $S$ and the conclusion of the proof of Theorem~\ref{thm:UpBndHDLosingSets} for $1/\beta$ an integer}  
In this section, we show that \[\bigcup_{p =1}^\infty \bigcup_{B\in \hh_{p-1}}F(B)\] is a cover for $B_1 \cap S$ (Proposition~\ref{lemm:UbiqLosSetInUnionHoles}) and use this result to conclude the proof of Theorem~\ref{thm:UpBndHDLosingSets}  for $1/\beta$ an integer.  Let \[ \mW:=\bigcap_{p =1}^\infty \bigcap_{B\in \hh_{p-1}}\left(F(B)\right)^c  \cap B_1.\]  For now, assume that $\mW \neq \emptyset$ and let \[\boldsymbol{x}:= (x_1, \cdots, x_d)^t\in \mW. \]  
 
 \begin{lemm}\label{lemm:CoverForS:BallSeq}  For every $p \in \NN$, there exists an element $B^\#_p \in \hh_{p-1}$ such that $\boldsymbol{x} \in B^\#_p$.  
 
\end{lemm}
 \begin{proof}
The proof is by induction on $p$.  For $p=1$, since $\mW \subset B_1$ and $\hh_0 = \{B_1\}$, we have that $B^\#_1 = B_1$.  For $p=2$, we note that \[\mW \subset \left(F( B_1)\right)^c  \cap B_1= B_1 \cap \bigcup_{\widecheck{B} \in \hh(B_1)}\widecheck{B}\] by Lemma~\ref{lemm:UppHDEstFBComp}, and, hence, since $\hh_1 = \hh(B_1)$, there exists $B^\#_2 \in \hh_{1}$ such that $\boldsymbol{x} \in B^\#_2$.  

Now, since $\boldsymbol{x} \in \mW$, we have that $\boldsymbol{x} \in \left(F( B^\#_2)\right)^c$.  Thus, we have that \[\boldsymbol{x} \in  \left(F( B^\#_2)\right)^c \cap B^\#_2 = B^\#_2 \cap \bigcup_{\widecheck{B} \in \hh(B^\#_2)}\widecheck{B},\] where the equality comes from the analog of Lemma~\ref{lemm:UppHDEstFBComp} for $B^\#_2$.  Note the analog of Lemma~\ref{lemm:UppHDEstFBComp} is proved in Section~\ref{subsubsec:ConstructWinningPlays}.  Consequently, we have, by (\ref{eqn:NestedHoles2}), that there exists $B^\#_3 \in \hh_{2}$ such that $\boldsymbol{x} \in B^\#_3$.  Continuing thus by induction, we obtain $\boldsymbol{x} \in \left(F( B^\#_{p-1})\right)^c \cap B^\#_{p-1}$, which by (\ref{eqn:NestedHolesp}) and the analog of Lemma~\ref{lemm:UppHDEstFBComp} gives that $B^\#_p \in \hh_{p-1}$.  This proves the lemma.
\end{proof}

Let $\{ B^\#_p\}_{p=1}^\infty$ be the sequence of balls constructed in Lemma~\ref{lemm:CoverForS:BallSeq}.  By Lemma~\ref{lemm:CoverForS:BallSeq} and Remark~\ref{rmk:HHpMaxRadDectoZero}, we have that \begin{align}\label{eqn:ACovForS:xinSeqBalls}
 \{\boldsymbol{x}\}= \bigcap_{p=1}^\infty B^\#_p. \end{align}

 \begin{lemm}\label{lemm:CoverForS:PlayExists}  There exists a subsequence $\{B^\#_{p_\kappa}\}_{\kappa=1}^\infty$ such that \[\{\boldsymbol{x}\} = \bigcap_{\kappa=1}^\infty B^\#_{p_\kappa}\] and $B^\#_{p_1} \supsetneq B^\#_{p_2} \supsetneq B^\#_{p_3} \supsetneq \cdots$.
 
\end{lemm}

\begin{proof}

Let $p_{0,\eta} = \eta$ for $\eta \in \NN$.  Then the sequence $\{p_{0,\eta}\}_{\eta = 1}^\infty$ is the sequence $\{p\}_{p=1}^\infty$.  Let \begin{align*}  
 B^\#_{p_{0,\eta}}=:[a_1^{(0,\eta)}, b_1^{(0,\eta)}] \times \cdots \times [a_d^{(0,\eta)}, b_d^{(0,\eta)}]. \end{align*}  We now recursively, for $\xi = 1, \cdots ,d$, construct subsequences of the sequence $\{p_{{0,\eta}}\}_{\eta = 1}^\infty$ as follows.  Let $\xi =1$.  Now we have three possible cases:  \begin{enumerate}
\item $x_\xi = a_\xi^{(\xi-1,\eta)}$ for infinitely many $\eta \in \NN$, 
\item $x_\xi = b_\xi^{(\xi-1,\eta)}$ for infinitely many $\eta \in \NN$, or
\item $x_\xi \in (a_\xi^{(\xi-1,\eta)}, b_\xi^{(\xi-1,\eta)})$ for infinitely many $\eta \in \NN$.
\end{enumerate}  Note that all three cases can occur.  We claim that at least one of the three cases must occur.  We now prove this claim.  If the claim does not hold, then there exists an $\eta' \in \NN$ such that, for all $\eta \geq \eta'$, we have that $x_\xi \notin [a_\xi^{(\xi-1,\eta)}, b_\xi^{(\xi-1,\eta)}]$, which implies that \[\boldsymbol{x} \notin \bigcap_{\eta=1}^\infty B^\#_{p_{\xi-1,\eta}}.\]  This contradicts (\ref{eqn:ACovForS:xinSeqBalls}) and proves the claim.

Note that, in each of the three cases, the infinitely many $\eta \in \NN$ yields a subsequence $\Fs$ of the sequence $\left(p_{{\xi-1,\eta}}\right)_{\eta = 1}^\infty$.  Our desired subsequence is either $\Fs$ or a subsequence of $\Fs$.  If Case~(1) holds, then denote the subsequence $\Fs$ given by the infinitely many $\eta \in \NN$ for Case~(1) as follows:  \begin{align}\label{eqn:ACovForS:SubSeq2superseq}
 p^*_{\xi,1} < p^*_{\xi,2} < p^*_{\xi,3} < \cdots .\end{align}  The desired subsequence is constructed from $\Fs$ as follows.  Let $p_{\xi,1}:=p^*_{\xi,1}$.  By Remark~\ref{rmk:HHpMaxRadDectoZero}, we can choose $p^*_{\xi,\eta} > p_{\xi,1}$ such that \begin{align*}
 \rho\left(B^\#_{p_{\xi,1}}\right) > \rho\left(B^\#_{p^*_{\xi,\eta}}\right). \end{align*}   Let $p_{\xi,2}:=p^*_{\xi,\eta}$.  Continuing thus by recursion, we obtain our desired subsequence \begin{align}\label{eqn:ACovForS:SubSeq2}
 p_{\xi,1} < p_{\xi,2} < p_{\xi,3} < \cdots, \end{align} and, letting \begin{align}\label{eqn:ACovForS:BallSubSeq}B^\#_{p_{\xi,\eta}} =: [a_1^{(\xi,\eta)}, b_1^{(\xi,\eta)}] \times \cdots \times [a_d^{(\xi,\eta)}, b_d^{(\xi,\eta)}],
  \end{align} which, in particular, gives \begin{align}\label{eqn:ACovForS:ProjectionSequence}
 \pi_\xi\left(B^\#_{p_{\xi,\eta}}\right)=  [a_\xi^{(\xi,\eta)}, b_\xi^{(\xi,\eta)}], \end{align} we note that \begin{align}\label{eqn:ACovForS:Case1Cond}    b_\xi^{(\xi,1)} > b_\xi^{(\xi,2)}> b_\xi^{(\xi,3)}> \cdots > x_\xi = a_\xi^{(\xi,1)} = a_\xi^{(\xi,2)} = \cdots
  \end{align} holds.  Note that $\pi_\xi(\cdot)$ is defined in (\ref{eqn:DefnProjMaps}).

 Otherwise, Case~(1) does not hold.  If Case~(2) holds, then denote the subsequence $\Fs$ given by the infinitely many $\eta \in \NN$ for Case~(2) by (\ref{eqn:ACovForS:SubSeq2superseq}).  The desired subsequence is constructed from $\Fs$ in the analogous way to that in Case~(1), yielding our desired subsequence (\ref{eqn:ACovForS:SubSeq2}).  Letting (\ref{eqn:ACovForS:BallSubSeq}) hold, which gives (\ref{eqn:ACovForS:ProjectionSequence}), we note that the analog of (\ref{eqn:ACovForS:Case1Cond}) is \begin{align}\label{eqn:ACovForS:Case2Cond} a_\xi^{(\xi,1)} < a_\xi^{(\xi,2)} < a_\xi^{(\xi,3)}< \cdots < x_\xi = b_\xi^{(\xi,1)} = b_\xi^{(\xi,2)}= \cdots.
  \end{align} 
 
 Otherwise, both Cases~(1) and~(2) do not hold.  Therefore, by the claim, Case~(3) must hold, and let us denote the subsequence $\Fs$ given by the infinitely many $\eta \in \NN$ for Case~(3) by (\ref{eqn:ACovForS:SubSeq2superseq}).  By Remark~\ref{rmk:HHpMaxRadDectoZero}, we can choose $p^*_{\xi,\eta} > p_{\xi,1}$ such that \begin{align*}\pi_\xi\left(B^\#_{p^*_{\xi,\eta}}\right) \subset \inte \left( \pi_\xi\left(B^\#_{p_{\xi,1}}\right)\right). \end{align*}   Let $p_{\xi,2}:=p^*_{\xi,\eta}$.  Continuing thus by recursion, we obtain our desired subsequence (\ref{eqn:ACovForS:SubSeq2}).  Letting (\ref{eqn:ACovForS:BallSubSeq}) hold, which gives (\ref{eqn:ACovForS:ProjectionSequence}), we note that the analog of (\ref{eqn:ACovForS:Case1Cond}) or (\ref{eqn:ACovForS:Case2Cond}) is \begin{align}\label{eqn:ACovForS:Case3Cond}    a_\xi^{(\xi,1)} < a_\xi^{(\xi,2)} < a_\xi^{(\xi,3)}< \cdots < x_\xi < \cdots < b_\xi^{(\xi,3)} <b_\xi^{(\xi,2)} < b_\xi^{(\xi,1)}.
  \end{align}

  Consequently, we have constructed a subsequence $\{p_{\xi,\eta}\}_{\eta = 1}^\infty$ from the sequence $\{p_{\xi-1,\eta}\}_{\eta = 1}^\infty$, and we have (\ref{eqn:ACovForS:BallSubSeq}).  Continuing thus by induction for $\xi =2, \cdots, d$, we have the subsequence $\{p_{d,\eta}\}_{\eta = 1}^\infty$ and  \[B^\#_{p_{d,\eta}}=[a_1^{(d,\eta)}, b_1^{(d,\eta)}] \times \cdots \times [a_d^{(d,\eta)}, b_d^{(d,\eta)}] .\] 
  
  Next we claim that \begin{align}\label{eqn:ACovForS:BallIncSubseqMain} B^\#_{p_{d,\eta}} \supsetneq B^\#_{p_{d,\eta+1}}.
  \end{align}  We now prove the claim.  For any integer $\xi$ such that $1 \leq \xi \leq d$, the above proof gives that Case~(1) holds, Case~(1) does not hold and Case~(2) holds, or Cases~(1) and~(2) do not hold and Case~(3) holds.  If Case~(1) holds, then (\ref{eqn:ACovForS:Case1Cond}) implies that \begin{align}\label{eqn:ACovForS:BallIncSubseq}
 \pi_\xi\left(B^\#_{p_{d,\eta}}\right) \supsetneq  \pi_\xi\left(B^\#_{p_{d,\eta+1}}\right). \end{align}  If Case~(2) holds, then (\ref{eqn:ACovForS:Case2Cond}) implies (\ref{eqn:ACovForS:BallIncSubseq}).  Finally, if Case~(3) holds, then (\ref{eqn:ACovForS:Case3Cond}) implies (\ref{eqn:ACovForS:BallIncSubseq}).  Since (\ref{eqn:ACovForS:BallIncSubseq}) holds for all $\xi$, we have that (\ref{eqn:ACovForS:BallIncSubseqMain}) holds.  This proves the claim.
 
 Let $p_\kappa = p_{d,\kappa}$ for $\kappa \in \NN$.  Then the claim gives the desired result.  This proves the lemma.\end{proof}

Let $\{B^\#_{p_\kappa}\}_{\kappa=1}^\infty$ be as in Lemma~\ref{lemm:CoverForS:PlayExists}. 
\begin{lemm}\label{lemm:ACoverForS:PlayWinBob}
The sequence \begin{align}\label{eqn:ACoverForS:PlayWinBob}
 B^\#_{p_1} \supset  B^\#_{p_2} \supset B^\#_{p_3} \supset \cdots \end{align} is a play for some $(1/j, \beta; S)$-accelerated game for Bob.  Moreover, the play (\ref{eqn:ACoverForS:PlayWinBob}) is winning for Bob.
\end{lemm}

\begin{proof}
The proofs in Sections~\ref{subsubsec:ADistingSet} and~\ref{subsubsec:ConstructWinningPlays} imply that, for all $\kappa \in \NN$, we have that $\rho\left(B^\#_{p_\kappa}\right) = \rho_1 (\beta/j)^\tau$ where $\tau \in \NN \cup \{0\}$ and depends on $\kappa$.  Furthermore, Lemma~\ref{lemm:CoverForS:PlayExists} gives that $\tau$ is strictly increasing as $\kappa$ increases.  Consequently, we have that (\ref{eqn:ACoverForS:PlayWinBob}) is a play for some $(1/j, \beta; S)$-accelerated game for Bob.

We now show that the play (\ref{eqn:ACoverForS:PlayWinBob}) is winning for Bob for this accelerated game for Bob by showing that (\ref{eqn:ACoverForS:PlayWinBob}) is the restriction of an element $\Fb$ of $\mW_\infty$ (see Remark~\ref{rmk:RestrictionPlay}).  Since the elements of $\mW_\infty$ are finite plays that are winning for Bob or plays that are winning for Bob and since a restriction is an infinite sequence, we must have that $\Fb$ is a play that is winning for Bob.

First consider $B^\#_{p_1} \neq B_1$.  Since $B^\#_{p_1} \in \overline{\hh}$, the proof in Section~\ref{subsubsec:ConstructWinningPlays} gives that there exists an integer $\ell$ such that $1 \leq \ell \leq 3^d$ and that \begin{align}\label{lemm:ACoverForS:PlayWinBob:Eqn1}
 B^{[\ell]}_1\supset B^\#_{p_1} \in \mW_\infty. \end{align}  Likewise, we must have \begin{align}\label{lemm:ACoverForS:PlayWinBob:Eqn2}
 B^{[\ell]}_1\supset B^\#_{p_2} \in \mW_\infty.\end{align}  Consequently, the proof in Section~\ref{subsubsec:ConstructWinningPlays} gives that $B^\#_{p_1}$ is inserted into (\ref{lemm:ACoverForS:PlayWinBob:Eqn2}) or $ B^\#_{p_2}$ is appended onto (\ref{lemm:ACoverForS:PlayWinBob:Eqn1}), yielding  \begin{align}\label{lemm:ACoverForS:PlayWinBob:Eqn3}
 B^{[\ell]}_1\supset B^\#_{p_1} \supset  B^\#_{p_2} \in \mW_\infty. \end{align}  Analogous to (\ref{lemm:ACoverForS:PlayWinBob:Eqn1},~\ref{lemm:ACoverForS:PlayWinBob:Eqn2}), we have \begin{align}\label{lemm:ACoverForS:PlayWinBob:Eqn4}
 B^{[\ell]}_1\supset B^\#_{p_3} \in \mW_\infty. \end{align}  Then we have the following possible cases of which one must occur:  \begin{itemize}
\item $B^\#_{p_1}$ is inserted into (\ref{lemm:ACoverForS:PlayWinBob:Eqn4}) to obtain   \begin{align}\label{lemm:ACoverForS:PlayWinBob:Eqn5}
 B^{[\ell]}_1\supset B^\#_{p_1} \supset  B^\#_{p_3} \in \mW_\infty \end{align} and then $B^\#_{p_2}$ is inserted into (\ref{lemm:ACoverForS:PlayWinBob:Eqn5}) to obtain  \begin{align}\label{lemm:ACoverForS:PlayWinBob:Eqn6}
 B^{[\ell]}_1\supset B^\#_{p_1} \supset  B^\#_{p_2}\supset B^\#_{p_3} \in \mW_\infty \end{align}
\item $B^\#_{p_2}$ is inserted into (\ref{lemm:ACoverForS:PlayWinBob:Eqn4}) to obtain  \begin{align}\label{lemm:ACoverForS:PlayWinBob:Eqn7}
 B^{[\ell]}_1\supset B^\#_{p_2} \supset  B^\#_{p_3} \in \mW_\infty \end{align} and then $B^\#_{p_1}$ is inserted into (\ref{lemm:ACoverForS:PlayWinBob:Eqn7}) to obtain (\ref{lemm:ACoverForS:PlayWinBob:Eqn6})

 \item $B^\#_{p_3}$ is appended onto (\ref{lemm:ACoverForS:PlayWinBob:Eqn1}) to obtain (\ref{lemm:ACoverForS:PlayWinBob:Eqn5}) and then $B^\#_{p_2}$ is inserted into (\ref{lemm:ACoverForS:PlayWinBob:Eqn5}) to obtain (\ref{lemm:ACoverForS:PlayWinBob:Eqn6}).
 
 \item $B^\#_{p_2}$ is appended onto (\ref{lemm:ACoverForS:PlayWinBob:Eqn1}) to obtain (\ref{lemm:ACoverForS:PlayWinBob:Eqn3}) and then $B^\#_{p_3}$ is appended onto (\ref{lemm:ACoverForS:PlayWinBob:Eqn3}) to obtain (\ref{lemm:ACoverForS:PlayWinBob:Eqn6})

\item $B^\#_{p_3}$ is appended onto (\ref{lemm:ACoverForS:PlayWinBob:Eqn2}) to obtain (\ref{lemm:ACoverForS:PlayWinBob:Eqn7}) and then $B^\#_{p_1}$ is inserted into (\ref{lemm:ACoverForS:PlayWinBob:Eqn7}) to obtain (\ref{lemm:ACoverForS:PlayWinBob:Eqn6}).
 
 \item $B^\#_{p_1}$ is inserted into (\ref{lemm:ACoverForS:PlayWinBob:Eqn2}) to obtain (\ref{lemm:ACoverForS:PlayWinBob:Eqn3}) and then $B^\#_{p_3}$ is appended onto (\ref{lemm:ACoverForS:PlayWinBob:Eqn3}) to obtain (\ref{lemm:ACoverForS:PlayWinBob:Eqn6}).
 
\end{itemize}

Consequently, we have that (\ref{lemm:ACoverForS:PlayWinBob:Eqn6}) holds.  Continuing thus by recursion on $\kappa$ yields that $\Fb$ is \[B^{[\ell]}_1\supset B^\#_{p_1} \supset  B^\#_{p_2}\supset B^\#_{p_3} \supset B^\#_{p_3} \supset \cdots.\]  By Remark~\ref{rmk:RestrictionPlay}, we have that the play (\ref{eqn:ACoverForS:PlayWinBob}) is winning for Bob as desired.  Finally, for $B^\#_{p_1} = B_1$, we replace $\ell$ with $1$ (see~(\ref{eqn:subsubsec:ConstructWinningPlays:BOneIsBOne})) and $B^{[\ell]}_1\supset B^\#_{p_1}$ with $B_1$ in (a simplified version of) the above proof.  This proves the lemma.

\end{proof}

Applying Lemmas~\ref{lemm:CoverForS:PlayExists} and~\ref{lemm:ACoverForS:PlayWinBob} yields \begin{align}\label{Eqn:ACoverForS:WinSComple}
\mW \subset S^c  \end{align}  for $\mW \neq \emptyset$.  For $\mW = \emptyset$, (\ref{Eqn:ACoverForS:WinSComple}) also holds.

\begin{prop}\label{lemm:UbiqLosSetInUnionHoles}
We have that \[B_1 \cap S \subset \bigcup_{p =1}^\infty \bigcup_{B\in \hh_{p-1}}F(B).\] 
\end{prop}
\begin{proof}  By (\ref{Eqn:ACoverForS:WinSComple}), we  have \begin{align*} \bigcap_{p =1}^\infty \bigcap_{B\in \hh_{p-1}}\left(F(B)\right)^c= \mW \cup \left(\bigcap_{p =1}^\infty \bigcap_{B\in \hh_{p-1}}\left(F(B)\right)^c  \cap B_1^c \right) \subset S^c \cup B_1^c,
  \end{align*} which implies the desired result.

\end{proof}

Applying (\ref{eqn:MainHDBndUbitLosSets}), Proposition~\ref{lemm:UbiqLosSetInUnionHoles}, the monotonicity of Hausdorff dimension, the remarks at the beginning of Section~\ref{subsec:ProofThm:UpBndHDLosingSets}, and Proposition~\ref{prop:AuxBndHDCalc} (below) proves Theorem~\ref{thm:UpBndHDLosingSets} for the case that $1/\beta$ is an integer.

\subsubsection{The conclusion of the proof of Theorem~\ref{thm:UpBndHDLosingSets}}  We now prove the general case for which $1/\beta$ need not be an integer by giving the necessary changes to the proof of the case for which $1/\beta$ is an integer.  The changes are to Section~\ref{subsubsec:ADistingSet}.  In particular, the complete tessellation $\overline{\B_{1,(m_1)}}$ is replaced by a complete tessellation $\B_{\RR^d}(\boldsymbol{x'}, R_1')$ for some $\boldsymbol{x'} \in \RR^d$.  We replace $N$ from (\ref{eqn:UpperBndDimCount}) with\begin{align}\label{eqn:UpperBndDimCountReals}
 N_\RR:= \left(\lceil j^s \beta^{-s-1}\rceil +1\right)^d -  \left(\lfloor j^s \beta^{-s}\rfloor -1\right)^d. \end{align}  The minimal outer tessellation and maximal inner tessellation used in Propositions~\ref{prop:RepsMinTess} and~\ref{prop:RepsMaxTess} to obtain the collection $\mC_{1,(m_1)}$ which covers (\ref{eq:subsubsec:ADistingSetA1RemoveB2}) are now appropriate subtessellations of $\B_{\RR^d}(\boldsymbol{x'}, R_1')$.  Likewise, the complete tessellation $\overline{B_{2,(m_1, \cdots, m_{2t-1})}}$ at the $t$-th step of the recursion is replaced by a complete tessellation $\B_{\RR^d}(\boldsymbol{x'}, R_t')$ and the minimal outer tessellation and maximal inner tessellation used in Propositions~\ref{prop:RepsMinTess} and~\ref{prop:RepsMaxTess} come from $\B_{\RR^d}(\boldsymbol{x'}, R_t')$.  (Note that the choice of $\boldsymbol{x'}$ does not matter and can be different for each step of the recursion over $t$.)

Finally, $\zeta$ is replaced by \[\zeta_\RR:=  \frac{d \log(j) + \log\left(\left(\lceil j^s \beta^{-s-1}\rceil +1\right)^d -  \left(\lfloor j^s \beta^{-s}\rfloor -1\right)^d\right)}{\log\left(j^{s+1}\beta^{-(s+1)} \right)},\]  which agrees with our replacement of $N$ with $N_\RR$, and Lemma~\ref{lemm:UppHDEstBndUbitLosing} is replaced by

\begin{lemm}\label{lemm:UppHDEstBndUbitLosingReals}  We have that \[\dim_H\left(F(B_1)\right) \leq\frac{\log\left(j^d N_\RR\right)}{\log\left(j^{s+1}\beta^{-(s+1)} \right)}= \zeta_\RR\]
 
\end{lemm}

\begin{proof}
 Replace $N$ with $N_\RR$, (\ref{eqn:UpperBndDimCount}) with (\ref{eqn:UpperBndDimCountReals}), and $\zeta$ with $\zeta_\RR$ in the proof of Lemma~\ref{lemm:UppHDEstBndUbitLosing}.
\end{proof}

The rest of the proof of Theorem~\ref{thm:UpBndHDLosingSets} for the general case for which $1/\beta$ need not be an integer is the same as in the proof of Theorem~\ref{thm:UpBndHDLosingSets} for the case for which $1/\beta$ is an integer.  This proves Theorem~\ref{thm:UpBndHDLosingSets}.

\subsection{Proposition~\ref{prop:AuxBndHDCalc}} The following proposition is used in Section~\ref{subsec:ProofThm:UpBndHDLosingSets} to show that the right-hand sides of (\ref{eqn:thm:UpBndHDLosingSets1}) and (\ref{eqn:thm:UpBndHDLosingSets2}) are strictly less than $d$.

\begin{prop}\label{prop:AuxBndHDCalc}  Let $d \in \NN$, $s\geq d, j \geq 2$ be integers, and $0 < \beta \leq 1/2$.  Then we have that \begin{align*}
 j^d\left(\left(\lceil j^s \beta^{-s-1}\rceil +1\right)^d -  \left(\lfloor j^s \beta^{-s}\rfloor -1\right)^d \right)< j^{(s+1)d}\beta^{-(s+1)d}. \end{align*} 
 
\end{prop}

\begin{proof}  Consider the case $d \geq 2$ first.  We have \begin{align*}j^d\left(\left(\lceil j^s \beta^{-s-1}\rceil +1\right)^d -  \left(\lfloor j^s \beta^{-s}\rfloor -1\right)^d \right)&\leq j^{(s+1)d} \beta^{-(s+1)d}\left(\left(1+ \frac{2\beta^{s+1}}{j^s} \right)^d -  \left(\beta - \frac {2\beta^{s+1}}{j^s} \right)^d \right)\\
&\leq j^{(s+1)d} \beta^{-(s+1)d}\left(\left(1+ \frac{\beta^{s+1}}{2^{d-1}} \right)^d -  \left(\beta - \frac {\beta^{s+1}}{2^{d-1}} \right)^d \right).
\end{align*}  By Taylor's theorem for the function $x^d$, we have that \[\left(1+ \frac{\beta^{s+1}}{2^{d-1}} \right)^d = 1 + \frac{d \beta^{s+1}}{2^{d-1}} + \varepsilon\left(1+ \frac{\beta^{s+1}}{2^{d-1}}\right)\] where, for $x \in [1-\beta^{s+1}/2^{d-1}, 1+\beta^{s+1}/2^{d-1}]$, we have that \begin{align*}
 |\varepsilon(x)|&\leq \frac {d(d-1)} 2 \left( 1+ \frac{\beta^{s+1}}{2^{d-1}}\right)^{d-2} \left(\frac{\beta^{s+1}}{2^{d-1}} \right)^2 \\ &\leq \frac {d(d-1)} 2 \left( 1+ \frac{1}{2^{2d}}\right)^{d-2} \left(\frac{\beta^{s+1}}{2^{d-1}} \right)^2 < \frac 1 2 \frac{\beta^{2s+2}}{2^{d-2}}
 \leq  \frac{\beta^{2d}}{2^{d+1}}.
 \end{align*}
 
 Similiarly, we have that  \[\left(1- \frac{\beta^{s}}{2^{d-1}} \right)^d = 1 - \frac{d \beta^{s}}{2^{d-1}} + \varepsilon'\left(1- \frac{\beta^{s}}{2^{d-1}}\right)\] where, for $x \in [1- \frac{\beta^{s}}{2^{d-1}}, 1+ \frac{\beta^{s}}{2^{d-1}}]$, we have that \begin{align*}
 |\varepsilon'(x)|&\leq \frac {d(d-1)} 2 \left( 1+ \frac{\beta^{s}}{2^{d-1}}\right)^{d-2} \left(\frac{\beta^{s}}{2^{d-1}} \right)^2 \\ &\leq \frac {d(d-1)} 2 \left( 1+ \frac{1}{2^{2d-1}}\right)^{d-2} \left(\frac{\beta^{2s}}{2^{2d-1}} \right)< \frac{\beta^{2s}}{2^{d-1}}
 \leq  \frac{\beta^{2d}}{2^{d-1}}.
 \end{align*}

Finally, to show the desired result, it suffices to show that $\beta^d$ is strictly larger than \[M:= \frac{d \beta^{s+1}}{2^{d-1}}+ \frac{d \beta^{s+d}}{2^{d-1}}+\frac{\beta^{2d}}{2^{d+1}}+\frac{\beta^{3d}}{2^{d-1}} \leq \beta^d\left(\frac d {2^d} + \frac d {2^{2d-1}} +\frac{1}{2^{2d+1}}+\frac{1}{2^{3d-1}}\right).\]  As the maximum of $\left(\frac d {2^d} + \frac d {2^{2d-1}} +\frac{1}{2^{2d+1}}+\frac{1}{2^{3d-1}}\right) < 1$ when $d \geq 2$, we have that $M < \beta^d$.  This proves the desired result for $d \geq 2$.

For $d=1$, we have \[\left(1+ \frac{2\beta^{s+1}}{j^s} \right) -  \left(\beta - \frac {2\beta^{s+1}}{j^s} \right) \leq 1 - \beta +2 \beta^2,\] which is strictly less than $1$ for $0 < \beta < 1/2$.  This proves the desired result for $d=1$ and $0 < \beta < 1/2$.  

Finally, for $d=1$ and $\beta = 1/2$, we have \[j\left(\left(\lceil j^s \beta^{-s-1}\rceil +1\right) -  \left(\lfloor j^s \beta^{-s}\rfloor -1\right)\right) \leq j^{s+1} \beta^{-(s+1)}\left(1 + \frac 1{j^s2^{s}}-\frac 1 2\right) < j^{s+1} \beta^{-(s+1)},\] which proves the desired result.  (Note that the arguments of the floor and ceiling functions are both integers here.)  This proves all the cases and thus the proposition.
 
\end{proof}

\section{Properties of $(\alpha, \beta)$-winning sets and $(\alpha, \beta)$-ubiquitously losing sets}\label{sec:PropertiesWinAndUbitLose}

Schmidt gave a lower bound for the Hausdorff dimension of winning sets~\cite[Corollary~1]{Sch66}.  We state Schmidt's result for the supremum norm $\| \cdot \|_\infty$ on $\RR^d$ in order to directly compare it to our upper bound in Theorem~\ref{thm:UpBndHDLosingSets} and also to directly use in proving some of our main results (see Section~\ref{sec:StateMainResults}).  Since all norms on $\RR^d$ are equivalent, the analogous result for another norm on $\RR^d$ can be obtained by applying the result on bilipschitz mappings and $(\alpha, \beta)$-winning sets in \cite[Proposition~5.3]{Dan87} to Theorem~\ref{thm:LowerBndHDWinSet}.

\begin{theo}\label{thm:LowerBndHDWinSet} Let $d \in \NN$, $\RR^d$ be equipped with the supremum norm $\| \cdot \|_\infty$, $0< \alpha <1$, $0 < \beta <1$, and $S \subset \RR^d$ be an $(\alpha, \beta)$-winning set of $\RR^d$.  Then we have that \[\dim_H(S) \geq \frac{d \log\left(\lfloor \beta^{-1}\rfloor\right)}{- \log(\alpha \beta)}.\] 
 
\end{theo}
\begin{proof}
 Let $B$ be a closed ball of $\RR^d$ given by $\| \cdot \|_\infty$.  Then $B$ contains a closed ball $B'$ given by $\| \cdot \|_\infty$ such that \[\rho(B')=\lfloor \beta^{-1}\rfloor \beta \rho(B).\]  Now $\mR_{\lfloor \beta^{-1}\rfloor}(B')$ is the $\lfloor \beta^{-1}\rfloor$-tessellation of $B'$, which, by Lemma~\ref{lemm:TesselBallFund}, contains $\lfloor \beta^{-1}\rfloor^d$ elements of $B^\beta$ with pairwise disjoint interiors.   The desired result now follows by applying~\cite[Corollary~1]{Sch66}.
 \end{proof}

The remaining results in this section have analogs for spaces more general than $\RR^d$.  See Remark~\ref{rmk:UbiqLosSet2}.  The following theorem is an adaption of~\cite[Theorem~2]{Sch66} and its proof.
\begin{theo}\label{thm:CtbleWinPropForWinSets}  Let $d, N \in  \NN$, $0< \alpha<1$, $0< \beta <1$, and $\{S_i\}_{i=1}^N$ be a family of $(\alpha, \beta(\alpha \beta)^{N-1})$-winning sets of $\RR^d$.  Then $\bigcap_{i=1}^N S_i$ is an $(\alpha, \beta)$-winning set of $\RR^d$.
 
\end{theo}

\begin{rema*}
An $(\alpha, \beta)$-winning set of $\RR^d$ is dense in $\RR^d$ for any $0< \alpha<1$, $0< \beta <1$.  Consequently, Theorem~\ref{thm:CtbleWinPropForWinSets} enables us to detect finite intersections that are dense (and thus nonempty).  Furthermore, for $\beta$ small enough, we can not only detect density but also obtain a strictly positive lower bound on the Hausdorff dimension by applying Theorem~\ref{thm:LowerBndHDWinSet} (with Proposition~\ref{prop:BilipschitzInvar2} if we are considering a norm different from the supremum norm $\| \cdot\|_\infty$).
\end{rema*}

\begin{proof}
Let $S:= \bigcap_{i=1}^N S_i$.  We play an $(\alpha, \beta; S)$-game.  Bob chooses $B_1$.  Let $k \in \NN$.  For $k \equiv 1 \pmod {N}$, Alice  applies an $(\alpha, \beta (\alpha\beta)^{N-1}; S_1)$-winning strategy to choose a ball $A_k$.  In particular, in this way, Alice chooses $A_1$.  Bob chooses $B_2$.  Now, for $k \equiv 2 \pmod {N}$, Alice regards $B_2$ as Bob's initial choice of ball and applies an $(\alpha, \beta (\alpha\beta)^{N-1}; S_2)$-winning strategy to choose a ball $A_k$.  Note that Alice will continue to index her balls according to the indexing for the $(\alpha, \beta; S)$-game, not according to the $(\alpha, \beta (\alpha\beta)^{N-1}; S_2)$-game. In particular, in this way, Alice chooses $A_2$.  Bob chooses $B_3$.  Repeating the above, we have that Bob chooses $B_{N}$ and that, for $k \equiv N \pmod {N}$, Alice regards $B_{N}$ as Bob's initial choice of ball and applies an $(\alpha, \beta (\alpha\beta)^{N-1}; S_N)$-winning strategy to choose a ball $A_k$.  As in previous steps, Alice indexes according to the $(\alpha, \beta; S)$-game.  This yields the dyadic play \[B_1 \supset A_1 \supset B_2 \supset A_2 \supset \cdots B_N \supset A_N \supset B_{N+1} \supset A_{N+1} \supset \cdots.\]

Consequently, we have that \[\bigcap_{\ell=0}^\infty A_{i+ \ell N} \in S_i\] for all $i  \in \NN$ such that $1 \leq i \leq N$.  Thus, we have\[\bigcap_{n=1}^\infty B_n \in \bigcap_{i=1}^N S_i,\] yielding the desired result.  

\end{proof}

An analog of~\cite[Theorem~2]{Sch66} holds for $(\alpha, \beta)$-ubiquitously losing sets.
\begin{theo}\label{thm:CtbleWinPropForLoseSets} Let $d \in \NN$, $0< \alpha<1$, $0< \beta <1$, and $\{S_i\}_{i=1}^\infty$ be a countable family of $(\alpha, \beta)$-ubiquitously losing sets of $\RR^d$.  Then $\bigcup_{i=1}^\infty S_i$ is an $(\alpha, \beta)$-ubiquitously losing set of $\RR^d$.
 
\end{theo}

\begin{proof}  The proof is an adaptation of the proof of \cite[Theorem~2]{Sch66}.  The details are as follows.  By Lemma~\ref{lemm:UbitLosImplies}, each $S_i$ is an $(\alpha, \beta^{\{s_m\}})$-ubiquitously losing set for every acceleration sequence $\{s_m\} \subset \NN \cup \{0\}$.  Let $B_1$ be a closed ball of $\RR^d$ that is the initial ball of a play for an $(\alpha, \beta;S)$-game.  For $k \equiv 2^{i-1} \pmod {2^i}$, Bob applies an $(\alpha, \beta (\alpha\beta)^{2^i-1}; S_i)$-winning strategy for Bob with initial ball $B_1$ to choose a ball $B_{k+1}$ (from Alice's choice of ball $A_k$).  Consequently, we have that \[\bigcap_{n=1}^\infty B_n \in S_i^c\] for all $i  \in \NN$.  Thus, we have\[\bigcap_{n=1}^\infty B_n \in \bigcap_{i=1}^\infty S_i^c = \left(\bigcup_{i=1}^\infty S_i \right)^c,\] yielding the desired result.  
\end{proof}

The proof of the following proposition is an adaptation of the proof of~\cite[Lemma~8]{Sch66}.
\begin{prop}\label{prop:UbitLosingInvariance}  Let $d \in  \NN$, $0< \alpha<1$, $0< \beta <1$, $0< \alpha'<1$, and $0< \beta' <1$ such that $\alpha \beta = \alpha' \beta'$ and $\beta' \leq \beta$.  Then every $(\alpha, \beta)$-ubiquitously losing set of $\RR^d$ is an $(\alpha', \beta')$-ubiquitously losing set of $\RR^d$.
 
\end{prop}
\begin{proof}  First, note that we may assume $\beta' < \beta$ without loss of generality.
Let $S \subset \RR^d$ be an $(\alpha, \beta)$-ubiquitously losing set of $\RR^d$.  Bob picks $B'_1$ for an $(\alpha', \beta'; S)$-game on $\RR^d$.  Now Alice chooses \[A'_1\in (B'_1)^{\alpha'}\] for the $(\alpha', \beta'; S)$-game. Since $S \subset \RR^d$ is an $(\alpha, \beta)$-ubiquitously losing set of $\RR^d$, Bob can choose an initial closed ball \[B \in (A'_1)^{\beta}\] according to an $(\alpha, \beta)$-winning strategy for the $(\alpha, \beta; S)$-game.  (In fact, for any choice of initial closed ball that Bob makes, Bob will have an $(\alpha, \beta)$-winning strategy.)  By reindexing, we refer to $B$ as $B_2$.

Now Bob chooses any closed ball \[B'_2 \in B_2^{(\beta'/\beta)}\] and Alice chooses \[A'_2 \in(B'_2)^{\alpha'}\] for the $(\alpha', \beta'; S)$-game.  Note that \[\rho(A'_2) = \alpha' \rho(B'_2) = \frac{\alpha' \beta'}{\beta}\rho(B_2)= \alpha \rho(B_2).\]  Consequently, Bob picks \[B_3 \in (A'_2)^\beta\] according to an $(\alpha, \beta)$-winning strategy for the $(\alpha, \beta; S)$-game with Bob's initial choice of closed ball as $B_2$.  Also Bob picks any closed ball \[B'_3 \in B_3^{(\beta'/\beta)}.\]  Continuing thus by recursion yields that \[\bigcap_{n=1}^\infty B'_n = \bigcap_{n=2}^\infty B_n \in S^c,\] giving the desired result.

\end{proof}

The proposition gives an analog of~\cite[Lemma~11]{Sch66} with essentially the same proof, given below for the convenience of the reader.
\begin{coro}\label{cor:bigeralphalosingsets} Let $d \in  \NN$.  If $0< \alpha < \alpha'<1$ and $0<  \beta_0<1$, then every $(\alpha \mid \beta_0)$-ubiquitously losing set of $\RR^d$ is an $(\alpha' \mid \frac {\alpha \beta_0}{\alpha'})$-ubiquitously losing set of $\RR^d$.
\end{coro}
\begin{proof}  Given any \[0 < \beta' \leq \frac {\alpha \beta_0}{\alpha'},\] we have \[0 < \beta:=\frac{\alpha' \beta'} {\alpha} \leq \beta_0\] and thus $\beta' \leq \beta$.  Now Proposition~\ref{prop:UbitLosingInvariance} yields the desired result.
 
\end{proof}

Winning sets are known to be invariant under bilipschitz mappings (\cite[Proposition~5.3]{Dan87} and see also~\cite[Theorem~1]{Sch66} and~\cite[Introduction]{McM10}), and we have the following analogous result.  Our proof is an adaptation of the proof for winning sets.

\begin{prop}\label{prop:BilipschitzInvar}  Let $d \in  \NN$, $f: \RR^d \rightarrow \RR^d$ be a bilipschitz mapping with bilipschitz constant $K \geq 1$, $0< \alpha <K^{-2}$, $0<  \beta_0<1$, and $S \subset \RR^d$ be an $(\alpha \mid \beta_0)$-ubiquitously losing set of $\RR^d$.  Then $f(S)$ is an $(\alpha K^2 \mid \beta_0K^{-2})$-ubiquitously losing set of $\RR^d$.
 
\end{prop}

\begin{proof}  Let $0 < \beta \leq \beta_0$.  Bob  chooses $B_1$ for an $(\alpha K^2, \beta K^{-2}; f(S))$-game, and Alice chooses \[A_1 \in B_1^{\alpha K^2}\] for the $(\alpha K^2, \beta K^{-2}; f(S))$-game.  It follows from the definition of bilipschitz mapping that $f^{-1}(A_1)$ contains a closed ball $A'_1$ such that  $\rho(A'_1) = K^{-1} \rho(A_1)$.  Now Bob now  chooses a closed ball  \[B' \in (A'_1)^{ \beta }\] according to an $(\alpha, \beta; S)$-winning strategy with Bob's initial choice of closed ball as $B'$.  By reindexing, we refer to $B'$ as $B'_2$.  The definition of bilipschitz mapping allows Bob to choose, for the $(\alpha K^2, \beta K^{-2}; f(S))$-game, a closed ball $B_2$ contained in $f(B'_2)$ such that \begin{align}\label{eqn:BilipInvarUbitLosing}\rho(B_2) =  
 K^{-1} \rho(B'_2)=K^{-2} \beta \rho(A_1) \end{align} and \begin{align}\label{eqn:BilipInvarUbitLosing2}
 B_2 \subset  f(B_2') \subset f(A'_1) \subset f \circ f^{-1} (A_1) = A_1. \end{align} 

Now Alice chooses \[A_2 \in B_2^{\alpha K^2}\] for the $(\alpha K^2, \beta K^{-2}; f(S))$-game.  Thus, $f^{-1}(A_2)$ contains a closed ball $A'_2$ such that \[\rho(A'_2)= K^{-1} \rho(A_2) = \alpha  \rho(B_2')\] by (\ref{eqn:BilipInvarUbitLosing}).  Note that here we have \[A'_2 \subset f^{-1}(A_2) \subset f^{-1}(B_2) \subset f^{-1}\circ f(B'_2) = B'_2 \] by (\ref{eqn:BilipInvarUbitLosing2}).  Now Bob  chooses $B'_3 \in (A'_2)^{ \beta }$ according to an $(\alpha, \beta; S)$-winning strategy with Bob's initial choice of closed ball as $B'_2$.  The definition of bilipschitz mapping allows Bob to choose, for the $(\alpha K^2, \beta K^{-2}; f(S))$-game, a closed ball $B_3$ contained in $f(B'_3)$ such that \begin{align*}\rho(B_3) =  
 K^{-1} \rho(B'_3)=K^{-2} \beta \rho(A_2) \end{align*} and \[B_3 \subset  f(B_3') \subset f(A'_2) \subset f \circ f^{-1} (A_2) = A_2.\] 
 
 Continuing thus by recursion yields that \[\bigcap_{n=2}^\infty B'_n \in S^c\] and, thus, that \[\bigcap_{n=2}^\infty B_n = \bigcap_{n=2}^\infty f(B'_n) = f\left( \bigcap_{n=2}^\infty B'_n \right) \in f(S^c) = (f(S))^c,\] which gives the desired result.

\end{proof}

For use with Theorem~\ref{thm:UpBndHDLosingSets} in computing upper bounds on the Hausdorff dimension for norms different from the supremum norm $\|\cdot\|_\infty$ on $\RR^d$, we have a version of Proposition~\ref{prop:BilipschitzInvar} in which the Lipschitz constants for the mapping and its inverse are explicitly given.  

\begin{prop}\label{prop:BilipschitzInvar2}  Let $d \in  \NN$, $f: \RR^d \rightarrow \RR^d$ be a bijective Lipschitz mapping with Lipschitz constant $K > 0$ whose inverse $f^{-1}$ has Lipschitz constant $L > 0$, $0< \alpha <(KL)^{-1}$, $0<  \beta_0<1$, and $S \subset \RR^d$ be an $(\alpha \mid \beta_0)$-ubiquitously losing set of $\RR^d$.  Then $f(S)$ is an $(\alpha KL \mid \beta_0(KL)^{-1})$-ubiquitously losing set of $\RR^d$.
 
\end{prop}

\begin{proof}
 The proof is analogous to that of Proposition~\ref{prop:BilipschitzInvar}.  Note that $KL \geq 1$ follows from $d_a(f(x), f(y)) \leq K d_b(x, y)$ and $d_b(x,  y) = d_b(f^{-1} \circ f(x), f^{-1} \circ f(y)) \leq L d_a(f(x), f(y))$.  Here, $d_a$ and $d_b$ are the metrics on the range and domain, respectively, of the mapping $f$.
\end{proof}

\end{document}